\documentclass{amsart}
\usepackage{amsmath}
\usepackage[psamsfonts]{amssymb}
\usepackage[all]{xy}
\usepackage{graphicx}
\usepackage[T1]{fontenc}
\usepackage[bitstream-charter]{mathdesign}
\usepackage{hyperref}
\newtheorem{theorem}{Theorem}[section]
\newtheorem{prop}[theorem]{Proposition}
\newtheorem{lemma}[theorem]{Lemma}
\newtheorem{cor}[theorem]{Corollary}

\theoremstyle{remark}
\newtheorem{dfn}[theorem]{Definition}
\newtheorem{remark}[theorem]{Remark}
\newtheorem{claim}[theorem]{Claim}

\def\co{\colon\thinspace}
\def\delbar{\bar{\partial}}

\def\barnu{\bar{\nu}}
\def\ep{\epsilon}

\def\R{\mathbb{R}}
\def\Z{\mathbb{Z}}

\def\mm{\mathop{\min\!\max}\nolimits}

\def\osc{\mathop{osc}\nolimits}

\begin{document}
\title{Hofer's metrics and boundary depth 
}
\author{Michael Usher}
\address{Department of Mathematics\\University of Georgia\\Athens, GA 30602}
\email{usher@math.uga.edu}

\maketitle
\begin{abstract}
We show that if $(M,\omega)$ is a closed symplectic manifold which admits a nontrivial Hamiltonian vector field all of whose contractible closed orbits are constant, then Hofer's metric on the group of Hamiltonian diffeomorphisms of $(M,\omega)$ has infinite diameter, and indeed admits infinite-dimensional quasi-isometrically embedded normed vector spaces.  A similar conclusion applies to Hofer's metric on various spaces of Lagrangian submanifolds, including those Hamiltonian-isotopic to the diagonal in $M\times M$ when $M$ satisfies the above dynamical condition.  To prove this, we use the properties of a Floer-theoretic quantity called the boundary depth, which measures the nontriviality of the boundary operator on the Floer complex in a way that encodes robust symplectic-topological information.
\end{abstract}



\section{Introduction}

Let $(M,\omega)$ be a symplectic manifold and let $H\co [0,1]\times M\to \R$ be a smooth function, which is compactly supported in $[0,1]\times int(M)$ in case $M$ is noncompact or has boundary.
$H$ then induces a time dependent Hamiltonian vector field by the prescription that \[ \omega(\cdot,X_H(t,\cdot))=d_M(H(t,\cdot)),\] and thence an isotopy $\phi_{H}^{t}\co M\to M$ by the prescription that $\phi_{H}^{0}=1_M$ and $\frac{d}{dt}\phi_{H}^{t}(m)=X_H(t,\phi_{H}^{t}(m))$.  

The Hamiltonian diffeomorphism group $Ham(M,\omega)$ is by definition the set of diffeomorphisms $\phi\co M\to M$ which can be written as $\phi=\phi_{H}^{1}$ for some $H$ as above (in particular if $M$ is noncompact or has boundary our convention is that all elements of $Ham(M,\omega)$ are compactly supported in the interior of $M$).   Of course $Ham(M,\omega)$ forms a group, all elements of which are symplectomorphisms of $(M,\omega)$.

For a function $H\co [0,1]\times M\to \R$ as above define \[ \osc H=\int_{0}^{1}\left(\max_{M}H(t,\cdot)-\min_MH(t,\cdot)\right)dt.\]  Now for $\phi\in Ham(M,\omega)$ let \[ \|\phi\|=\inf\left\{\osc H|\phi_{H}^{1}=\phi\right\}.\]  The \emph{Hofer metric on }$Ham(M,\omega)$ is then defined by, for $\phi,\psi\in Ham(M,\omega)$, \[ d(\phi,\psi)=\|\phi^{-1}\circ\psi\|.\]  As was shown for $\R^{2n}$ in \cite{Ho90} and for general symplectic manifolds in \cite{LM95a}, $d$ is a nondegenerate, biinvariant metric on $Ham(M,\omega)$.

Notwithstanding a significant amount of fairly deep work on this metric, our understanding of its global properties remains somewhat limited.  In particular, it is not yet known whether the metric is always unbounded.  It is widely believed that this is most likely the case, and we provide in this paper further evidence for this belief, as follows:

\begin{theorem}\label{hammain}
 Suppose that a closed symplectic manifold $(M,\omega)$ admits a nonconstant autonomous Hamiltonian $H\co M\to \R$ such that all contractible closed orbits of $X_H$ are constant.  Then the diameter of $Ham(M,\omega)$ with respect to Hofer's metric is infinite. 
In fact, there is a homomorphism \[ \Phi\co \R^{\infty}\to Ham(M,\omega) \] such that, for all $v,w\in \R^{\infty}$, \[ \|v-w\|_{\ell_{\infty}}\leq d(\Phi(v),\Phi(w))\leq \osc(v-w).\]
\end{theorem}

Theorem \ref{hammain} is proven in Section \ref{largebeta}.

To clarify notation, $\R^{\infty}$ denotes the direct sum of infinitely many copies of $\R$, \emph{i.e.}, the vector space of sequences $\{v_i\}_{i=1}^{\infty}$ where $v_i\in \R$ and all but finitely many $v_i$ are zero, and for $v=\{v_i\}_{i=1}^{\infty}$
we write $\osc(v)=\max_{i,j}|v_i-v_j|$ and $\|v\|_{\ell_{\infty}}=\max_i|v_i|$.  Thus $\|v\|_{\ell_{\infty}}\leq \osc(v)\leq 2\|v\|_{\ell_{\infty}}$, and if either all $v_i$ are nonnegative or all $v_i$ are nonpositive then $\|v\|_{\ell_{\infty}}=\osc(v)$.
It will be apparent from the construction that  $\Phi(v)$ is generated by a Hamiltonian $G_v$ with $\osc G_v=\osc v$.  From this it follows that, for those $v\in\R^{\infty}$ with $\|v\|_{\ell_{\infty}}=\osc(v)$, every segment of the path $s\mapsto \Phi(sv)$  minimizes the Hofer length among \emph{all} paths connecting its endpoints.  For comparison, there are criteria guaranteeing that a path will be Hofer-length minimizing within its homotopy class in \cite{MSl01}, \cite{S06} (and our paths do satisfy these criteria), but (except in the rare case that $Ham(M,\omega)$ is known to be simply connected) it seems to be unusual to find such globally length-minimizing paths in the Hamiltonian diffeomorphism group of a closed symplectic manifold.

To put Theorem \ref{hammain} into context we indicate some examples of symplectic manifolds $(M,\omega)$ obeying its hypotheses:
\begin{enumerate}
	\item[(A)] Any positive-genus surface $\Sigma$ with area form $\omega$ admits Hamiltonians as in Theorem \ref{hammain}.  Indeed if $\gamma\subset \Sigma$ is a noncontractible closed curve and $U\cong\{(s,\theta)|s\in(-\ep,\ep),\theta\in S^1\}$ is a Darboux--Weinstein neighborhood of $\gamma$ and if $f\co (-\ep,\ep)\to\R$ is a compactly supported smooth function then where $H(s,\theta)=f(s)$ for $(s,\theta)\in U$ and $H(z)=0$ for $z\notin U$, all orbits of $X_H$ either will be constant or will wrap around a noncontractible loop parallel to $\gamma$.  
	
	Generalizing this somewhat, consider fiber bundles $\pi\co M\to \Sigma$ which admit a Thurston-type symplectic form $\Omega=\Omega_0+K\pi^*\omega$ where $\Omega_0$ is closed and fiberwise symplectic and $K\in\R$.  The $\Omega_0$-orthogonal complements to the fibers determine a horizontal subbundle $T^hM$, and in order to ensure that $\Omega$ is symplectic one should take $K$ sufficiently large as to guarantee that at every point it holds that $\Omega|_{T^hM}$ is a positive multiple of the pullback of $\omega$.  As long as this condition on $K$ holds, one can check that if $H\co \Sigma\to\R$ is as in the previous paragraph then $\widetilde{H}=H\circ \pi$ will obey the hypothesis of the theorem, as all orbits $\gamma$ of $X_{\widetilde{H}}$ which are not constant will be contained in $\pi^{-1}(U)$ and will have  $\int_{\gamma}\pi^*d\theta\neq 0$.
Of course this property depends only on the behavior of the symplectic form near  $\pi^{-1}(\gamma)\subset M$ and so the property will continue to hold for suitable symplectic forms if instead the map $\pi\co M\to\Sigma$ has singularities away from $\gamma$ (\emph{e.g.}, if $\pi$ is a Lefschetz fibration).

\item[(B)] Work of Perutz implies that if $\Sigma$ is a positive-genus surface and $d\geq 2$ then the symmetric product $M=Sym^d\Sigma$ obeys the hypothesis of Theorem \ref{hammain}, when $M$ is equipped with any of the continuous family of K\"ahler forms from \cite[Theorem A]{P}.  Indeed, let $\gamma\co S^1\to \Sigma$ be a homologically essential simple closed curve, and let $\Sigma_{\gamma}$ denote the result of surgery along $\gamma$ (\emph{i.e.}, cut $\Sigma$ along $\gamma$ and cap off the resulting boundary components by discs).  Perutz then obtains a Lagrangian correspondence $\hat{V}_{\gamma}\subset Sym^d\Sigma\times Sym^{d-1}\Sigma_{\gamma}$ with the property that the first projection embeds $\hat{V}_{\gamma}$ as a hypersurface $V_{\gamma}\subset Sym^d\Sigma$ while the second projection exhibits $\hat{V}_{\gamma}$ as a $S^1$-bundle over $Sym^{d-1}\Sigma_{\gamma}$.  One can then find a tubular neighborhood $U=(-\ep,\ep)\times V_{\gamma}\subset Sym^d\Sigma$ such that, where $s$ denotes the $(-\ep,\ep)$ coordinate, a Hamiltonian $H$ which is compactly supported in $U$ and such that $H|_U$ depends only on $s$ will have the property that, at all points, $X_H$ either vanishes or is directed along the fibers of the $S^1$-bundle $V_{\gamma}\to Sym^{d-1}\Sigma_{\gamma}$.  Thus any nonconstant closed orbits of $X_H$ are homotopic to iterates of these $S^1$ fibers.  It follows from \cite[Lemma 3.16]{P} that the $S^1$ fibers are homotopic in $Sym^d\Sigma$ to loops of the form $t\mapsto\{\gamma(t),p_1,\ldots,p_{d-1}\}$ for any fixed choice of $p_1,\ldots,p_{d-1}\notin Im(\gamma)$.  So the fact that $\gamma$ is homologically essential in $\Sigma$ implies (by standard facts about the topology of symmetric products, see \emph{e.g.} the proof of \cite[Theorem 9.1]{BT}) that the fibers have infinite order in $\pi_1(Sym^d\Sigma)$. Thus indeed such a Hamiltonian $H\co Sym^d\Sigma\to\R$ obeys the requirements of Theorem \ref{hammain}.

\item[(C)] A variety of symplectic manifolds $(M,\omega)$ which admit a nonconstant autonomous Hamiltonian $H\co M\to\R$ such that $X_H$ has no nonconstant closed orbits at all (contractible or otherwise) are exhibited in \cite{U11}.  Especially in dimension four, these examples are quite topologically diverse: they include for instance the elliptic surfaces $E(n)$ with $n\geq 2$ as well as infinitely many manifolds homeomorphic but not diffeomorphic to them; the symplectic four-manifolds $X_G$ constructed by Gompf \cite{Go} having $\pi_1(X_G)=G$ for any finitely presented group $G$; and simply-connected symplectic four-manifolds whose Euler characteristics and signatures can be arranged to realize many different values.  In general, these examples have a hypersurface $V\subset M$ diffeomorphic to the three-torus such that a suitable Hamiltonian $H$  supported near $V$ will have the property that $X_H$ points along an irrational line on the torus and so has no nonconstant closed orbits.  The construction in \cite{U11} requires $\omega$ to represent an irrational de Rham cohomology class in $H^2(M;\R)$; it is not clear whether one can obtain such Hamiltonians when $[\omega]$ is rational.

\item[(D)] Obviously, if $(M,\omega)$ obeys the hypothesis of Theorem \ref{hammain} then so will $(M\times N,\omega\oplus \sigma)$ for any closed symplectic manifold $(N,\sigma)$ (regardless of whether $(N,\sigma)$ obeys the hypothesis).  Namely, we can just pull back the Hamiltonian $H\co M\to\R$ to $M\times N$.

\item[(E)] If $(M,\omega)$ obeys the hypothesis of Theorem \ref{hammain} and if $(\widetilde{M},\widetilde{\omega})$ is obtained by blowing up a sufficiently small ball $B\subset M$, then  $(\widetilde{M},\widetilde{\omega})$ will also obey the hypothesis.  For if $H\co M\to\R$ is as in Theorem \ref{hammain} and if the ball $B$ is small enough that  $\overline{H(B)}$ is properly contained in $H(M)$, we can choose a nonconstant smooth function $f\co H(M)\to\R$ such that $f|_{\overline{H(B)}}=0$.  Then since $X_{f\circ H}=f'(H)X_H$, the vector field $X_{f\circ H}$ will still have no nonconstant contractible closed orbits.  But $f\circ H$ now lifts to a Hamiltonian on $\widetilde{M}$, whose Hamiltonian vector field again has no nonconstant contractible closed orbits.  

\item[(F)] A well-established criterion (used \emph{e.g.} in \cite{LP}) for $(M,\omega)$ to obey the hypothesis of Theorem \ref{hammain} is for there to exist a Lagrangian submanifold $L\subset M$ such that the inclusion-induced map $\pi_1(L)\to\pi_1(M)$ is injective and such that $L$ admits a Riemannian metric of nonpositive sectional curvature (for in this case the metric on $L$ will have no contractible closed geodesics, and one can take a Hamiltonian supported in a Darboux--Weinstein neighborhood of $L$ which generates a reparametrization of the geodesic flow).  Of course the case of a noncontractible closed curve in a surface as in (A) above is a baby example of this.  In the presence of such a Lagrangian submanifold, a somewhat weaker version of Theorem \ref{hammain} was proven in \cite{Py08}---namely Py proves that for all $N$ one has an embedding $\phi\co \Z^N\to Ham(M,\omega)$ obeying a bound $C_{N}^{-1}|v-w|_{\ell_{\infty}}\leq d(\phi(v),\phi(w))\leq C_N|v-w|_{\ell_{\infty}}$.  (Actually, our embedding in Theorem \ref{hammain} appears to reduce to Py's in this special case, and so Theorem \ref{hammain} improves Py's constants.)

It should be clear from the examples that we have provided that the hypothesis of Theorem \ref{hammain} is substantially more general than the assumption that $M$ contains a $\pi_1$-injective Lagrangian submanifold which admits a metric with nonpositive sectional curvature. 
Writing $2n=\dim M$, so that $\dim L=n$, in order for $L$ to admit such a metric $L$ would have to be either flat and hence (by old results of Bieberbach) a finite quotient of $T^n$, or else by \cite[Theorem A]{BE} $\pi_1(L)$ would contain a nonabelian free group. Thus $\pi_1(M)$ would have to contain either $\Z^n$ or the free group on two generators.
But in many of the above examples $\pi_1(M)$ is not large enough to accommodate such subgroups---indeed in some of the examples $M$ is even simply connected.
\end{enumerate}

There are however closed symplectic manifolds to which Theorem \ref{hammain} can be proven not to apply, namely those which have finite $\pi_1$-sensitive Hofer--Zehnder capacity.  It is shown in \cite[Corollary 1.19]{Lu06} (using an argument that essentially dates back to \cite{HV}) that any closed symplectic manifold which admits a nonvanishing genus-zero Gromov--Witten invariant counting pseudoholomorphic spheres that pass through two generic points has finite $\pi_1$-sensitive Hofer--Zehnder capacity; if the manifold is simply connected one can instead use arbitrary-genus Gromov--Witten invariants counting curves through two generic points.  For instance this applies to all closed toric manifolds (to see this one can use Iritani's theorem \cite{Ir} that a toric manifold has generically semisimple big quantum homology,  so  that in particular the class of a point is not nilpotent in quantum homology), and also to any simply-connected closed symplectic four-manifold with $b^+=1$ (this follows from work of Taubes and Li--Liu; see \cite[Appendix A]{U11} for the argument).

There is a substantial history of results showing Hofer's metric on $Ham(M,\omega)$ to have infinite diameter for a variety of symplectic manifolds $(M,\omega)$; Theorem \ref{hammain} overlaps somewhat with these prior results but also includes many new cases (and conversely, there are a some examples which are covered by previous results but are not covered by Theorem \ref{hammain}, including $\mathbb{C}P^n$).  Notable early results in this direction include those in \cite[Section II.5.3]{LM95}, \cite{Po98}, \cite[Section 5.1]{Sc00}, and \cite[Remark 1.10]{EP03}.   More recent work of McDuff \cite[Lemma 2.7]{M09} shows that the Hofer metric has infinite diameter provided that the asymptotic spectral invariants, which \emph{a priori} are defined on the universal cover $\widetilde{Ham}(M,\omega)$, descend to $Ham(M,\omega)$.  \cite[Theorems 1.1 and 1.3]{M09} provide a range of sufficient conditions for the asymptotic spectral invariants to descend, which are general enough to encompass nearly all of the cases in which infinite Hofer diameter has been proven for closed $(M,\omega)$ until now.\footnote{The only exceptions to this that I am aware of are products of positive genus surfaces with other manifolds (for which the result follows from the stabilized non-squeezing theorem of \cite{LP99}, as mentioned on \cite[II, p. 64]{LM95}---of course this case is also covered by Theorem \ref{hammain}) and the case of a small blowup of $\mathbb{C}P^2$ which is covered in \cite{M10}.}   The argument in \cite{M09} combines a construction of Ostrover \cite{Os03} of a path $\{\phi_t\}_{t\in\R}$ in $Ham(M,\omega)$ for \emph{any} closed $(M,\omega)$ for which the asymptotic spectral invariants (and hence the lifted Hofer pseudo-norm on $\widetilde{Ham}(M,\omega)$) diverge to $\infty$, with a detailed analysis of the properties of the Seidel representation \cite{Se} of $\pi_1(Ham(M,\omega))$ which finds that the asymptotic spectral invariants descend and hence that Ostrover's path has $\|\phi_t\|\to \infty$ under  the conditions given in  \cite[Theorems 1.1 and 1.3]{M09}.  Roughly speaking, the hypotheses    of \cite[Theorems 1.1 and 1.3]{M09} ask for $(M,\omega)$ to either have large minimal Chern number (at least $n+1$, or $n$ under additional hypotheses, if $\dim M=2n$) or else to admit few nonvanishing genus zero Gromov--Witten invariants (for instance $(M,\omega)$ could be weakly exact or, under mild topological hypotheses, negatively monotone).  As is shown in \cite{M09}, once these conditions are violated the asymptotic spectral invariants can very well fail to descend---for instance by \cite[Proposition 1.8]{M09} they never descend when $(M,\omega)$ is a point blowup of a non-symplectically-aspherical manifold; in this case the minimal Chern number of $M$ can be as large as $n-1$.

There are many manifolds obeying Theorem \ref{hammain} which are not covered by the results of \cite{M09} or by any other results on infinite Hofer diameter that I am aware of.  For instance McDuff's criteria are not robust under taking products or point blowups, whereas we have noted above that (at least for sufficiently small blowups) the criterion in Theorem \ref{hammain} is preserved under these operations.   Thus for instance while the non-symplectically-aspherical \emph{minimal} examples from (C) above obey both Theorem \ref{hammain} and McDuff's criteria, when these examples are blown up or when they are replaced by their products with (say) $S^2$ they obey only Theorem \ref{hammain}.  Prior results also do not seem to suffice to prove infinite Hofer diameter for a variety of nontrivial bundles over positive genus surfaces (for instance nontrivial irrational ruled surfaces) as in (A) above. Also from the calculations of Gromov--Witten invariants in \cite{BT} one can see that $Sym^d\Sigma_g$ does not satisfy the hypotheses of \cite[Theorems 1.1 and 1.3]{M09} when $d\geq g\geq 1$.  

Of course, another advantage of Theorem \ref{hammain} is that it yields not just infinite diameter but also a quasi-isometrically embedded infinite-dimensional normed vector space in $Ham(M,\omega)$.  In the case that $(M,\omega)$ obeys \emph{both} the assumptions of Theorem \ref{hammain} and the property that the asymptotic spectral invariants descend to $Ham(M,\omega)$ as in \cite{M09}, one can use \cite[Proposition 4.1]{U10a} to prove Theorem \ref{hammain}---in fact in this case the embedding $\Phi\co \R^{\infty}\to Ham(M,\omega)$ can actually be seen to obey precisely $d(\Phi(v),\Phi(w))=\osc(v-w)$ rather than just being quasi-isometric (verification of this is left to the reader).   While it seems likely that $Ham(M,\omega)$ always has infinite Hofer diameter, there is less consensus as to whether $Ham(M,\omega)$ should always admit embeddings of infinite-dimensional normed vector spaces like those in Theorem \ref{hammain}.  For instance L. Polterovich has pointed out that nothing currently known about $Ham(S^2)$ is incompatible with it being quasi-isometric to $\R$.

Our proof of Theorem \ref{hammain}, like McDuff's proof of \cite[Lemma 2.7]{M09}, uses a quantity arising from filtered Hamiltonian Floer theory as a lower bound for the Hofer norm.  Whereas McDuff uses the asymptotic spectral invariants for this purpose, we use a different quantity called the \textbf{boundary depth}, which was formally introduced in \cite{U09}, though one can find hints of it earlier---in particular an argument in \cite{Oh09} was influential in leading me to it. Unlike the asymptotic spectral invariants, the boundary depth is, as we will show, \emph{a priori} well-defined on $Ham(M,\omega)$ rather than just on $\widetilde{Ham}(M,\omega)$; consequently there is no need for a subtle analysis of the Seidel morphism as in \cite{M09}.  On the other hand, while Ostrover's construction in \cite{Os03} produces a sequence of Hamiltonians with diverging asymptotic spectral invariants on any closed symplectic manifold, it is not clear whether there always exists a sequence in $Ham(M,\omega)$ with diverging boundary depths---indeed it seems plausible that no such sequence exists for $M=\mathbb{C}P^n$.  In particular Proposition \ref{dispbeta} shows that the boundary depths of the Hamiltonians in Ostrover's sequence remain bounded.  However for manifolds obeying Theorem \ref{hammain} many sequences with diverging boundary depths do exist.

Another advantage of the boundary depth is that it quite naturally and generally adapts to Lagrangian Floer theory and yields results concerning Hofer's metric on Lagrangian submanifolds, as we now discuss.

\subsection{The Lagrangian Hofer metric}

Now suppose that $(M,\omega)$ is tame (\emph{i.e.}, there is an $\omega$-compatible almost complex structure on $M$ whose induced Riemannian metric is complete with injectivity radius bounded below and with bounded sectional curvature).  Fix a closed Lagrangian submanifold $L\subset M$ and let \[ \mathcal{L}(L)=\{\phi(L)|\phi\in Ham(M,\omega)\};\] thus $\mathcal{L}(L)$ is the orbit of $L$ under the natural action of the Hamiltonian diffeomorphism group on the set of Lagrangian submanifolds.\footnote{To be clear, elements of $\mathcal{L}(L)$ are viewed as \emph{unparametrized} submanifolds; equivalently we can think of $\mathcal{L}(L)$ as the set of Lagrangian embeddings of $L$ modulo precomposition by diffeomorphisms of $L$.}  Now for $L_0,L_1\in\mathcal{L}(L)$ define \[ \delta(L_0,L_1)=\inf\{\|\phi\| |\phi(L_0)=L_1\}.\]  Chekanov showed in \cite{Ch00} that $\delta$ defines a nondegenerate metric on $\mathcal{L}(L)$; obviously Hamiltonian diffeomorphisms act by isometries with respect to this metric, which we will refer to as \emph{the Hofer metric on }$\mathcal{L}(L)$.

Relatively little is known about the global properties of the Hofer metric on $\mathcal{L}(L)$, especially when $M$ is closed.  In the model noncompact case in which $M=T^*L$ with its standard symplectic structure and where $L$ is the zero section, results of Oh and Milinkovi\'c imply that, where $C^{\infty}_{0}(L)$ denotes the space of smooth functions on $L$ modulo addition of constants, the embedding $f\mapsto graph(df)$ is isometric with respect to the norm $\osc$ on $C^{\infty}_{0}(L)$ and the Hofer norm on $\mathcal{L}(L)$ (this does not seem to be explicitly stated in Oh and Milinkovi\'c's work, but can be extracted from \cite[Theorem 3]{Mil}).  More recently Khanevsky \cite{Kh09} proved that $\mathcal{L}(L)$ has infinite diameter in case $M=S^1\times (-1,1)$ and $L=S^1\times\{0\}$, or when $M=D^2$ and $L=\{(x,0)|-1\leq x\leq 1\}$.  It is also mentioned in \cite{Kh09} that arguments from \cite{LM95} can be used to show that $\mathcal{L}(L)$ has infinite diameter when $L$ is a homologically essential curve on a positive genus surface.  Another approach to this statement in the case that $L$ is a meridian on a torus appears in \cite[Remark 5.3]{Le08}, where spectral invariants in Lagrangian Floer theory are used.  Leclercq's approach could also be used in some other weakly exact cases (for instance for standard Lagrangian tori in $T^{2n}$); however extensions beyond the weakly exact case seem more difficult due to the lack of a more general theory of Lagrangian spectral invariants.

In contrast to the Hamiltonian case, it should not be expected that Hofer's metric on $\mathcal{L}(L)$ \emph{always} has infinite diameter; indeed we prove by an elementary argument in Section \ref{circlesect} that when $L$ is the unit circle in $\R^2$ the diameter of $\mathcal{L}(L)$ is no larger than $2\pi$.

By using the Lagrangian Floer-theoretic version of the boundary depth, we extend the class of $L$ for which $\mathcal{L}(L)$ has infinite diameter in two directions.

For the first of our results in this regard, note that if $(M,\omega)$ is a symplectic manifold, then if we endow $M\times M$ with the symplectic structure $(-\omega)\oplus \omega$ and denote by $\Delta$ the diagonal, we have an embedding \begin{align*} Ham(M,\omega)&\hookrightarrow \mathcal{L}(\Delta) \\ \phi &\mapsto \Gamma_{\phi}=\{(x,\phi(x))|x\in M\}\end{align*}  This embedding preserves lengths of paths, and hence we have a relation \[ \delta(\Gamma_{\phi},\Delta)\leq \|\phi\| \] (of course equality can in principle fail to hold, since there might be a shorter path from $\Delta$ to $\Gamma_{\phi}$ which leaves the image of the embedding).  We show:

\begin{theorem}\label{diagmain} Let $(M,\omega)$ be a closed symplectic manifold such that there is a nonconstant autonomous Hamiltonian $H\co M\to\R$ such that all contractible closed orbits of $X_H$ are constant.  Then the embedding $\Phi\co \R^{\infty}\to Ham(M,\omega)$ from Theorem \ref{hammain} has the property that, for all $v,w\in \R^{\infty}$, \[ \|v-w\|_{\ell_{\infty}}\leq \delta(\Gamma_{\Phi(v)},\Gamma_{\Phi(w)})\leq \osc(v-w) .\]
\end{theorem}

In other words, our lower bound on the Hofer distance persists when we pass from the Hamiltonian to the Lagrangian context by replacing Hamiltonian diffeomorphisms by their graphs.  Thus for any $(M,\omega)$ as in Theorem \ref{hammain}, the space $\mathcal{L}(\Delta)$ of Lagrangian submanifolds of $M\times M$ Hamiltonian-isotopic to the diagonal has infinite diameter, and indeed contains an infinite-dimensional quasi-isometrically embedded normed vector space.  Theorem \ref{diagmain} is proven just after the proof of Theorem \ref{mainproplag} in Section \ref{lagsect}.

This behavior should be contrasted with that seen in \cite{Os03}.  As mentioned earlier, Ostrover constructs therein a path $\{\phi_t\}$ in $Ham(M,\omega)$ for any closed $(M,\omega)$ which, at least in $\widetilde{Ham}(M,\omega)$, goes arbitrarily far away from the identity.  Under topological conditions on $(M,\omega)$ such as those from \cite{M09}, one will indeed have $\|\phi_t\|\to\infty$ where $\|\cdot\|$ denotes the Hofer norm on $Ham(M,\omega)$.  However, Ostrover shows in \cite{Os03} that the Lagrangian submanifolds $\Gamma_{\phi_t}$ remain within a finite distance from $\Delta$.  Thus the Hamiltonian diffeomorphisms in Theorem \ref{diagmain} exhibit rather different behavior than those in Ostrover's path.

To state our other main result on the Lagrangian Hofer metric, we prepare some notation.  We denote $S^1=\R/\mathbb{Z}$, and, for $m\in \Z_+$, denote \[ C_{m}^{\infty}(S^1)=\{f\co S^1\to\R|(\forall x\in S^1)(f(x+1/m)=f(x))\} \] and \[ C_{m,0}^{\infty}(S^1)=\frac{C_{m}^{\infty}(S^1)}{\R}\] where $\R$ acts by addition of constants.  Thus $C_{m,0}^{\infty}(S^1)$
carries the norm $\osc(f)=\max f-\min f$.   Let $T^2=\R^2/\Z^2$ and for $f\in C_{m,0}^{\infty}(S^1)$ denote  $L_f=\{(x,f'(x))|x\in S^1\}$ (where of course the both coordinates are evaluated mod $\Z$).  We then have:

\begin{theorem}\label{s1prodmain} Let $L\subset M$ be a monotone Lagrangian submanifold of a tame symplectic manifold $M$ with minimal Maslov number at least $2$ whose Floer homology $HF(L,L)$ is nonzero.  Consider the space $\mathcal{L}(L_0\times L)$ of Lagrangian submanifolds of $T^2\times M$ Hamiltonian-isotopic to $L_0\times L$.  Then there is a constant $C\geq 0$ such that for any integer $m\geq 2$ and any $f,g\in C_{m,0}^{\infty}(S^1)$, we have \[ \osc(f-g)-C\leq \delta(L_f\times L,L_g\times L)\leq \osc(f-g).\]  In the case that $HF(L,L)$ is isomorphic to the singular homology of $L$, the constant $C$ may be set to zero, so that $f\mapsto L_f\times L$ is an isometric embedding of $C_{m,0}^{\infty}(S^1)$ into $\mathcal{L}(L_0\times L)$.
\end{theorem}   

Theorem \ref{s1prodmain} is proven at the end of Section \ref{lagsect} (with key input provided by Theorem \ref{prodbeta}). Various small modifications to this result can also be established, as will be apparent in the proof.  First, the torus $T^2$ can be replaced by the infinite cylinder $T^*S^1$, yielding the same conclusion.  Moreover in this latter statement one could replace $S^1$ by a more general closed manifold $L_0$, so that one considers Lagrangian submanifolds in $(T^*L_0)\times M$, and one would obtain at least that $\mathcal{L}(L_0\times L)$ has infinite diameter.  Also the monotonicity assumption on $L$ appears to be only technical; assuming that $HF(L,L)\neq 0$ all that  is really needed is a K\"unneth-type formula relating the Floer complex of $L\subset M$ to that of $S^1\times L\subset T^2\times M$.  This K\"unneth formula is well-known in the monotone context, but likely is true in the more general setup of \cite{FOOO09}; there is work in progress by L. Amorim aimed at showing this.  

In the case where $L=M$ is a point (so that we are just considering Lagrangians in $T^2$ Hamiltonian-isotopic to the meridian) Theorem \ref{s1prodmain} can be inferred from Leclercq's arguments in \cite{Le08} using spectral invariants; indeed in this case there is no need to assume $m\geq 2$.  However our use of the boundary depth requires one to take $m\geq 2$ in order to get nontrivial lower bounds.

\subsection{Boundary depth}

As mentioned earlier, the proofs of our main results are based on the properties of a Floer-theoretic quantity called the boundary depth, which was introduced in the Hamiltonian context in \cite{U09}.  We indicate in this subsection some of the basic features of this quantity.  Either the Hamiltonian Floer complex associated to a Hamiltonian diffeomorphism, or the Lagrangian Floer complex associated to two Hamiltonian-isotopic Lagrangian submanifolds, can be seen formally as the Morse--Novikov complex of an action functional on a cover of a suitable path space.\footnote{Of course, the same is true of the Lagrangian Floer complex of a pair of non-Hamiltonian-isotopic Lagrangians, but since we have not (yet) found interesting applications of the boundary depth in this more general context this paper will restrict to Floer theory for Hamiltonian-isotopic Lagrangians in order to simplify the discussion.}  As such, the complex carries a natural filtration by $\R$, obtained by considering sublevel sets of the action functional.  Given a chain complex $(C,\partial)$ with a filtration by $\R$, its boundary depth $b(C,\partial)$ is the infimal (actually, in the cases considered in this paper, minimal by Proposition \ref{depthattained}) number $\beta$ with the following property:  whenever $x$ lies in the image of  $\partial$, there must be a chain $y$ with $\partial y=x$ and with filtration level at most $\beta$ larger than that of $x$  (see Section \ref{algintro} for a more formal definition).  Thus $b(C,\partial)$ can be seen as a quantitative measurement, in terms of the filtration, of the nontriviality of the differential $\partial$.  In particular if $\partial=0$ then $b(C,\partial)=0$. Unlike, for instance, spectral invariants, $b$ has relatively little to do with the \emph{homology} of the complex; indeed in some cases in Lagrangian Floer theory (and also in the sectors of Hamiltonian Floer theory corresponding to noncontractible loops) the homology vanishes but the boundary depth still provides nontrivial information.

Now the Floer complexes associated to Hamiltonian diffeomorphisms $\phi$ or to pairs of Hamiltonian-isotopic Lagrangian submanifolds $(L,L'=\phi^{-1}(L))$ depend on some additional data, notably including a specific Hamiltonian function $H\co [0,1]\times M\to\R$ inducing $\phi$ as its time-one map.  We will see however that the boundary depth is unaffected by changes in the choices of additional data, and so gives an invariant of the diffeomorphism $\phi$ or of the pair of Lagrangians $(L,L')$.\footnote{In the Lagrangian case, at least if $L$ is not monotone, one must also choose at the outset a relative spin structure and bounding cochain for $L$; the boundary depth (like the homology) may depend on this choice.}  This occurs because different choices result in chain complexes which are what we call in Section \ref{algintro} ``shift-isomorphic''---roughly speaking, up to isomorphism of filtered chain complexes, the complexes associated to different choices differ only by uniform shifts in their filtrations (because in this paper we incorporate homotopically nontrivial loops and paths into the definition of Floer theory the appropriate definition is slightly more complicated than just allowing for a single uniform shift; see Definition \ref{shiftiso}).  Since the boundary depth is obtained by considering \emph{differences} of filtration levels, it is unaffected by such uniform shifts.

This allows one to canonically define the boundary depth for a Hamiltonian diffeomorphism $\phi$ or a pair of Lagrangian submanifolds $(L,L')$ which satisfies the standard nondegeneracy hypotheses required in Floer theory.  Moreover, the notion so obtained is continuous with respect to the Hofer norm, and so by continuity one can then extend the definition to degenerate cases.

We will consistently work over a field $K$ in this paper (standard choices for $K$ in various situations are $\mathbb{Z}/2\Z$ or $\mathbb{Q}$---of course, there is also typically a Novikov ring (which we denote by $\Lambda^{K,\Gamma}$) involved in the definition of the Floer complex, but what we call $K$ refers not to the Novikov ring but to the field in which the coefficients of elements of the Novikov ring take values). The boundary depth can of course be formulated for complexes over rings which are not fields, but for the proofs of some of our algebraic results about the behavior of the boundary depth (\emph{e.g.}, Proposition \ref{depthattained} and Theorem \ref{prodbeta}) it is convenient to take $K$ to be a field.

In any case, for a closed manifold $(M,\omega)$ we obtain a boundary depth function \[ \beta(\cdot;K)\co Ham(M,\omega)\to \R, \] and for a closed Lagrangian submanifold $L$ of a tame symplectic manifold $(M,\omega)$, equipped if necessary with a relative spin structure $\frak{s}$ and a bounding cochain $b$ as in \cite{FOOO09} (we denote by $\hat{L}$ the tuple $(L,\frak{s},b)$), we obtain a boundary depth function \[ \beta_{\hat{L}}(\cdot;K)\co \mathcal{L}(L)\to \R.\]  Here $K$ denotes any field over which the appropriate Floer complex can be defined.  We give complete definitions of these functions in Sections \ref{hamsect} and \ref{lagsect}, but presently we state some of their properties. 

In the Hamiltonian case, we can form the Floer complex over $K$ where $K$ is equal to any field of characteristic zero on arbitrary closed symplectic manifolds (\cite{FO},\cite{LT}), or to any field whatsoever if $(M,\omega)$ is semipositive \cite{HS}.  
\begin{theorem}\label{mainpropham} Let $(M,\omega)$ be a closed $2n$-dimensional symplectic manifold, and let $K$ be a field, with characteristic zero if $(M,\omega)$ is not semipositive.  The boundary depth function $\beta(\cdot;K)\co Ham(M,\omega)\to [0,\infty)$ obeys the following properties: 
\begin{itemize}\item[(i)] If $\psi\in Symp(M,\omega)$ and $\phi\in Ham(M,\omega)$ then $\beta(\psi^{-1}\phi\psi;K)=\beta(\phi;K)$.
\item[(ii)] For any $\phi$ in $Ham(M,\omega)$, \[ \beta(\phi;K)=\beta(\phi^{-1};K).\]
\item[(iii)] $\beta$ is $1$--Lipschitz with respect to the Hofer norm: for any $\phi,\psi\in Ham(M,\omega)$ we have \[ |\beta(\phi;K)-\beta(\psi;K)|\leq \|\phi^{-1}\psi\|.\]
\item[(iv)] $\beta(1_M;K)=0$, where $1_M$ is the identity.
\item[(v)] If $(N,\theta)$ is another closed symplectic manifold and $\phi\in Ham(M,\omega),\psi\in Ham(N,\theta)$, then the diffeomorphism $\phi\times\psi\co M\times N\to M\times N$ obeys \[ \beta(\phi\times \psi;K)\geq \max\{\beta(\phi;K),\beta(\psi;K)\}.\]
\end{itemize}\end{theorem}

Modulo an algebraic result (Theorem \ref{prodbeta}) which is needed in the proof of part (v), Theorem \ref{mainpropham} is proven in Section \ref{hambetaprop}.

Of course (iii) and (iv) combine to yield the following important corollary, which drives most of our applications:

\begin{cor}\label{hoferhambound} For any $\phi\in Ham(M,\omega)$  we have \[ \beta(\phi;K)\leq \|\phi\|.\]
\end{cor}

Before describing the next important property of $\beta$ we introduce some notation.
If $(M,\omega)$ is a closed symplectic manifold we will denote by $\Gamma_{\omega}$ the subgroup of $\mathbb{R}$ given by\[ \Gamma_{\omega}= \left\{\langle[\omega],A\rangle|A\in H_{2}^{T}(M;\Z)\right\}\] where $H_{2}^{T}(M;\Z)$ is the subgroup of  $H_2(M;\Z)$ generated by classes of form $u_*[S^1\times S^1]$ where $u\co S^1\times S^1\to M$ is continuous.  (As a technical point, our use of toroidal classes $H_{2}^{T}$ rather than just spherical classes has to do with the fact that we consider the sectors of Floer theory given by noncontractible orbits in addition to the contractible ones.  Of course, any spherical class is also toroidal.)

Let $\mathcal{L}M$ denote the free loopspace of $M$.  For each path component $\mathfrak{c}$ choose an element $\gamma_{\mathfrak{c}}$ representing $\mathfrak{c}$.  If $\gamma\co S^1\to M$ and $u\co [0,1]\times S^1\to M$ obeys $u(0,\cdot)=\gamma_{\mathfrak{c}}$ and $u(1,\cdot)=\gamma$, and if $H\co S^1\times M\to\R$ is smooth let \[ \mathcal{A}_H(\gamma,u)=-\int_{[0,1]\times S^1}u^*\omega+\int_{0}^{1}H(t,\gamma(t))dt.\]  The $\mathfrak{c}$-action spectrum of $H$ is then by definition \[ \mathcal{S}_{H}^{\mathfrak{c}}=\{\mathcal{A}_H(\gamma,u)|\dot{\gamma}(t)=X_H(t,\gamma(t))\}.\]  In general we have $\mathcal{A}_H(\gamma,u)-\mathcal{A}_H(\gamma,u')\in \Gamma_{\omega}$; thus $\mathcal{S}_{H}^{\mathfrak{c}}$ is a union of cosets of $\Gamma_{\omega}$, one for each $1$-periodic orbit of $H$.  Note that $\mathcal{S}_{H}^{\mathfrak{c}}$ depends on the basepoint $\gamma_{\mathfrak{c}}$ that was chosen for $\mathfrak{c}$; however one easily sees that the \emph{difference set} $\{s-t|s,t\in \mathcal{S}_{H}^{\mathfrak{c}}\}$ is independent of that choice.

\begin{theorem}\label{spectrality}
Assume that $\phi$ is nondegenerate and is generated by the Hamiltonian $H\co S^1\times M\to\mathbb{R}$.  Then where  $\mathcal{S}_{H}^{\mathfrak{c}}$ is the $\mathfrak{c}$-action spectrum of $H$ for $\mathfrak{c}\in \pi_0(\mathcal{L}M)$, we have \[ \beta(\phi;K)\in \{0\}\cup\bigcup_{\mathfrak{c}\in \pi_0(\mathcal{L}M)} \{s-t|s,t\in \mathcal{S}_{H}^{\mathfrak{c}}\}\]   Moreover, again assuming that $\phi$ is nondegenerate, $\beta(\phi;K)=0$ if and only if the number of fixed points of $\phi$ is equal to $\sum_{k=0}^{2n}rank H_k(M;K)$.
\end{theorem}

\begin{proof}
See the end of Section \ref{hambetaprop}.
\end{proof}

We now turn to the Lagrangian case.  Lagrangian Floer theory was formulated for closed monotone Lagrangian submanifolds $L$ of tame symplectic manifolds $(M,\omega)$ in \cite{Oh93} over $\mathbb{Z}/2\mathbb{Z}$ coefficients assuming that the minimal Maslov number of $L$ is at least $2$; and for relatively spin Lagrangians which satisfy an unobstructedness condition in \cite{FOOO09} over $\mathbb{Q}$.  In the formulation in \cite{FOOO09} one must additionally choose a relative spin structure on $L$ and a ``bounding cochain'' in order to obtain a chain complex (and the quasi-isomorphism type will depend on these choices); our convention throughout this paper will be to denote by $\hat{L}$ a choice of a Lagrangian submanifold together with whatever such additional structure is needed in the case at hand.

Where $K$ is a field as above, the following theorem describes some salient properties of the boundary depth function $\beta_{\hat{L}}(\cdot;K)\co \mathcal{L}(L)\to \R$ on Lagrangian submanifolds Hamiltonian-isotopic to $L$.  As in \cite{Ch98}, for an almost complex structure $J$ compatible with $\omega$ we let $\sigma(M,L,J)$ equal the smaller of either the minimal area of a nonconstant $J$-holomorphic disc with boundary on $L$, or the minimal energy of a nonconstant $J$-holomorphic sphere intersecting $L$.

\begin{theorem} \label{mainproplag}  For any $L_1,L_2\in \mathcal{L}(L)$ we have:
\begin{itemize} 
\item[(i)]  $|\beta_{\hat{L}}(L_1;K)-\beta_{\hat{L}}(L_2;K)|\leq \delta(L_1,L_2)$.
\item[(ii)] If the Floer homology $HF(L,L)$ is isomorphic to the singular homology $H_*(L)$ (with the appropriate Novikov ring coefficients) then $\beta_{\hat{L}}(L;K)=0$.  However, if $HF(L,L)$ is not isomorphic to $H_*(L)$, then $\beta_{\hat{L}}(L;K)\geq \sigma(M,L,J)$ for any $\omega$-compatible almost complex structure $J$.
\item[(iii)] If $L\cap L_1=\varnothing$ then $\beta_{\hat{L}}(L_1)=0$.
\item[(iv)] Let $(M,\omega)$ be a closed symplectic manifold, $\phi\in Ham(M,\omega)$, and let $\Gamma_{\phi}\subset M\times M$ be the graph of $\phi$.  Then where $M\times M$ is endowed with the symplectic structure $(-\omega)\oplus \omega$ and where $\Delta$ is the diagonal, for a suitable relative spin structure and bounding cochain on $\Delta$ we will have \[ \beta_{\hat{\Delta}}(\Gamma_{\phi};K)=\beta(\phi;K).\]
\end{itemize}
\end{theorem}

More specifically, in \cite[p. 32]{FOOO09b} the authors construct a relative spin structure on $\Delta$ such that $0$ is a bounding cochain, and in (iv) we may use this relative spin structure and the zero bounding cochain.

Points (i)--(iii) above evidently combine to recover Chekanov's famous result \cite{Ch98} (at least for Lagrangians with well-defined Floer homology) that any Lagrangian submanifold has displacement energy equal to at least $\sup_J \sigma(M,L,J)$.  
When one unravels the arguments underlying the proofs, though, it becomes clear that this is not really a new proof of Chekanov's theorem, as similar ideas (though organized differently, of course) have been used in proofs such as the one in \cite[Section 4.3]{CL05}.  This approach to estimating the displacement energy of a Lagrangian submanifold seems to be very closely related to the approach using ``torsion thresholds'' in Floer homology in \cite{FOOO11}.

We next consider boundary depths of products of Lagrangians.  Suppose that we have  Lagrangian submanifolds $L\subset M$, $L'\subset M'$ with well-defined Floer homologies (at least after enriching them with appropriate additional data to give $\hat{L}$, $\hat{L'}$).  Let us say that $\hat{L}$ and $\hat{L'}$ \emph{satisfy the K\"unneth property} if, for generic Hamiltonians $H\co [0,1]\times M\to\R$, $H'\co [0,1]\times M\to\R$, auxiliary data can be chosen in such a way that the Floer complex $CF(\widehat{L\times L'}:H+H';K)$ is isomorphic as a filtered chain complex to the tensor product of $CF(\hat{L}:H;K)$ and $CF(\hat{L'}:H';K)$ once coefficients are extended so that all three chain complexes are defined over the same Novikov ring.  When $L,L',L\times L'$ are all monotone the K\"unneth property is well-known; see for instance \cite[Section 4]{Li04}.  In the much more general setting of \cite{FOOO09}, the K\"unneth property has yet to appear in the literature, but is the subject of work in progress by L. Amorim.

\begin{theorem}\label{lagprod}  Assume that $\hat{L}$ and $\hat{L'}$ satisfy the K\"unneth property, and let $N\in \mathcal{L}(L)$, $N'\in\mathcal{L}(L')$.  Assume moreover that the Floer homology $HF(\hat{L'},\hat{L'})$ is nonzero.  Then \[ \beta_{\widehat{L\times L'}}(N\times N')\geq \beta_{\hat{L}}(N).\]
\end{theorem}

\subsection{Outline of the paper}

The focus of this paper alternates between geometry and algebra.  In the upcoming Section \ref{prologue} we introduce the notion of the boundary depth in a simple yet still interesting case, namely that of Morse functions on $S^1$.  This discussion will later become relevant in the proof of Theorem \ref{s1prodmain}; it has been placed near the start of the paper in the hope that it will also help the reader develop an intuition for the boundary depth.  

Section \ref{algintro} introduces the boundary depth from an algebraic standpoint, proving for instance continuity properties of the boundary depth under algebraic assumptions which model the behavior of continuation maps in both Hamiltonian and Lagrangian Floer theory.  Section \ref{qcor} is devoted to an algebraic result about the boundary depths of filtered chain complexes that are obtained as ``quantum corrections'' of unfiltered complexes in a way reminiscent of the ``Morse--Bott'' approach to calculating Floer homology.  This result is the key ingredient in Theorem \ref{mainproplag} (ii).  

Our main results about $Ham(M,\omega)$ are proven in Section \ref{hamsect}, while those about Lagrangian submanifolds are proven in Section \ref{lagsect}, modulo some algebraic details which are deferred to the following two sections, as well as a technical point relating to transversality that is needed for the proof of Theorem \ref{hammain} and is deferred to the Appendix.  Another result of Section \ref{hamsect}, Corollary \ref{encap}, asserts an \emph{a priori} upper bound on the boundary depths of Hamiltonian diffeomorphisms generated by Hamiltonians supported in a given displaceable set (a slightly weaker result appears in \cite{U09}).  Thus, while one might in principle hope to show that an arbitrary closed symplectic manifold admits a sequence of Hamiltonian diffeomorphisms with diverging boundary depths (and hence Hofer norms), Corollary  \ref{encap} suggests that one could not expect  to obtain such a sequence by any purely local construction.

We turn back to algebra in Section \ref{alg2}, whose main result is Proposition \ref{depthattained}, which asserts roughly speaking that the supremum in one version of the definition of the boundary depth is attained.  This fact is used in the proof of Theorem \ref{spectrality}, and also in the following Section \ref{tprod}, which concerns boundary depths of tensor products of chain complexes.  Theorem \ref{prodbeta} plays an important role in the proof of Theorem \ref{lagprod} and hence also of Theorem \ref{s1prodmain}.
 
In Section \ref{circlesect} we give an elementary proof of the fact that the Hofer metric on the space of Lagrangian submanifolds of $\R^2$ Hamiltonian-isotopic to the unit circle has finite diameter.  It would not surprise me if this fact is already known, but I was unable to find documentation of this.   This example suggests that some degree of rigidity should be required of a Lagrangian submanifold $L$ before one expects $\mathcal{L}(L)$ to have infinite diameter.

Finally, the Appendix proves a technical result which implies that the index-one, $t$-independent solutions to the Floer equation for a time-independent Hamiltonian and almost complex structure can be arranged to be cut out transversely in certain circumstances.  This is needed in order to justify the identification of the Floer- and Morse-theoretic boundary depths that occurs in the proof of Theorem \ref{mmcomp}, and hence to complete the proof of Theorem \ref{hammain}. 

\subsection*{Acknowledgements}
I am grateful to the organizers of the Workshop on Symplectic Geometry and Topology (Kyoto, February 2011) and of GESTA 2011 (Castro Urdiales, June 2011) for the opportunity to present some of this work as it was in development.  In particular, talks by and conversations with F. Lalonde, A. Oancea, Y.-G. Oh, and C. Viterbo in Kyoto were influential in prompting me to see that the methods discussed here could prove more than I originally realized.  I also thank the University of Chicago for its hospitality during a visit in March 2011.  L. Polterovich's penetrating questions about the boundary depth during this time motivated me to develop the material in Section \ref{prologue}.  Finally, I am grateful to the anonymous referee for helpful comments and corrections.
This work was partially supported by NSF grant DMS-1105700.

\section{Morse-theoretic boundary depth in $S^1$}\label{prologue}

As an introduction to our main tool, we examine its behavior in perhaps the simplest nontrivial context, namely for Morse functions on the circle.  The calculation which we perform here will be of use to us later in our proof of Theorem \ref{s1prodmain}.

If $M$ is a closed manifold, $K$ a field, $f\co M\to\mathbb{R}$ a Morse function, and $g$ a metric with respect to which the gradient flow of $f$ is Morse-Smale, one obtains a Morse complex $CM_*(f)$ in a standard way (see, \emph{e.g.}, \cite{Sc93}): let $Crit_k(f)$ be the set of critical points of $f$ with Morse index $k$, and let  $CM_k(f;K)$ be $K$-vector space spanned by $Crit_k(f)$.  The boundary operator $\partial\co CM_k(f;K)\to CM_{k-1}(f;K)$ is given by, for $p\in Crit_k(f)$, \[ \partial p=\sum_{q\in Crit_{k-1}(f)}n(p,q)q,\] where $n(p,q)$ is the number of negative $g$-gradient flowlines from $p$ to $q$, counted in the field $K$ with appropriate signs as determined by chosen orientations of the unstable manifolds of the various critical points of $f$.  Of course one has $\partial\circ\partial=0$, and the resulting homology is isomorphic to the singular homology $H_*(M;K)$.  
For an element $a=\sum_i a_ip_i\in CM_*(f;K)$ set \[ \ell(a)=\max\{f(p_i)|a_i\neq 0\} \] and, for $\lambda\in\R$, define \[ CM^{\lambda}_{*}(f;K)=\{a\in CM_*(f;K)|\ell(a)\leq \lambda\}.\]  This gives a filtration by $\mathbb{R}$ on $CM_*(f;K)$.  Now define the boundary depth by \[ \beta_{Morse}(f;K)=\inf\{\beta\geq 0|(\forall \lambda\in\R)(CM^{\lambda}_{*}(f;K)\cap (Im\partial)\subset \partial(CM_{*}^{\lambda+\beta}(f;K)))\}.\]  Equivalently, as follows from a moment's thought, \[ \beta_{Morse}(f;K)=\left\{\begin{array}{ll} 0 & \mbox{ if }\partial=0, \\ \sup_{0\neq x\in Im\partial}\inf\{\ell(y)-\ell(x)|\partial y=x\} & \mbox{ if }\partial\neq 0.\end{array}\right.\]

One can also express the boundary depth in terms of the homologies of the sublevel sets of $f$ and their inclusions into $M$; we leave it to the reader to find a suitable formula.

We will now give another formula for the boundary depth of a Morse function on $S^1$.  As notation, if $k\geq 2$, define \[ S^{1}_{cyc,k}=\left\{(p_1,\ldots,p_k)\in (S^1)^{k}\left|\begin{array}{cc} p_1,\ldots,p_k \mbox{ are distinct and }\\\mbox{ in counterclockwise cyclic order on $S^1$}\end{array}\right.\right\}.\]

\begin{theorem}\label{s1calc}  Let $f\co S^1\to \R$ be a Morse function.  Then \begin{equation}\label{mainlink} \beta_{Morse}(f;K)=\sup\left\{\left.\min\{f(t_1),f(t_3)\}-\max\{f(t_2),f(t_4)\}\right|(t_1,t_2,t_3,t_4)\in S^{1}_{cyc,4}\right\}.\end{equation}
\end{theorem} 

\begin{proof} A Morse function on $S^1$ has the same number of local maxima as local minima; let $m$ denote this number.  Moreover maxima and minima obviously alternate as one goes around the circle, so we can denote the maxima by $p_1,\ldots,p_m$ and the minima as $q_1,\ldots,q_m$, labeling in such a way that \[ f(p_1)=\max_{S^1}f\quad \mbox{and}\quad (p_1,q_1,p_2,q_2,\ldots,p_m,q_m)\in S^{1}_{cyc,2m}.\]  With respect to the standard orientations of the unstable manifolds one has \[ \partial p_i=q_i-q_{i-1} \] where the indices are evaluated modulo $m$ (so $\partial p_1=q_1-q_{m}$).  

We first dispense with a trivial case: that in which $m=1$.  In this case the Morse differential vanishes, and so $\beta_{Morse}(f;K)=0$.  Meanwhile we assert that in this case the set over which the $\sup$ is taken in (\ref{mainlink}) contains no positive numbers.  Indeed, if $(t_1,t_2,t_3,t_4)\in S^{1}_{cyc,4}$ had the property that $\min\{f(t_1),f(t_3)\}>\max\{f(t_2),f(t_4)\}$ then we would obtain two distinct local minima $q_1,q_2$ of $f$ by having $q_1$ be a minimum of $f$ on the oriented  interval from $t_1$ to $t_3$ and having $q_2$ be a minimum of $f$ on the oriented  interval from $t_3$ to $t_1$, thus contradicting the assumption that $m=1$.  On the other hand the set over which the $\sup$ is taken in (\ref{mainlink}) certainly contains numbers arbitarily close to zero: just take the $t_i$ to be very close to each other.  Thus when $m=1$ both sides of ($\ref{mainlink}$) are zero.

So for the rest of the proof assume that $m\geq 2$.  Since $p_1$ is a global maximum for $f$ and $p_2$ is the unique local maximum on the oriented interval from $q_1$ to $q_2$ we have $\min\{f(p_1),f(p_2)\}-\max\{f(q_1),f(q_2)\}>0$ with $(p_1,q_1,p_2,q_2)\in S^{1}_{cyc,4}$, so the right hand side of (\ref{mainlink}) is positive.  

\begin{claim}The right hand side of (\ref{mainlink}) is equal to \begin{equation}\label{gammaS1} \gamma(f)=\max\left\{\min\{f(p_i),f(p_j)\}-\max\{f(q_k),f(q_l)\}\left|(p_i,q_k,p_j,q_l)\in S^{1}_{cyc,4}\right.\right\}.\end{equation}  
\end{claim}  In other words, we are claiming that (when $m\geq 2$) we may evaluate the right hand side of (\ref{mainlink}) by restricting to the case that $t_1,t_3$ are local maxima and $t_2,t_4$ are local minima.  Indeed, suppose we have some $(t_1,t_2,t_3,t_4)\in S^{1}_{cyc,4}$ with $\min\{t_1,t_3\}>\max\{t_2,t_4\}$ (by our earlier remarks such $t_i$ do exist when $m\geq 2$).  Set $p_i$ equal to a global maximum of $f$ on the oriented interval from $t_4$ to $t_2$, and $p_j$ equal to a global maximum on the oriented interval from $t_2$ to $t_4$.  Then $(p_i,t_2,p_j,t_4)\in S^{1}_{cyc,4}$ with $\min\{f(p_i),f(p_j)\}>\max\{f(t_2),f(t_4)\}$, so set $q_k$ equal to a global minimum of $f$ on the oriented interval from $p_i$ to $p_j$ and $q_l$ equal to a global minimum of $f$ on the oriented interval from $p_j$ to $p_i$.  Then $(p_i,q_k,p_j,q_l)\in S^{1}_{cyc,4}$ and \[ \min\{f(p_i),f(p_j)\}-\max\{f(q_k),f(q_l)\}\geq  \min\{f(t_1),f(t_3)\}-\max\{f(t_2),f(t_4)\},\] which clearly suffices to prove the claim.

It thus remains only to show that $\beta_{Morse}(f;K)=\gamma(f)$ where $\gamma(f)$ is defined in (\ref{gammaS1}).  Note that since $\partial p_i=q_i-q_{i-1}$ we easily find that \begin{equation}\label{kerS1} \ker\partial=\left\{\left. n\sum_{i=1}^{m}p_i\right|n\in\Z\right\} \end{equation} and \begin{equation}\label{imS1} Im\partial=\left\{\left.\sum_{j=1}^{m}n_jq_j\right|\sum n_j=0\right\}.\end{equation}

Choose $i,j,k,l$ achieving the maximum in (\ref{gammaS1}); without loss of generality (since $S^{1}_{cyc,4}$ is invariant under cyclic permutations) say $k<l$.  Let \[ x_0=q_l-q_k\in CM_0(f;K).\]  Then where \[ y_0=\sum_{r=k+1}^{l}p_r \] we have $\partial y_0=x_0$.  So given (\ref{kerS1}), any $y\in CM_1(f;K)$ such that $\partial y=x_0$ has the form \[ y=\sum_{r\in\{k+1,\ldots,l\}}(n+1)p_r+  \sum_{r\notin\{k+1,\ldots,l\}}np_r \] for some $n\in \Z$.  Now since $(p_i,q_k,p_j,q_l)\in S^{1}_{cyc,4}$ we have $j\in \{k+1,\ldots,l\}$ and $i\notin\{k+1,\ldots,l\}$.  Hence if $y=\sum a_rp_r$ has $\partial y=x_0$, the coefficients on $a_i$ and $a_j$ differ by one; in particular (regardless of what field we are working over) they are not both zero.  Thus $\ell(y)\geq \min\{f(p_i),f(p_j)\}$.  So by our choice of $i,j,k,l$ we have shown that \[ \inf\{\ell(y)-\ell(x_0)|\partial y=x_0\}\geq \min\{f(p_i),f(p_j)\}-\max\{f(q_k),f(q_l)\}=\gamma(f).\]  This  proves that \[ \beta_{Morse}(f;K)\geq \gamma(f).\]

We now prove the reverse inequality.  To do this we must show that, if $x=\sum_i n_i q_i$ is a nonzero element of $Im\partial$, then there is $y$ such that $\partial y=x$ and $\ell(y)-\ell(x)\leq \gamma(f)$.

We will show, specifically, that the element \[ y=-\sum_{j=2}^{m}\left(\sum_{i=1}^{j-1}n_i\right)p_j \] satisfies the required property.  Given that $x\in Im\partial$ and hence $\sum_{i=1}^{m} n_i=0$ by (\ref{imS1}), it is easy to see that $\partial y=x$.  Now we have \begin{equation}\label{elly} \ell(y)=\max\left\{f(p_j)\left|\sum_{i=1}^{j-1}n_i\neq 0\right.\right\} \end{equation} (of course, the set above is nonempty since we assume $x\neq 0$),
and \begin{equation}\label{ellx} \ell(x)=\max\{f(q_i)|n_i\neq 0\}\end{equation}

Let $j_1$ be the index corresponding to the maximum in (\ref{elly}) and let $i_1$ be the index corresponding to the maximum in (\ref{ellx}).  Since our ordering was such that $p_1$ was a global maximum for $f$, we have \[ \ell(y)=f(p_{j_1})=\min\{f(p_1),f(p_{j_1})\}.\]  
Let $i_0$ and $i_2$ be, respectively, the minimal and maximal elements of $\{i|n_i\neq 0\}$.  Since \[ \sum_{i=1}^{m}n_i=\sum_{i=i_0}^{i_2}n_i=0 \] we have $i_0+1\leq j_1\leq i_2$.   Of course, \[ \ell(x)=f(q_{i_1})=\max\{f(q_{i_0}),f(q_{i_1})\}=\max\{f(q_{i_1}),f(q_{i_2})\}.\]

If $i_1<j_1$, then we have $(p_1,q_{i_1},p_{j_1},q_{i_2})\in S^{1}_{cyc,4}$ with \[ \ell(x)-\ell(y)=\min\{f(p_1),f(p_{j_0})\}-\max\{f(q_{i_1}),f(q_{i_2})\}\leq \gamma(f),\] while if $j_1\leq i_1$ we have $(p_1,q_{i_0},p_{j_1},q_{i_1})\in S^{1}_{cyc,4}$ with \[ \ell(x)-\ell(y)=\min\{f(p_1),f(p_{j_0})\}-\max\{f(q_{i_0}),f(q_{i_1})\}\leq \gamma(f).\]  So we have indeed shown that any nonzero $x\in Im\partial$ has a preimage $y$ with $\ell(y)-\ell(x)\leq \gamma(f)$, implying that $\beta_{Morse}(f;K)\leq \gamma(f)$.
\end{proof}

\begin{remark} While $\beta_{Morse}(f;K)$ is defined using Morse theory, the formula on the right-hand side of (\ref{mainlink}) obviously does not require $f$ to be a Morse function---or even to be differentiable.  Moreover Theorem \ref{s1calc} clearly shows that $\beta_{Morse}(\cdot;K)$ is continuous with respect to the $C^0$-norm.  This latter property continues to hold on a general manifold $M$, and so implies that $\beta_{Morse}(\cdot;K)$ can always be canonically extended to all of $C^0(M;\R)$; we will prove analogues of this fact in the Floer-theoretic context later.
\end{remark}

\begin{remark} In (\ref{mainlink}), the fact that $(t_1,t_2,t_3,t_4)\in S^{1}_{cyc,4}$ amounts to the fact that the two copies of $S^0\subset S^1$ given by $\{t_1,t_3\}$ and $\{t_2,t_4\}$ are \emph{linked} in the sense that, while both copies of $S^0$ are nullhomologous, any chain whose boundary is equal to one of the copies of $S^0$ must intersect the other copy.  Thus on $S^1$ the statement that the boundary depth of a function $f$ is nonzero amounts to the statement that one can find two linked copies $C_0,C_1$ of $S^0\subset S^1$ such that $f|_{C_0}>f|_{C_1}$.  A similar ``linking'' interpretation of the Morse-theoretic boundary depth on more general manifolds has recently been obtained in \cite[Proposition 5.6 and Theorem 5.9]{U12}.
\end{remark}

\section{General algebraic considerations}\label{algintro}

Let us formulate abstractly some of the relevant algebraic notions.  First of all, as notation, if $S$ is a set equipped with an action of the integers $\mathbb{Z}$ and if $k\in S$ we denote by $k+1$ and $k-1$ the results of acting on $k$ by, respectively, $1$ and $-1$.  

\begin{dfn}\label{rfilt} Let $S$ be a set equipped with an action of  $\mathbb{Z}$, and let $K$ be a field.  An \emph{S}-graded, $\mathbb{R}$-filtered complex over $K$ consists of the following data:
\begin{itemize}\item A $K$-vector space $C$ together with a $K$-linear map $\partial\co C\to C$ such that $\partial\circ\partial=0$.
\item For each $k\in S$, a subspace $C_k\leq C$ such that $C=\oplus_{k\in S}C_k$ and $\partial(C_k)\leq C_{k-1}$.
\item For each $k\in S$ and each $\lambda\in \R$, a subspace $C_{k}^{\lambda}\leq C_k$ such that \begin{itemize}\item $C_k=\bigcup_{\lambda\in\R}C_{k}^{\lambda}$; \item $\bigcap_{\lambda\in \R}C_{k}^{\lambda}=\{0\}$; \item  if $\lambda<\mu$ then $C_{k}^{\lambda}\leq C_{k}^{\mu}$; and \item $\partial(C_{k}^{\lambda})\leq C_{k-1}^{\lambda}$.\end{itemize}
We write \[ C^{\lambda}=\oplus_{k\in S}C_{k}^{\lambda}.\]
\end{itemize}
\end{dfn}

\begin{dfn} Let $(C,\partial)$ and $(D,\delta)$ be two $S$-graded, $\R$-filtered complexes over $K$, and let $c\in \R$.  A \emph{$c$-morphism} $\Phi\co C\to D$ is a $K$-linear map such that\begin{itemize}\item $\Phi$ is a chain map: $\Phi\circ \partial=\delta\circ \Phi$ \item For all $k\in S$ and $\lambda\in \R$ we have \[ \Phi(C_{k}^{\lambda})\leq D_{k}^{\lambda+c}.\]   \end{itemize}
\end{dfn}

\begin{dfn}  Let $\Phi_1,\Phi_2\co C\to D$ be two chain maps where $(C,\partial)$ and $(D,\delta)$ are two $S$-graded, $\R$-filtered complexes over $K$, and let $c\in \R$.  A \emph{$c$-homotopy} from $\Phi_1$ to $\Phi_2$ is a $K$-linear map $\mathcal{K}\co C\to D$ which, for each $k\in S$ and $\lambda\in \R$, obeys \begin{itemize}\item $\mathcal{K}(C_{k}^{\lambda})\leq D_{k+1}^{\lambda+c}$, and \item $\Phi_2-\Phi_1=\mathcal{K}\partial+\delta\mathcal{K}.$\end{itemize}
\end{dfn}

\begin{dfn}\label{shiftiso}  Let $(C,\partial)$ and $(D,\delta)$ be two $S$-graded, $\R$-filtered complexes over $K$.  A \emph{shift-isomorphism} from $C$ to $D$ consists of the following data:\begin{itemize}\item
A bijection $\phi\co S\to S$ such that $\phi(k+1)=\phi(k)+1$ for all $k\in S$.
\item A function $\sigma\co S\to \mathbb{R}$ such that $\sigma(k+1)=\sigma(k)$ for all $k\in S$.
\item A chain map $\Phi\co C\to D$ such that, for each $k\in S$ and $\lambda\in \mathbb{R}$, $\Phi$ restricts as an isomorphism from $C_{k}^{\lambda}$ to $D_{\phi(k)}^{\lambda+\sigma(k)}$.\end{itemize}
\end{dfn}

Of course, using that compositions and inverses of bijective chain maps are bijective chain maps, shift-isomorphism defines an equivalence relation on $S$-graded, $\R$-filtered complexes.

\begin{dfn}  Let $(C,\partial)$ be an $S$-graded, $\mathbb{R}$-filtered complex over $K$.  If $k\in S$, the \emph{boundary depth of $C$ in grading $k$} is the quantity \[ b_k(C,\partial)=\inf\left\{\beta\geq 0\left|(\forall \lambda\in \mathbb{R})\left((Im \partial)\cap C_{k}^{\lambda} \subset \partial(C_{k+1}^{\lambda+\beta})\right)\right.\right\}.\]
The \emph{boundary depth of $C$} is simply the quantity \[ b(C,\partial)=\sup_{k\in S}b_k(C,\partial).\]
\end{dfn}

It is easy to check that, equivalently, \[ b(C,\partial)=\inf\left\{\beta\geq 0\left|(\forall \lambda\in \mathbb{R})\left((Im \partial)\cap C^{\lambda} \subset \partial(C^{\lambda+\beta})\right)\right.\right\},\] \emph{i.e.}, $b(C,\partial)$ is just the boundary depth that we would obtain by forgetting about the grading (or rather, by considering $C$ to be graded by a one-element set).  

In principle, the set over which we take the infimum in the definition of $b_k(C,\partial)$ could be empty, in which case we would set $b_k(C,\partial)=\infty$.  However, in the situations that we will consider, the complex $(C,\partial)$ will have additional structure which will guarantee the finiteness of $b_k(C,\partial)$ using results such as \cite[Theorem 1.3]{U08} and/or \cite[Proposition 6.3.9]{FOOO09}. 

Note that in the case where $\partial|_{C_{k+1}}=0$ (including the case where $C_k$ or $C_{k+1}$ is zero), we have $b_k(C,\partial)=0$.

\begin{prop}\label{shiftisoprop} If there is a shift-isomorphism $(\Phi,\phi,\sigma)$ from $(C,\partial)$ to $(D,\delta)$ then, for all $k\in S$, \[ \beta_k(C,\partial)=\beta_{\phi(k)}(D,\delta).\]  In particular \[ \beta(C,\partial)=\beta(D,\delta).\]
\end{prop}

\begin{proof} For $x\in C$, we have $x\in C_{k}^{\lambda}$ iff $\Phi x\in D_{\phi(k)}^{\lambda+\sigma(k)}$, and, using that $\Phi$ is a chain map while $\phi(k+1)=\phi(k)+1$ and $\sigma(k+1)=\sigma(k)$, it holds that $y\in C_{k+1}^{\lambda+\beta}$ and $\partial y=x$ iff $\Phi y\in C_{\phi(k)+1}^{\lambda+\sigma(k)+\beta}$ and $\delta\Phi y=\Phi x$.  Since $\Phi$ and $\phi$ are bijections the proposition follows immediately from the definitions. 
\end{proof}

\begin{dfn}  Let $c\in\mathbb{R}$, $k\in S$ and let $(C,\partial)$, $(D,\delta)$ be two $S$-graded, $\R$-filtered complexes over $K$.  We say that $(C,\partial)$ and $(D,\delta)$ are \emph{$c$-quasiequivalent} if there are:\begin{itemize} \item numbers $c_1$ and $c_2$ such that $c_1+c_2\leq c$; \item a  $c_1$-morphism $\Phi\co C\to D$ and a  $c_2$-morphism $\Psi\co D\to C$; and \item a  $c$-homotopy $\mathcal{K}_1\co C\to C$ from the identity to $\Psi\circ \Phi$, and a $c$-homotopy $\mathcal{K}_2\co D\to D$ from the identity to $\Phi\circ \Psi$.\end{itemize}
\end{dfn}

\begin{prop}\label{quasicont}  If $(C,\partial)$ and $(D,\delta)$ are $c$-quasiequivalent then for all $k\in S$ we have \[ |b_k(C,\partial)-b_k(D,\delta)|\leq c.\]  Thus in particular $|b(C,\partial)-b(D,\delta)|\leq c$.\footnote{In case some of the terms involved here are infinite, these equations should be read as stating that if one of the terms is infinite then so is the other.}.
\end{prop}

\begin{proof}  By symmetry it suffices to prove that $b_k(C,\partial)\leq b_k(D,\delta)+c$.  In other words, we should show that, if $\beta>0$ has the property that, for all $\lambda\in \R$, $(Im \delta)\cap D^{\lambda}_{k}\subset \delta(D^{\lambda+\beta}_{k+1})$, then it also holds that, for all $\lambda\in\R$, $(Im\partial)\cap C^{\lambda}_{k}\subset \partial(C^{\lambda+c+\beta}_{k+1})$.

So let $c_1,c_2,\Phi,\Psi,\mathcal{K}_1$ be as in the definition of $c$-quasiequivalence and let $x\in (Im\partial)\cap C^{\lambda}_{k}$.  Since $\Phi$ is a $c_1$-morphism, we have $\Phi x\in (Im\delta)\cap D_{k}^{\lambda+c_1}$.  So by the assumption on $\beta$ there is $y\in D_{k+1}^{\lambda+c_1+\beta}$ such that $\delta y=\Phi x$.  So since $\Psi$ is a $c_2$-morphism and $c_1+c_2\leq c$ we have \[ \partial \Psi y=\Psi\circ \Phi(x)\quad\mbox{and}\quad \Psi y\in C_{k+1}^{\lambda+c_1+c_2+\beta}\leq C_{k+1}^{\lambda+c+\beta}.\]  Of course, since $x$ is a boundary we have $\partial x=0$, and so the chain homotopy equation reads
\[ \Psi\circ \Phi(x)-x=\partial \mathcal{K}_1 x,\] where $\mathcal{K}_1x\in C_{k+1}^{\lambda+c}$ since $\mathcal{K}_1$ is a $c$-homotopy.  So (using that $\beta\geq 0$) we have \[ x=\partial(\Psi y-\mathcal{K}_1x)\quad \mbox{where} \quad \Psi y-\mathcal{K}_1x\in C_{k+1}^{\lambda+c+\beta},\] as desired.

\end{proof}

\section{Quantum corrections and the boundary depth}\label{qcor}

The Lagrangian Floer homology $HF(L,L)$ of a Lagrangian submanifold $L$ can be obtained as the homology of a chain complex whose boundary operator is obtained by adding to the standard Morse boundary operator on $L$ an operator defined over a Novikov field which represents ``quantum corrections'' (see the construction of the ``pearl complex'' in \cite{BC07}, and also the construction in \cite{FOOO09c}).  We presently put this idea in abstract algebraic terms, and then make an observation concerning the boundary depths of complexes obtained in this way.

First, if $K$ is a field and $\Gamma\leq\R$ is an additive subgroup, the \emph{Novikov field} associated to $K$ and $\Gamma$ is the field \[ \Lambda^{K,\Gamma}=\left\{\sum_{g\in\Gamma}a_gT^g\left|a_g\in K,\,(\forall C\in \R)(\#\{g\in \Gamma|a_g\neq 0,\,g<C\}<\infty)\right.\right\},\] 
equipped with the obvious ``power series'' addition and multiplication.

\begin{dfn} If $K$ is a field, we say that a field $\Lambda$ is a \emph{Novikov field over $K$} if we have $\Lambda=\Lambda^{K,\Gamma}$ for some subgroup $\Gamma\leq\R$.
\end{dfn}

If $\Lambda$ is a Novikov field over $K$, we view $K$ as the subring of $\Lambda$ consisting of those elements $\sum a_gT^g$ with $a_g=0$ for all $g\neq 0$.  Also, let \[ \Lambda_{\geq 0}=\left\{\sum a_gT^g\in \Lambda\left|a_g\neq 0\Rightarrow g\geq 0\right.\right\}\] and \[ \Lambda_+=\left\{\sum a_gT^g\in \Lambda\left|a_g\neq 0\Rightarrow g> 0\right.\right\}.\]

\begin{dfn} Where $S$ is a set with an action of $\Z$, let $(\bar{C}=\oplus_{k\in S}\bar{C}_k,\partial_0)$ be a chain complex of $K$-vector spaces.   Let $\Lambda=\Lambda^{K,\Gamma}$ be a Novikov field over $K$, and let $\mu\co S\to \Gamma\cap (0,\infty)$ be a function.  A chain complex $(C=\oplus_{k\in S}C_k,\partial)$ of $\Lambda$-modules is called a \emph{quantum correction of $\bar{C}$ of gap at least $\mu$} if the following holds:
\begin{itemize} \item For all $k\in S$ we have $C_k=\bar{C}_k \otimes_K \Lambda$
\item  Where we use the inclusion of $K$ to view each $\bar{C}_k$ as a subgroup of $C_k\otimes_K\Lambda$, we have \[ (\partial-\partial_0\otimes 1)(\bar{C}_{k+1})\leq T^{\mu(k)}(\bar{C}_k \otimes_K \Lambda_{\geq 0}).\] \end{itemize}
\end{dfn}

In particular, in the context of the above definition it always holds that $\partial(\bar{C}\otimes_K\Lambda_{\geq 0})\leq \bar{C}\otimes_K\Lambda_{\geq 0}.$  Of course, if $\partial|_{\bar{C}_{k+1}}=(\partial_0\otimes 1)|_{\bar{C}_{k+1}}$ then we may choose to take $\mu(k)$ as large as we like.

If $(C,\partial)$ is a quantum correction of $(\bar{C},\partial_0)$, then $(C,\partial)$ naturally has the structure of an $S$-graded, $\mathbb{R}$-filtered complex over $K$ (not over $\Lambda$!) in the sense defined earlier.  Namely, for any $k$ define the function \[ \barnu\co C_k\to \mathbb{R}\cup\{\infty\}\] by \begin{equation}\label{barnudef} \barnu(x)=\sup\{\mu\in \R|x\in T^{\mu}(\bar{C}\otimes_K\Lambda_{\geq 0})\}.\end{equation} Then for $\lambda\in \R$ and $k\in S$ we set \[ C_{k}^{\lambda}=\{x\in C| -\barnu(x)\leq \lambda \}.\]  Verification of the relevant axioms is straightforward.

\begin{prop}\label{qdef}  Let $(\bar{C},\partial_0)$ be a chain complex of $K$-vector spaces such that each $\bar{C}_k$ is finite-dimensional over $K$, and let $(C,\partial)$ be a quantum correction of $(\bar{C},\partial_0)$ of gap at least $\mu$.  Then $b_k(C,\partial)<\infty$ for all $k$.  Moreover, for all $k\in S$, exactly one of the following two alternatives holds:
\begin{itemize}\item[(i)] $b_k(C,\partial)=b_{k-1}(C,\partial)=0$, and $\dim_{\Lambda}H_k(C,\partial)=\dim_{R}H_k(\bar{C},\partial_0)$; \textbf{or} \item[(ii)] $b_k(C,\partial)\geq \mu(k)$ or $b_{k-1}(C,\partial)\geq \mu(k-1)$, and $\dim_{\Lambda}H_k(C,\partial)<\dim_{R}H_k(\bar{C},\partial_0)$.\end{itemize}
\end{prop}

\begin{proof}

 If we choose a basis of $\bar{C}_k$ over $K$ and use this basis to identify $C_k=\bar{C}_k\otimes_K\Lambda$ with $\Lambda^N$ for some $N$ then the function $\barnu$ defined in (\ref{barnudef}) coincides with the function $\barnu$ defined at the start of \cite[Section 2]{U08}.  The assertion that $b_k(C,\partial)$ is finite then follows immediately from \cite[Theorem 2.5]{U08} (in the notation of that theorem, set $\vec{t}=0$ and let $A$ be a matrix representing the boundary operator $\partial\co C_{k+1}\to C_k$).  Further borrowing notation from \cite{U08} and \cite{U10b}, if $U\leq C_k$ is a subspace 
write \[ U_{\geq 0}=\{x\in U|\barnu(x)\geq 0\},\quad U_+=\{x\in U|\barnu(x)>0\},\quad \widetilde{U}=\frac{U_{\geq 0}}{U_+}.\]  Thus $U_{\geq 0}$ is a $\Lambda_{\geq 0}$-module, and $\widetilde{U}$ is a $K$-vector space.  If $x\in U_{\geq 0}$ we denote its image in the quotient $\widetilde{U}$ by $\widetilde{x}$.  Also, we define \begin{align*} \nu\co \Lambda&\to \mathbb{R}\cup\{\infty\} \\ \sum_g a_gT^g&\mapsto \min\{g:a_g\neq 0\}.\end{align*}

As in \cite{GG67},\cite{U10b}, if $U\leq C_k$ we will call a basis $\{u_1,\ldots,u_m\}$ for $U$ \emph{orthonormal} if, for all $\lambda_1,\ldots,\lambda_m\in \Lambda$, we have \begin{equation}\label{oneqn} \barnu\left(\sum_{j=1}^{m}\lambda_ju_j\right)=\min_{1\leq j\leq m}\nu(\lambda_j).\end{equation}

\begin{lemma}[\cite{GG67},\cite{U10b}]\label{onlemma} If $U\leq C_k$, then a subset $\{u_1,\ldots,u_m\}\subset U_{\geq 0}$ is an orthonormal basis for $U$ if and only if $\{\widetilde{u}_1,\ldots,\widetilde{u}_m\}$ is a basis for $\widetilde{U}$.  Consequently any subspace $U\leq C_k$ has an orthonormal basis, and if $U\leq V\leq C_k$ then any orthonormal basis of $U$ can be extended to an orthonormal basis of $V$. Moreover, $\dim_{\Lambda}U=\dim_K\widetilde{U}$.
\end{lemma}

\begin{proof}[Proof of Lemma \ref{onlemma}] The sufficiency of the condition in the first sentence is proven in the proof of \cite[Lemma 2.1]{U10b}\footnote{Strictly speaking it is assumed throughout \cite{U10b} that the subgroup $\Gamma\leq \R$ used to define the Novikov field is countable, but this assumption is not used in the proof of \cite[Lemma 2.1]{U10b}.  We also mention here that the notion of an orthonormal basis is closely related to that of a \emph{standard basis} from \cite[Section 6.3]{FOOO09}.}.
Since a basis for $\widetilde{U}$ can always be found, it follows that $U$ has an orthonormal basis, and that $\dim_{\Lambda}U=\dim_{K}\widetilde{U}$.  For the necessity of the condition in the first sentence, if $\{u_1,\ldots,u_m\}$ is an orthonormal basis then (\ref{oneqn}) immediately implies that the $\widetilde{u}_i$ are linearly independent in $\widetilde{U}$, and so they form a basis for $\widetilde{U}$ since $\dim_{\Lambda}U=\dim_{K}\widetilde{U}$.  

Now that we have proven the first sentence, the only remaining statement, namely that any orthonormal basis for $U$ can be extended to a orthonormal basis of $V$ if $U\leq V$, follows directly from the facts that if $U\leq V$ then $\widetilde{U}\leq \widetilde{V}$, and that any basis of $\widetilde{U}$ can be extended to a basis of $\widetilde{V}$.
\end{proof}




Since for all $k$ we have \[ \dim_{\Lambda}H_k(C,\partial)=\dim_{\Lambda}\ker(\partial|_{C_k})-rank(\partial|_{C_{k+1}})=\dim_{\Lambda}C_k-rank(\partial|_{C_{k+1}})-rank(\partial|_{C_k})\] and likewise \[ \dim_{K}H_k(\bar{C}_k,\partial_0)=\dim_K \bar{C}_k-rank(\partial_0|_{\bar{C}_{k+1}})-rank(\partial_0|_{\bar{C}_k}), \] and since $\dim_{\Lambda}C_k=\dim_K \bar{C}_k$, the proposition now follows from the following lemma:

\begin{lemma}\label{qlem} For all $k$ we have $rank(\partial_0|_{\bar{C}_{k+1}})\leq rank(\partial|_{C_{k+1}})$.  Moreover, if $rank(\partial_0|_{\bar{C}_{k+1}})< rank(\partial|_{C_{k+1}})$ then $b_k(C,\partial)\geq \mu(k)$, while if $rank(\partial_0|_{\bar{C}_{k+1}})= rank(\partial|_{C_{k+1}})$ then $b_k(C,\partial)=0$.
\end{lemma}

Indeed, given Lemma \ref{qlem}, the alternative (i) in Proposition \ref{qdef} occurs exactly when both $rank(\partial_0|_{\bar{C}_{k+1}})= rank(\partial|_{C_{k+1}})$ and $rank(\partial_0|_{\bar{C}_{k}})= rank(\partial|_{C_{k}})$; otherwise, alternative (ii) in Proposition \ref{qdef} holds, thus completing the proof of that proposition modulo the proof of Lemma \ref{qlem}. \end{proof}

\begin{proof}[Proof of Lemma \ref{qlem}] Throughout this proof we make implicit use of the embedding of $\bar{C}_k$ into $(C_k)_{\geq 0}\leq C_k$ induced by the inclusion of $K$ into $\Lambda$ (as elements consisting only of multiples of $T^0$).  If $\bar{U}\leq \bar{C}_k$ is any subspace, this embedding induces an isomorphism \[ \bar{U}\cong  (\bar{U} \otimes_K \Lambda)^{\widetilde{}}.\]  

Choose $x_1,\ldots,x_m\in \bar{C}_{k+1}\leq C_{k+1}$ so that $\partial_0x_1,\ldots,\partial_0x_m$ forms a basis for $Im(\partial_0|_{\bar{C}_{k+1}})$.   We claim that the elements $\partial x_1,\ldots,\partial x_m$ are linearly independent over $\Lambda$ in $C_k$.  Indeed, if $\lambda_1,\ldots,\lambda_m\in \Lambda$ are not all zero and if $g=\min_i \nu(\lambda_i)$, then since $\barnu(\partial x_i-\partial_0x_i)>0$ for all $i$ we see that the element \[ \widetilde{\left(T^{-g}\sum_{i=1}^{m}\lambda_i\partial x_i\right)}\] is a nontrivial linear combination of the $\partial_0x_i$ and so is nonzero.  This proves that \begin{equation}\label{rkineq} rank(\partial_0|_{\bar{C}_{k+1}})\leq rank(\partial|_{C_{k+1}}).\end{equation}

Now suppose that equality holds in (\ref{rkineq}).  Then by the last sentence of Lemma \ref{onlemma} we have \[ \dim_K\left(\widetilde{Im(\partial|_{C_{k+1}})}\right)=\dim_{\Lambda}Im(\partial|_{C_{k+1}})=m;\] thus since $\partial_0 x_i=\widetilde{\partial x_i}$ a dimension count shows that the $\partial_0 x_i$ form a basis for $\widetilde{Im(\partial|_{C_{k+1}})}$, and hence by Lemma \ref{onlemma} the $\partial x_i$ form an orthonormal basis for $Im(\partial|_{C_{k+1}})$.  Meanwhile the $x_i$ (being linearly independent elements of the ``level zero'' subspace $\bar{C}_{k+1}\leq C_{k+1}$) obviously form an orthonormal basis for the subspace of $C_{k+1}$ which they span.  Consequently if $a\in Im(\partial|_{C_{k+1}})$ we can find $\lambda_1,\ldots,\lambda_m\in \Lambda$ so that \[ a=\sum_i\lambda_i\partial x_i=\partial\left(\sum_i\lambda_i x_i\right),\] and using orthonormality we see that \[ \barnu(a)=\min_i \nu(\lambda_i)=\barnu\left(\sum_i\lambda_i x_i\right).\]  This proves that, when equality holds in (\ref{rkineq}), we have $b_k(C,\partial)=0$.

It remains to consider the case that $rank(\partial_0|_{\bar{C}_{k+1}})< rank(\partial|_{C_{k+1}})$.  In this case, we again let $x_1,\ldots,x_m\in \bar{C}_{k+1}\leq C_{k+1}$ have the property that the $\partial_0 x_i$ form a basis for the image of $\partial_0|_{C_{k+1}}$.  Let $U$ denote the subspace of $C_{k+1}$ spanned by the $x_i$.  As noted earlier, the $x_i$ (viewed now as elements of $C_{k+1}$) form an orthonormal basis for $U$.  Using Lemma \ref{onlemma}, extend this basis to an orthonormal basis $\{x_1,\ldots,x_m,\ldots,x_p\}$ for the subspace $U\oplus \ker\partial=\partial^{-1}(\partial U)$.  The fact that $rank(\partial_0|_{\bar{C}_{k+1}})< rank(\partial|_{C_{k+1}})$ implies that this subspace of $C_{k+1}$ is proper, so we extend the basis further to an orthonormal basis $\{x_1,\ldots,x_p,z_1,\ldots,z_q\}$ for all of $C_{k+1}$, where $q\geq 1$ by the assumption on the ranks.  By Lemma \ref{onlemma}, to perform this further extension it suffices to choose the $z_i\in (C_k)_{\geq 0}$ in such a way that the reductions $\widetilde{x}_1,\ldots,\widetilde{x}_p,\widetilde{z}_1,\ldots,\widetilde{z}_q$ are linearly independent over $K$.  In particular, we may assume that $z_1$ belongs to $\bar{C}_{k+1}$: indeed, if it does not initially, then we can simply subtract off all of its higher order terms, which does not change $\widetilde{z}_1$ and so does not affect the orthonormality of the basis.

Having done this, the condition on $x_1,\ldots,x_m$ shows that there are $c_1,\ldots,c_m\in K$ such that $\partial_0z_1=\sum_{i=1}^{m}c_i\partial_0 x_i$.  Now set \[ z=z_1-\sum_{i=1}^{m}c_ix_i,\] so that $z\in \bar{C}_{k+1}\leq C_{k+1}$ with $\partial_0z=0$, and let \[ a=\partial z.\]  By construction, we have \[ a-\partial z_1\in \partial U,\] so since $\partial z_1\notin\partial U$ we have $a\neq 0$.  Additionally, since $\partial_0z=0$ we have \[ a=(\partial-\partial_0\otimes 1)z\in T^{\mu(k)}(\bar{C}_{k}\otimes_K\Lambda_{\geq 0}).\]  Meanwhile, if $z'\in C_{k+1}$ has $\partial z'=a$, then $z'-z_1\in \partial^{-1}(\partial U)$, and so there are $\lambda_1,\ldots,\lambda_p$ such that \[ z'=z_1+\sum_{i=1}^{p}\lambda_ix_i.\]  So by the orthonormality of the basis $\{x_1,\ldots,x_p,z_1,\ldots,z_q\}$ we have \[ \barnu(z')=\min\{\nu(1),\nu(\lambda_1),\ldots,\nu(\lambda_p)\}\leq \nu(1)=0.\]  Thus we have found $a\in Im(\partial|_{C_{k}})$ such that $\barnu(a)\geq \mu(k)$ and such that any $z'\in C_{k+1}$ with $\partial z'=a$ has $\barnu(z)\leq 0$.  This proves that $b_k(C,\partial)\geq \mu(k)$.
\end{proof}

\section{Hamiltonian Floer theory}\label{hamsect}

Let $(M,\omega)$ be a closed symplectic manifold and let $H\co S^1\times M\to \R$ be a smooth function.\footnote{Throughout this paper we will identify $S^1$ with $\R/\Z$.  Of course, any element $\phi\in Ham(M,\omega)$ can be generated by a smooth Hamiltonian whose domain is $S^1\times M$ rather than $[0,1]\times M$, by replacing a generating Hamiltonian $H\co [0,1]\times M\to\R$ by $\chi'(t)H(\chi(t),m)$ for a suitable monotone homeomorphism $\chi\co [0,1]\to[0,1]$ with $\chi'(0)=\chi'(1)=0$.}     One then obtains the time-dependent Hamiltonian vector field $X_H$ by the prescription that $\omega(X_H(t,\cdot),\cdot)=-d(H(t\cdot))$, and the Hamiltonian flow $\{\phi_{H}^{t}\}_{t\in \R}$ as the flow of $X_H$.  

Assume for the time being that $H$ is nondegenerate in the sense that at each fixed point $p$ of $\phi_{H}^{1}$ the linearization $d_p\phi_{H}^{1}\co T_pM\to T_pM$ does not have $1$ as an eigenvalue.  Let $\mathcal{L}M$ denote the free loopspace of $M$ (for definiteness we require elements of $\mathcal{L}M$ to be $C^1$).  For each path component $\frak{c}$ of $\mathcal{L}M$, choose a smooth loop $\gamma_{c}$ representing $\frak{c}$, and choose a symplectic trivialization $\tau_{\frak{c}}$ of the bundle $\gamma_{\frak{c}}^{*}TM\to S^1$.  (For the special case in which $\frak{c}$ is the component of contractible loops we will take $\gamma_{\frak{c}}$ equal to a constant.  Hereinafter the component of contractible loops will be denoted by $\frak{c}_0$.).  Now a loop in $\mathcal{L}M$ corresponds in obvious fashion to a map $u\co T^2\to M$ of the torus into $M$; consequently  each $\frak{c}\in \pi_0(\mathcal{L}M)$ gives rise via consideration of loops in $\frak{c}$ to a subgroup $H_{2}^{\frak{c}}\leq H_{2}(M;\Z)$, all of whose elements may be represented by maps of $2$-tori into $M$ (when $\frak{c}=\frak{c}_0$ is the trivial class elements of $H_{2}^{\frak{c}}$ can indeed be represented by spheres, but this is typically not so for nontrivial classes).  

For any $\frak{c}\in \pi_0(\mathcal{L}M)$ consider pairs $(\gamma,w)$ where $\gamma\in \frak{c}$ and $w\co [0,1]\times S^1\to M$ is a map such that $w(0,\cdot)=\gamma_{\frak{c}}$ and $w(1,\cdot)=\gamma$.  Declare two such pairs $(\gamma,w),(\gamma',w')$ equivalent if and only if each of the following holds: \begin{itemize}
\item[(i)] $\gamma=\gamma'$, and  \item[(ii)] $\int_{[0,1]\times S^1}w^*\omega=\int_{[0,1]\times S^1}w'^*\omega$. 
\end{itemize}

Let $\widetilde{\frak{c}}$ denote the set of equivalence classes of pairs $(\gamma,w)$ under the above equivalence relation and \[ \widetilde{\mathcal{L}}M=\cup_{\frak{c}\in\pi_0(\mathcal{L}M)}\widetilde{\frak{c}}.\]  We then have a well-defined map \begin{align*} 
\mathcal{A}_H\co \widetilde{\mathcal{L}}M&\to\R \\ 
[\gamma,w]&\mapsto -\int_{D^2}w^*\omega+\int_{0}^{1}H(t,\gamma(t))dt \end{align*}  

The critical points of $\mathcal{A}_H$ are those $[\gamma,w]$ for which $\dot{\gamma}(t)=X_{H}(t,\gamma(t))$, \emph{i.e.} such that $\gamma(t)=\phi_{H}^{t}(\gamma(0))$. 

For any $\frak{c}\in \pi_0(\mathcal{L}M)$ let $N_{\frak{c}}$ denote the nonnegative generator of the subgroup of $\Z$ generated by integers of the form $2\langle c_1(TM),A\rangle$ where $A\in H_{2}^{\frak{c}}$.  

For $\frak{c}\in \pi_0(\mathcal{L}M)$ define \[ \mathcal{O}_{\frak{c},H}=\{\gamma\in \frak{c}|\dot{\gamma}(t)=X_{H}(t,\gamma(t))\},\] so $[\gamma,w]\in Crit(\mathcal{A}_H)\cap \frak{c}$ if and only if $\gamma\in \mathcal{O}_{\frak{c},H}$

We then have well-defined map \begin{align*} 
\mu\co  \mathcal{O}_{\frak{c},H}&\to \Z/N_{\frak{c}}\Z \\ 
\gamma&\mapsto n-\mu_{CZ}\left( (t\mapsto d\phi_{H}^{t})\co T_{\gamma(0)}M\to T_{\gamma(t)}M \right)\end{align*}  where we choose an arbitrary homotopy $w$ from $\gamma_{\frak{c}}$ to $\gamma$, extend the previously-chosen trivialization $\tau_{\frak{c}}$ of $\gamma_{\frak{c}}^*TM$ to a symplectic trivialization of $w^*TM$ and hence of $\gamma^*TM$, and use this trivialization to compute the Conley--Zehnder index as in \cite[Remark 5.4]{RS} of  $t\mapsto d\phi_{H}^{t}$. Two different choices of the homotopy $w$ from $\gamma_{\frak{c}}$ to $\gamma$ will have associated Conley--Zehnder indices which differ by a multiple of $N_{\frak{c}}$, so this prescription yields a well-defined element of $\Z/N_{\frak{c}}\Z$.

Let \[ S_M=\bigcup_{\frak{c}\in\pi_0(\mathcal{L}M)}\{\frak{c}\}\times \Z/N_{\frak{c}}\Z,\] and endow $S_M$ with the obvious $\Z$-action in which $m\in \Z$ sends $(\frak{c},k)\in S_M$ to $(\frak{c},k+m)$.  

 For $\lambda\in \R$,  $\frak{c}\in \pi_0(\mathcal{L}M)$ and $k\in \Z/N_{\frak{c}}\Z$ denote \[ Crit^{\lambda}_{\frak{c},k}(\mathcal{A}_H)=\{[\gamma,w]|\gamma\in \mathcal{O}_{\mathfrak{c},H},\,\mathcal{A}_H([\gamma,w])\leq \lambda,\,\mu(\gamma)=k\}.\]

For $\lambda\in \R$, $(\frak{c},k)\in S_M$ and $K$ a field (with $K$ of characteristic zero if $(M,\omega)$ is not semipositive) let \[ CF^{\lambda}_{\frak{c},k}(H;K)=\left\{\left.\sum_{[\gamma,w]\in Crit^{\lambda}_{\frak{c},k}(\mathcal{A}_H)}a_{[\gamma,w]}[\gamma,w]\right| a_{[\gamma,w]}\in K,\,(\forall C\in \R)(\#\{[\gamma,w]|a_{[\gamma,w]}\neq 0,\,\mathcal{A}_H([\gamma_i,w_i])\geq C\}<\infty)\right\}\] and let \[ CF_{\frak{c},k}(H;K)=\cup_{\lambda\in\R}CF^{\lambda}_{\frak{c},k}(H;K) \quad\mbox{ and }CF_{\frak{c}}(H;K)=\oplus_k CF_{\frak{c},k}(H;K).\]  The standard construction (\cite{FO},\cite{HS},\cite{LT}) of the Floer boundary operator  $\delbar_{J,H}$ (where we use $J$ as a shorthand for the auxiliary data involved in the construction, including a loop of almost complex structures and any necessary virtual cycle machinery)
makes $(CF(H;K)=\oplus CF_{\frak{c},k}(H;K),\partial_{J,H})$ into a $S_M$-graded, $\R$-filtered complex over $K$ in the sense of Definition \ref{rfilt}.  Restricting attention to  $CF_{\frak{c}}(H;K)$ for any given $\frak{c}\in \pi_0(\mathcal{L}M)$ in turn yields a $\Z/N_{\frak{c}}\Z$-graded, $\R$-filtered complex over $K$.

If we let \[ \Gamma_{\frak{c}}=\{g\in \R|(\exists A\in H_{2}^{\frak{c}})(\langle[\omega],A\rangle=g)\},\] each $CF_{\frak{c},k}(H;K)$ is a vector space over the Novikov field $\Lambda^{K,\Gamma_{\frak{c}}}$ of dimension equal to the number of $\gamma\in \mathcal{O}_{\frak{c},H}$ such that $\mu(\gamma)=k$.  Here the element $T^g$ in $\Lambda^{K,\Gamma_{\frak{c}}}$ acts by $[\gamma,w]\mapsto [\gamma,w\#A_g]$ where $A_g\in H_{2}^{\frak{c}}$ has $\langle[\omega],A\rangle=g$ and $\#$ denotes the obvious gluing operation.  For any given $\lambda\in\R$, $CF_{\frak{c},k}^{\lambda}(H;K)$ is a module over the positive part $\Lambda_{\geq 0}^{K,\Gamma_{\frak{c}}}$ of the Novikov field.

The complex $CF_{\frak{c}}(H;K)$ is thus a vector space over the Novikov field $\Lambda^{K,\Gamma_{\frak{c}}}$ of dimension equal to the number of elements of $\mathcal{O}_{\frak{c},H}$.  Our choice of conventions here (in particular the fact that we have not used the first Chern class in the definition of the equivalence relation that is used to construct $\widetilde{\mathcal{L}}M$) is motivated in part by the fact that it results in the complexes $CF_{\frak{c}}$ each being finite-dimensional over a field, as this facilitates application of some of the algebraic results proven in Sections \ref{alg2} and \ref{tprod}. 

Let us rephrase some standard results about the relationships between the Floer complexes associated to different Hamiltonians and different choices of auxiliary data into the language of Section \ref{algintro}.  First, for any continuous $G\co S^1\times M\to \R$ denote \[ \mathcal{E}^+(G)=\int_{0}^{1}\max_M G(t,\cdot)dt\quad \mathcal{E}^-(G)=-\int_{0}^{1}\min_M G(t,\cdot)dt \] so that, in the notation of the introduciton, \[ \osc(G)=\mathcal{E}^+(G)+\mathcal{E}^-(G).\]  Then standard facts (as summarized for example in \cite[Propositions 2.1 and 2.2]{U09} after adjusting for a different sign for the Hamiltonian vector field) show:

\begin{prop}\label{revcont} Let $(H_-,J_-)$ and $(H_+,J_+)$ be two choices of Hamiltonian function together with auxiliary data for which the Floer complex $(CF(H;K)=\oplus C_{\frak{c},k}(H;K),\partial_{J,H})$  as constructed in \cite{HS},\cite{LT} or \cite{FO} is well-defined.  Then there exist:\begin{itemize} \item an $\mathcal{E}^+(H_+-H_-)$-morphism $\Phi\co (CF(H_-;K),\partial_{J_-,H_-})\to (CF(H_+;K),\partial_{J_+,H_+})$ \item an $\mathcal{E}^-(H_+-H_-)$-morphism $\Psi\co (CF(H_+;K),\partial_{J_+,H_+})\to (CF(H_-;K),\partial_{J_-,H_-})$
\item $\osc(H_+-H_-)$-homotopies $\mathcal{K}_{\pm}\co CF(H_{\pm};K)\to CF(H_{\pm};K)$ from $\Phi\circ \Psi$ and $\Psi\circ\Phi$ to the respective identities.\end{itemize}
In particular the Floer complexes  $(CF(H_-;K),\partial_{J_-,H_-})$ and $(CF(H_+;K),\partial_{J_+,H_+})$ are $\osc(H_+-H_-)$-quasiequivalent.
\end{prop}

We also have:

\begin{prop}{\cite[Lemma 3.8]{U09}}\label{oldinvt} Let $H_-$ and $H_+$ be two nondegenerate Hamiltonians which are both normalized (\emph{i.e.}, $\int_M H_{\pm}(t,\cdot)\omega^n=0$ for all $t$) with the property that the paths $\{\phi_{H_{+}}^{t}\}_{t\in [0,1]}$ and 
 $\{\phi_{H_{-}}^{t}\}_{t\in [0,1]}$ are homotopic rel endpoints.  Then, for any auxiliary data $J_{\pm}$ making the Floer complexes well-defined, there is a chain map $\Phi\co (CF(H_-;K),\partial_{J_-,H_-})\to (CF(H_+;K),\partial_{J_+,H_+})$ such that, for each $\lambda,\frak{c},k$, $\Phi$ maps $CF^{\lambda}_{\frak{c},k}(H_-;K)$ isomorphically to $CF^{\lambda}_{\frak{c},k}(H_+;K)$.
\end{prop}

(Strictly speaking, the discussion in \cite{U09} only considered the part of the Floer complex coming from contractible loops; however the proof clearly extends to the noncontractible sectors, provided of course that we use the same basepoints $\gamma_{\frak{c}}$ in each component $\frak{c}\in \pi_0(\mathcal{L}M)$.)

It obviously follows from Proposition \ref{oldinvt} that, under its hypotheses, the boundary depths associated to the Floer complexes of $(H_-,J_-)$ and $(H_+,J_+)$ will coincide.  Thus the boundary depth (or rather depths, if we take grading into account) can be seen as an invariant of an element in the universal cover $\widetilde{Ham}(M,\omega)$ of the Hamiltonian diffeomorphism group, a fact which was exploited in the applications in \cite{U09}.  Crucial for our purposes in this paper is the stronger statement that the boundary depth is actually an invariant of a given element of $Ham(M,\omega)$, at least if one ignores grading.  Indeed we have the following, which follows  from  observations that go back to \cite[(4.3)]{Se}. 

\begin{prop}\label{newinvt}
Consider the Floer complexes $(CF(H_-,J_-),\partial_{J_-,H_-})$, $(CF(H_+,J_+),\partial_{J_+,H_+})$ associated to two normalized Hamiltonians $H_-$ and $H_+$ with the property that the time-one maps $\phi_{H_-}^{1}$ and $\phi_{H_+}^{1}$ are equal.  Then there is a shift-isomorphism  $\Phi\co (CF(H_-;K),\partial_{J_-,H_-})\to (CF(H_+;K),\partial_{J_+,H_+})$.  Moreover, for any component $\frak{c}\in \pi_0(\mathcal{L}M)$, $\Phi$ restricts as a shift-isomorphism from $CF_{\frak{c}}(H_-;K)$ to $CF_{\frak{c}}(H_+;K)$.
\end{prop}

\begin{proof} Changing notation slightly to make the proof more readable, we are to show that if $\{\psi_t|t\in [0,1]\}$ is a loop in $Ham(M,\omega)$, and if $H\co S^1\times M\to\R$ is a nondegenerate normalized Hamiltonian generating the path $\{\phi_t|t\in[0,1]\}$
then the Floer complexes associated to $H$ and to the normalized Hamiltonian $H^{\psi}$ which generates the path $\{\psi_t\circ\phi_t\}$ are shift-isomorphic.  We remark that, by the case of Proposition \ref{oldinvt} in which $H_-=H_+$, up to isomorphism of $S_M$-graded, $\R$-filtered complexes it makes sense to speak of ``the Floer complex associated to a nondegenerate Hamiltonian,'' as different choices of the auxiliary data involved in the construction of the Floer complex will give rise to isomorphic complexes.  In particular, in studying the Floer complex of $H^{\psi}$ we are free to choose any loop of almost complex structures that we like.

Let $G\co S^1\times M\to\R$ denote the normalized Hamiltonian generating the loop $t\mapsto \psi_{t}^{-1}$.  Then the original loop $\psi$ is generated by $\bar{G}(t,m)=-G(t,\psi_{t}^{-1}(m))$, and the Hamiltonian $H^{\psi}$ which generates $\psi_t\circ \phi_t$ is given by the formula \[ H^{\psi}(t,m)=(H-G)(t,\psi_{t}^{-1}(m)).\] 

The assignment to any loop $\gamma\in M$ the loop $\psi\gamma\co t\mapsto \psi_t(\gamma(t))$ gives a map $\psi_*\co \pi_0(\mathcal{L}M)\to \pi_0(\mathcal{L}M)$.  We claim that this map is the identity.  Indeed, it is a standard consequence of the proof of the Arnold conjecture (see \cite[Corollary 9.1.2]{MS}) that where $\frak{c}_0$ is the component of contractible loops we have $\psi_*(\frak{c}_0)=\frak{c}_0$ (for otherwise the Hamiltonian flow of $G$ would have no contractible $1$-periodic orbits).  Once one knows this, if $\gamma\co S^1\to M$ is any loop then the loop $\psi\gamma$ is easily seen to be homotopic to \[ t\mapsto \left\{\begin{array}{ll}\gamma(2t) & 0\leq t\leq 1/2, \\ \psi_{2t-1}(\gamma(0)) & 1/2\leq t\leq 1.\end{array}\right. \]  But the loop $t\mapsto \psi_{2t-1}(\gamma(0)) $ is now known to be contractible\footnote{at least, freely contractible, but of course any freely contractible loop is also contractible with basepoint fixed}, and so $\psi\gamma$ is homotopic to \[ t\mapsto \left\{\begin{array}{ll}\gamma(2t) & 0\leq t\leq 1/2, \\ \gamma(0) & 1/2\leq t\leq 1,\end{array}\right.\] which is obviously homotopic to $\gamma$.

Recall that in each component $\frak{c}$ of $\mathcal{L}M$ we have fixed a basepoint $\gamma_{\frak{c}}$.  By the previous paragraph $\gamma_{\frak{c}}$ is freely homotopic to $\psi\gamma_{\frak{c}}$, so fix a homotopy $w_{\frak{c}}\co [0,1]\times S^1\to M$ from $\gamma_{\frak{c}}$ to $\psi\gamma_{\frak{c}}$. Where $\widetilde{\frak{c}}$ is the covering of $\frak{c}$ introduced earlier, this choice of $w_{\frak{c}}$ induces a map $\psi_*\co \widetilde{\frak{c}}\to \widetilde{\frak{c}}$ which sends an equivalence class $[\gamma,w]$ to $[\psi\gamma,w_{\frak{c}}\#\psi w]$ where $(\psi w)(s,t)=\psi_t(w(s,t))$ and $\#$ denotes the obvious concatenation operation.  We calculate:
\begin{align*} \mathcal{A}_{H^{\psi}}&(\psi_*[\gamma,w])-\mathcal{A}_H([\gamma,w])\\&= -\int_{[0,1]\times S^1}w_{\frak{c}}^{*}\omega-\int_{0}^{1}\int_{0}^{1}\omega\left(\psi_{t*}\frac{\partial w}{\partial s},\psi_{t*}\frac{\partial w}{\partial t}+X_{\bar{G}}(t,\psi_t(w(s,t)))\right)dsdt\\&\quad +\int_{[0,1]\times S^1}w^*\omega+\int_{0}^{1}\left(H^{\psi}(t,\psi_t(\gamma(t)))-H(t,\gamma(t))\right)dt \\
&=-\int_{[0,1]\times S^1}w_{\frak{c}}^*\omega-\int_{0}^{1}\int_{0}^{1}(\psi_{t}^{*}d_M\bar{G})_{w(s,t)}\left(\frac{\partial w}{\partial s}\right)dsdt-\int_{0}^{1} G(t,\gamma(t))dt \\&=-\int_{[0,1]\times S^1}w_{\frak{c}}^*\omega+\int_{0}^{1}\left(G(t,w(1,t))-G(t,w(0,t))\right)dt-\int_{0}^{1} G(t,\gamma(t))dt \\&
=-\int_{[0,1]\times S^1}w_{\frak{c}}^{*}\omega+\int_{0}^{1}\bar{G}(t,\psi_t(\gamma_{\frak{c}}(t)))dt=
\mathcal{A}_{\bar{G}}([\psi\gamma_{\frak{c}},w_{\frak{c}}]).\end{align*}  (Recall that our notation is that $G$ generates $\psi_{t}^{-1}$, and that $\bar{G}(t,\psi_t(m))=-G(t,m)$).

Meanwhile we have also fixed a trivialization $\tau_{\frak{c}}$ of $\gamma_{\frak{c}}^{*}TM$; via the linearizations $\psi_{t*}$ we obtain from this trivialization a trivialization of $(\psi\gamma_{\frak{c}})^{*}TM$.  Comparing this trivialization to the trivialization of $(\psi\gamma_{\frak{c}})^{*}TM$ obtained by extending $\tau_{\frak{c}}$ across the homotopy $w_{\frak{c}}$ we obtain a relative Maslov index $\mu_{\frak{c}}$, and it is easy to see that, in $\Z/N_{\frak{c}}\Z$, \[ \mu(\psi\gamma)-\mu(\gamma)=\mu_{\frak{c}},\] independently of the choice of $\gamma\in\mathcal{O}_{\frak{c},H}$.

Now the map $\psi_*\co\widetilde{\frak{c}}\to \widetilde{\frak{c}}$ clearly takes critical points $[\gamma,w]$ of $\mathcal{A}_H$ (\emph{i.e.}, those $[\gamma,w]$ with $\gamma$ a $1$-periodic orbit of $\phi_t$) bijectively to critical points of $\mathcal{A}_{H^{\psi}}$ (\emph{i.e.}, those $[\gamma,w]$ with $\gamma$ a $1$-periodic orbit of $\psi_t\circ\phi_t$).  So by extending linearly in the obvious way and setting $I_{\frak{c}}=\mathcal{A}_{\bar{G}}([\psi\gamma_{\frak{c}},w_{\frak{c}}])$ we obtain from the calculations above an isomorphism \[ \psi_{*}\co CF_{\frak{c},k}^{\lambda}(H;K)\cong CF_{\frak{c},k+\mu_{\frak{c}}}^{\lambda+I_{\frak{c}}}(H^{\psi};K).\]  The conclusion  that $\psi_*$ is a shift-isomorphism will follow immediately once we conclude that $\psi_*$ is a chain map on the Floer complexes, provided that these are constructed appropriately.  But this was already observed by Seidel in \cite{Se}.  Indeed, if we use the loop $j_t$ of almost complex structures to construct the Floer complex of $H$, so that the matrix elements for the differential are obtained by counting solutions $u\co \R\times S^1\to M$ to \begin{equation}\label{floereqn} \frac{\partial u}{\partial s}+j_t(u(s,t))\left(\frac{\partial u}{\partial s}-X_H(t,u(s,t))\right)=0,\end{equation} then one observes that where $(\psi u)(s,t)=\psi_t(u(s,t))$ and $j_{t}^{\psi}=\psi_{t*}j_t\psi_{t*}^{-1}$, (\ref{floereqn}) is equivalent to \[ \frac{\partial(\psi u)}{\partial s}+j_{t}^{\psi}\left(\frac{\partial(\psi u)}{\partial t}-X_{H^{\psi}}(t,(\psi u)(t))\right)=0.\]  Using this correspondence it is straightforward to see that if we use the loop  $j_{t}^{\psi}$ of almost complex structures (together with appropriately compatible coherent orientations and abstract perturbations, as necessary) to define the differential on the Floer complex $CF(H^{\psi};K)$, then the matrix elements for the latter will coincide with the matrix elements of the original differential on $CF(H;K)$.  Thus $\psi_{*}$ induces an isomorphism of chain complexes, which by our earlier observations is a shift-isomorphism in the sense of Definition \ref{shiftiso}.
\end{proof}

\begin{cor} For each $\frak{c}\in \pi_0(\mathcal{L}M)$ and any nondegenerate $\phi\in Ham(M,\omega)$, the boundary depth of the Floer complex $CF_{\frak{c}}(H;K)$ is independent of the choice of Hamiltonian $H\co S^1\times M\to\R$ such that $\phi_{H}^{1}=\phi$ and of the auxiliary data used in the construction of the Floer boundary operator.
\end{cor}
\begin{proof} If we restrict to normalized Hamiltonians $H$ this follows directly from Propositions \ref{newinvt} and \ref{shiftisoprop}.  Extending to non-normalized $H$ is then trivial, since any Hamiltonian may changed to a normalized Hamiltonian function by adding a function of the $S^1$-variable, and this normalization only affects filtered Floer complex by a uniform shift in the filtration, which does not change the boundary depth.
\end{proof}

Accordingly, for any nondegenerate $\phi\in Ham(M,\omega)$ we may set \[ \beta_{\frak{c}}(\phi;K)=b\left(CF_{\frak{c}}(H;K),\partial_{J,H}\right)\] for any (and hence all) $H$ such that $\phi_{H}^{1}=\phi$.  Now for \emph{any} $\phi=\phi_{H}^{1}\in Ham(M,\omega)$, not necessarily nondegenerate, if we choose any sequence of nondegenerate Hamiltonians $H_n$ such that $H_n\to H$ in $C^0$, then the sequence $\{\beta_{\frak{c}}(H_n;K)\}_{n=1}^{\infty}$ will be a Cauchy sequence by Propositions \ref{quasicont} and \ref{revcont}.  Thus we may define $\beta_{\frak{c}}(H;K)=\lim_{n\to\infty}\beta_{\frak{c}}(\phi_{H_n}^{1};K)$ for any such approximating sequence $H_n$.

We have thus defined, for any $\frak{c}\in \pi_0(\mathcal{L}M)$ and any field $K$ (with characteristic zero if $(M,\omega)$ is not semipositive), a map \[ \beta_{\frak{c}}(\cdot;K)\co Ham(M,\omega)\to [0,\infty].\] Taking the supremum over all $\frak{c}$ results in a map \[ \beta(\cdot;K)=\sup_{\frak{c}\in \pi_0(\mathcal{L}M)} \beta_{\frak{c}}(\cdot;K)\co Ham(M,\omega)\to [0,\infty].\] (In fact, as we will see, these maps never take the value $\infty$.)

\subsection{Properties of the Hamiltonian boundary depth} \label{hambetaprop}
Having defined the boundary depth function, we can now prove most of the main structural results of the introduction.

\begin{proof}[Proof of Theorem \ref{mainpropham} (i) $\mathrm{(\beta(\psi^{-1}\phi\psi;K)=\beta(\phi;K) \mbox{ for }\psi\in Symp(M,\omega))}$] By continuity it suffices to prove this for $\phi$ (and hence $\psi^{-1}\phi\psi$) nondegenerate.  Now if $\phi=\phi_{H}^{1}$, then $\psi^{-1}\phi\psi$ is the time-one map of the Hamiltonian $H_{\psi}(t,m)=H(t,\psi(m))$, and a standard argument (similar to but simpler than that in the proof of Proposition \ref{newinvt}) shows that if the auxiliary data are chosen appropriately then for each $\frak{c}\in \pi_0(\mathcal{L}M)$, $CF_{\frak{c}}(H;K)$ is shift-isomorphic to $CF_{\psi^{-1}_{*}\frak{c}}(H_{\psi};K)$.  Thus the result follows from Proposition \ref{shiftisoprop}.\end{proof}

\begin{proof}[Proof of Theorem \ref{mainpropham} (ii) $\mathrm{(\beta(\phi;K)=\beta(\phi^{-1};K))}$]  By an approximation argument it suffices to consider the case that $\phi$ is nondegenerate.  Choose $H\co S^1\times M\to\R$ so that $\phi_{H}^{1}=\phi$; then $\phi^{-1}$ is generated by the Hamiltonian $\bar{H}(t,m)=-H(t,\phi_{H}^{t}(m))$.  If $\frak{c}\in \pi_0(\mathcal{L}M)$ let $\bar{\frak{c}}\in \pi_0(\mathcal{L}M)$ be the free homotopy class obtained by reversing the orientations of elements of $\frak{c}$.  Then $CF_{\bar{\frak{c}}}(\bar{H};K)$ is (with appropriate choices of auxiliary data) the ``opposite complex'' of $CF_{\frak{c}}(H;K)$ in the sense described in \cite{U10b} (\emph{i.e.}, the Floer trajectories determining the differential of the former are obtained by reversing the $s$-direction of those determining the differential of the latter).  Consequently \cite[Corollary 1.4]{U10b} shows that $\beta_{\bar{\frak{c}}}(\phi^{-1};K)=\beta_{\frak{c}}(\phi;K)$, from which the result follows by taking suprema. \end{proof}

\begin{proof}[Proof of Theorem \ref{mainpropham} (iii) $\mathrm{(|\beta(\phi;K)-\beta(\psi;K)|\leq \|\phi^{-1}\psi\|)}$]  This is almost implicit in the construction of $\beta$ for degenerate elements of $Ham(M,\omega)$. By an approximation argument it again suffices to assume that $\phi$ and $\psi$ are nondegenerate.  Now if $G,H\co S^1\times M\to\R$ have the properties that $\phi_{H}^{1}=\phi$ and $\phi_{G}^{1}=\phi^{-1}\psi$, then $\psi$ will be the time-one map of the Hamiltonian $H\ast G\co S^1\times M\to\R$ given by \[ H\ast G(t,m)=H(t,m)+G(t,(\phi_{H}^{t})^{-1}(m)).\]  Thus by Propositions \ref{quasicont} and \ref{revcont}, \[ |\beta(\phi;K)-\beta(\psi;K)|\leq \osc(H\ast G-H)=\int_{0}^{1}\left(\max_{m\in M}G(t,(\phi_{H}^{t})^{-1}(m))-\min_{m\in M}G(t,(\phi_{H}^{t})^{-1}(m))\right)=\osc(G).\]  So since $G$ may be chosen in such a way that $\osc(G)$ is arbitrarily close to $\|\phi^{-1}\psi\|$, the result follows. \end{proof}

\begin{proof}[Proof of Theorem \ref{mainpropham} (iv) $\mathrm{(\beta(1_M;K)=0)}$]  If the nondegenerate Hamiltonian $H$ is sufficiently $C^2$-close to the identity, then its flow will have no  noncontractible $1$-periodic orbits and so we will have $\beta_{\frak{c}}(H;K)=0$ for all $\frak{c}\in\pi_0(\mathcal{L}M)$ other than the trivial class $\frak{c}_0$.  Taking the limit as $H\to 0$ in $C^2$-norm then proves that $\beta_{\frak{c}}(1_M;K)=0$ for $\frak{c}\neq \frak{c}_0$.

To deal with the case where $\frak{c}=\frak{c}_0$, for any nondegenerate $H$ we can use the PSS map (\cite{PSS}, see also \cite{Lu04}, \cite{OZ}) to set up a $\osc(H)$-quasiequivalence between $CF_{\frak{c}_0}(H;K)$ and the Morse complex $CM_*(f;K)\otimes_K \Lambda^{K,\Gamma_{\frak{c}_0}}$, where the latter has grading reduced modulo $N_{\frak{c}_0}$ and is equipped with the trivial filtration given by setting the filtration level of an element $\sum_{i=1}^{r}\lambda_ip_i$ where $\lambda_i\in \Lambda^{K,\Gamma_{\frak{c}_0}}$ and $p_i\in Crit(f)$ equal to $\max(-\nu(\lambda_i))$. The boundary depth of the Morse complex is easily seen to be zero (indeed this is a baby case of Proposition \ref{qdef} where $\partial=\partial_0\otimes 1$), and so it follows from Proposition \ref{quasicont} that $\beta_{\frak{c}_0}(\phi_{H}^{1};K)\leq \osc(H)$.  Taking the limit of this relation as $H\to 0$ in $C^2$ implies that $\beta_{\frak{c}_0}(1_M;K)=0$, completing the proof.  

(Alternately, one can directly examine the Floer complex of a $C^2$-small Morse function to deduce this result; such a method is effectively used in \cite{Oh09}.)\end{proof}

\begin{proof}[Start of the proof of Theorem \ref{mainpropham} (v) $\mathrm{(\beta(\phi\times\psi;K)\geq \max\{\beta(\phi;K),\beta(\psi;K)\})}$] Again by continuity we can assume that $\phi\co M\to M$ and $\psi\co N\to N$ are nondegenerate; choose Hamiltonians $G\co S^1\times M\to \R$ and $H\co S^1\times N\to\R$ so that $\phi=\phi_{G}^{1}$ and $\psi=\phi_{H}^{1}$.  Then $\phi\times \psi\co M\times N\to M\times N$ is the time-one map of the Hamiltonian $\pi_{M}^{*}G+\pi_{N}^{*}H$ where $\pi_M,\pi_N$ are the projections of $M\times N$ to its factors.  An easy and standard argument shows that, with suitable auxiliary data and after extending coefficients so that all Floer complexes involved are defined over the same Novikov ring, the Floer complex of $\pi_{M}^{*}G+\pi_{N}^{*}H$ is isomorphic, as a filtered complex, to the tensor product of the Floer complexes $CF(G;K)$ and $CF(H;K)$, where the filtration on the tensor product is defined using the prescription in Section \ref{tprod}.  This reduces Theorem \ref{mainpropham} (v) to an algebraic statement about the behavior of the boundary depth with respect to tensor products.  We prove this statement below as Theorem \ref{prodbeta}.  (Of course, the homologies of $CF(G;K)$ and $CF(H;K)$ are both nontrivial, being isomorphic to the homologies of $M$ and $N$, so part (b) of Theorem \ref{prodbeta} applies.  Also, the extension of coefficients does not affect the boundary depth by Remark \ref{coind}.)\end{proof}

Modulo Theorem \ref{prodbeta}, this completes the proof of Theorem \ref{mainpropham}

\begin{proof}[Start of the proof of Theorem \ref{spectrality}]  Like Theorem \ref{mainpropham} (v), Theorem \ref{spectrality} will follow from an essentially purely algebraic result about the boundary depth that we will prove below; in this case the relevant result is Proposition \ref{depthattained}.  Namely, applying Proposition \ref{depthattained} with $A$ equal to the Floer boundary operator on $CF_{\frak{c}}(H;K)$ shows that $\beta_{\frak{c}}(\phi_{H}^{1};K)$  either is zero (when the Floer boundary operator is zero---in particular this holds if there are no periodic orbits in the free homotopy class $\frak{c}$) or is a member of the set $\{s-t|s,t\in \mathcal{S}_{H}^{\frak{c}}\}$.  Now we have assumed $H$ to be nondegenerate, so $\phi_{H}^{1}$ has only finitely many fixed points and so only finitely many homotopy classes $\frak{c}$ are represented by $1$-periodic orbits.  Thus there are only finitely many $\frak{c}$ for which $\beta_{\frak{c}}(\phi_{H}^{1};K)$ is nonzero.  So $\beta(\phi_{H}^{1};K)=\sup_{\frak{c}}\beta_{\frak{c}}(\phi_{H}^{1};K)$ is the supremum over a finite set of numbers from the set $\{0\}\cup\bigcup_{\frak{c}}\{s-t|s,t\in \mathcal{S}_{H}^{\frak{c}}\}$ and therefore belongs to this set.

For the final statement of Theorem \ref{spectrality}, recall that the Floer boundary operator for a nondegenerate Hamiltonian strictly decreases the filtration level: in the notation of Proposition \ref{depthattained} we have $\ell(\partial_{J,H}y)<\ell(y)$ for any nonzero $y$.  Thus by Proposition \ref{depthattained} the only way for the boundary depth to be zero is if the Floer boundary operator is identically zero.  Now the homology of the complex $CF_{\frak{c}}(H;K)$ is zero when $\frak{c}$ is any class other than the trivial one, so this forces there to be no noncontractible $1$-periodic orbits of $\phi_{H}^{t}$.  As for the class $\frak{c}_0$ of contractible periodic orbits, the Floer homology is isomorphic to $H_*(M;\Lambda^{K,\Gamma_{\frak{c}_0}})$ with grading reduced modulo $N_{\frak{c}_0}$,   which has the same dimension as $H_*(M;K)$  since $\Lambda^{K,\Gamma_{\frak{c}_0}}$ is a field extension of $K$. But if the boundary operator is identically zero then the Floer homology would be isomorphic to a
vector space spanned by the fixed points of $\phi$.  Comparing the sums of the dimensions of the homologies then proves that the number of fixed points of $\phi$ is indeed equal to the sum of the $K$-Betti 
numbers of $M$ when $\beta(\phi;K)=0$.
\end{proof}

\begin{remark} \label{indivcham} Of course, we have shown in the course of the above proofs that, for each $\frak{c}\in \pi_0(\mathcal{L}M)$, the numbers $\beta_{\frak{c}}(\phi;K)$ individually obey various similar properties to their supremum $\beta(\phi;K)$.  In particular:
\begin{itemize} \item $|\beta_{\frak{c}}(\phi;K)-\beta_{\frak{c}}(\psi;K)|\leq \|\phi^{-1}\psi\|$;
\item $\beta_{\frak{c}}(1_M;K)=0$;
\item If $\phi=\phi_{H}^{1}$ is nondegenerate (or, more generally, if for all $p\in Fix(\phi_{H}^{1})$ such that $t\mapsto \phi_{H}^{t}(p)$ represents the specific class $\frak{c}$, the linearization $d_p\phi_{H}^{1}\co T_pM\to T_pM$ does not have one as an eigenvalue), then \[ \beta_{\frak{c}}(\phi;K)\in \{0\}\cup\{s-t|s,t\in \mathcal{S}_{H}^{\frak{c}}\}.\]
\end{itemize}
\end{remark}

\subsection{Hamiltonians with large boundary depth} \label{largebeta}

We now begin the process of obtaining the embedding promised in Theorem \ref{hammain}

Suppose that, as in Theorem \ref{hammain}, a closed $2n$-dimensional symplectic manifold $(M,\omega)$ admits a nonconstant autonomous Hamiltonian $H\co M\to\mathbb{R}$ such that the Hamiltonian vector field $X_H$ defined by $\omega(\cdot,X_H)=dH$ has no nonconstant contractible closed orbits.  By Sard's theorem and the compactness of $M$ (and hence of the set of critical values of $H$), there exists a nontrivial closed interval $[a,b]$ contained in the image of $H$ which consists entirely of regular values of $H$.  Since adding a constant to $H$ or rescaling $H$ does not affect the existence of closed orbits, we can and hereinafter do assume that $[a,b]=[0,1]$.

Denote by $C^{\infty}_{c}(0,1)$ the space of smooth functions $f\co (0,1)\to\mathbb{R}$ whose support is compact.  For $f\in C^{\infty}_{c}(0,1)$ define \[ \mm f=\inf\{f(s)|s\mbox{ is a local maximum of }f\}.\]  Since we assume $f$ to be compactly supported we obviously have $\mm f\leq 0$.  (Our convention is that a ``local maximum'' need not be strict; in particular, for example, for a constant function every point is a local maximum.) For $f\in C^{\infty}_{c}((0,1))$ the function $f\circ H$ is \emph{a priori} defined only on $H^{-1}((0,1))$, but it extends smoothly by zero to all of $M$, and we continue to denote by $f\circ H$ this smooth extension.

Here is our key computation of the boundary depth in the Hamiltonian context:

\begin{theorem}\label{mmcomp}  Under the above hypotheses, with respect to any coefficient field $K$ of characteristic zero we have \[ \beta(\phi_{f\circ H}^{1};K)\geq \mm f-\min f.\]
\end{theorem}

\begin{proof}  The theorem is trivial if $\mm f=\min f$, so assume $\mm f>\min f$.  We will make use of the following lemma:

\begin{lemma}\label{fHlemma} Let $y$ be a regular value of $f$ such that $\mm f>y>\min f$. Then in the composition \begin{equation}\label{longcomp} H_{2n-1}(\{f\circ H=\min f\};K)\to H_{2n-1}(\{f\circ H\leq y\};K)\to H_{2n-1}(M;K),\end{equation} the first map is injective but the full composition $H_{2n-1}(\{f\circ H=\min f\};K)\to H_{2n-1}(M;K)$ has nontrivial kernel.
\end{lemma}

\begin{proof} $\{f\leq y\}$ is a compact submanifold with boundary of the interval $(0,1)$, so we have \[ \{f\leq y\}=\cup_{i=1}^{m}[a_i,b_i]\] for some numbers $a_1<b_1<a_2<\cdots<a_m<b_m$.  Now since $(0,1)$ consists entirely of regular values of $H$, for any $i$ and any $c_i\in [a_i,b_i]$ the preimage $H^{-1}[a_i,b_i]$ deformation retracts via the gradient flow of $H$ to $H^{-1}\{c_i\}$.  

Now the assumption that $y<\mm f$ implies that $f$ has no local maxima on $[a_i,b_i]$, so there cannot be two distinct points $c_i,d_i\in [a_i,b_i]$ such that $f(c_i)=f(d_i)=\min f$ (otherwise there would be a local maximum between them).  Let $I\subset \{1,\ldots,m\}$ be the set of those $i$ such that there exists some $c_i\in [a_i,b_i]$ such that $f(c_i)=\min f$.


Now we have, as topological spaces, \[ \{f\circ H\leq y\}=\coprod_{i=1}^{m}H^{-1}[a_i,b_i],\] while \[ \{f\circ H=\min f\}=\coprod_{i\in I}H^{-1}\{c_i\},\] where $H^{-1}[a_i,b_i]$ deformation retracts to $H^{-1}\{c_i\}$ for all $i\in I$.  So since $\{(f\circ H)^{-1}\{\min f\}\}$ is a disjoint union of deformation retracts of a subset of the connected components of 
$\{f\circ H\leq y\}$, the fact that the first map in (\ref{longcomp}) is injective follows immediately. 

Finally, for any $i\in I$, consider $H^{-1}\{c_i\}$.  This is a regular level set of $H$, so it is a closed $(2n-1)$-dimensional manifold which acquires an orientation from the orientation of $M$ together with the coorientation given by $dH$.  Thus $H_{2n-1}(H^{-1}\{c_i\};K)$ is nontrivial and has a distinguished fundamental class $C_i$ (the sum of the fundamental classes of the oriented components), which appears as a nonzero element of $H_{2n-1}(\{f\circ H=\min f\};K)=\oplus_{i\in I}H_{2n-1}(H^{-1}\{c_i\};K)$.  But since \[ H^{-1}\{c_i\}=\partial(\{H\leq c_i\}),\]  $H^{-1}\{c_i\}$ bounds in $M$, and so the fundamental class $C_i$ vanishes upon inclusion into $M$.  This proves that the full composition (\ref{longcomp}) has nontrivial kernel.
\end{proof}

\begin{cor}\label{fHcor} With $y$ as in Lemma \ref{fHlemma}, let $\delta>0$ be such that $\min f+3\delta<y$. If $G\co M\to\mathbb{R}$ is a Morse function with $\|G-f\circ H\|_{C^0}<\delta$ then we have a strict containment \[ \ker\left(H_{2n-1}(\{G\leq \min f+\delta\};K)\to H_{2n-1}(\{G\leq y-\delta\};K)\right)<\ker \left(H_{2n-1}(\{G\leq \min f+\delta\};K)\to H_{2n-1}(M;K)\right).\]

Consequently \[ \beta_{Morse}(G;K)\geq y-\min f-2\delta.\]
\end{cor}

\begin{proof} The assumption on $\delta$ and the  fact that $\|G-f\circ H\|_{C^0}<\delta$ imply that we have inclusions \[ \{f\circ H=\min f\}\subset \{G\leq \min f+\delta\}\subset\{G\leq y-\delta\}\subset \{f\circ H\leq y\}.\]  

If $c\in H_{2n-1}(\{f\circ H=\min f\};K)$ is a nontrivial element of $\ker(H_{2n-1}(\{f\circ H=\min f\};K)\to H_{2n-1}(M;K))$, then $jc\in 
\ker \left(H_{2n-1}(\{G\leq \min f+\delta\};K)\to H_{2n-1}(M;K)\right)$ where $j\co H_{2n-1}(\{f\circ H=\min f\};K)\to H_{2n-1}(\{G\leq \min f+\delta\};K)$ is the inclusion-induced map.  On the other hand the commutativity of the diagram \[ \xymatrix{ H_{2n-1}(\{f\circ H=\min f\};K)\ar[d]^{j}\ar[r] &  H_{2n-1}(\{f\circ H\leq y\};K) \\ 
H_{2n-1}(\{G\leq \min f+\delta\};K)\ar[r] & H_{2n-1}(\{G\leq y-\delta\};K)\ar[u]
}\]
and the fact that the top line is injective show that \[ jc\notin \ker\left(H_{2n-1}(\{G\leq \min f+\delta\};K)\to H_{2n-1}(\{G\leq y-\delta\};K)\right).\]  This proves the strict containment.

Now choose a metric with respect to which the gradient flow of $G$ is Morse-Smale and form the Morse complex $CM(G;K)$ of $G$.  As in Section \ref{prologue} for any $\lambda\in\R$ we have a $\lambda$-filtered complex $CM^{\lambda}(G;K)$, the homology of which is isomorphic to $H_*(\{G\leq \lambda\};K)$; moreover under these isomorphisms the chain complex inclusions $CM^{\lambda}(G;K)\to CM^{\mu}(G;K)$ for $\lambda<\mu$ induce the inclusion-induced maps on singular homology.  Letting $jc\in \ker \left(H_{2n-1}(\{G\leq \min f+\delta\};K)\to H_{2n-1}(M;K)\right)$ be as described above, we can then find a cycle $x\in CM^{\min f+\delta}(G;K)\leq CM(G;K)$  which represents the class $jc\in H_{2n-1}(\{G\leq \min f+\delta\};K)$.  Since $jc$ vanishes under inclusion into $H_{2n-1}(M;K)$ the cycle $x$ must be a boundary in $CM(G;K)$.  But since \[ jc\notin \ker\left(H_{2n-1}(\{G\leq \min f+\delta\};K)\to H_{2n-1}(\{G\leq y-\delta\};K)\right),\] any $y\in CM(G;K)$ with the property that $\partial y=x$ must have $\ell(y)>y-\delta$.  Thus we have found a nonzero element $x\in Im\partial$ such that \[ \inf\{\ell(y)-\ell(x)|\partial y=x\}\geq y-\min f-2\delta,\] and so $\beta_{Morse}(G;K)\geq y-\min f-2\delta$.
\end{proof}

With this preparation, we can now complete the proof of Theorem \ref{mmcomp}, using an approach inspired by the proof of \cite[Theorem 5.1]{Oh05}.

Since the autonomous Hamiltonian $f\circ H\co M\to \R$ has no nonconstant contractible periodic orbits, \cite[Theorem 4.5]{U10a} shows that for any $\delta>0$ there is a smooth function $G\co M\to\R$ such that \begin{itemize} \item[(i)] $\|G-f\circ H\|_{C^0}<\delta$. \item[(ii)] $G$ is a Morse function. \item[(iii)] All contractible periodic orbits of $X_G$ with period at most $1$ are constant. \item[(iv)] At each critical point $p$ of $G$, the Hessian of $G$ has operator norm less than $\pi$.  Here we measure the norm of the Hessian of $G$ using the Riemannian metric induced by an almost complex structure which coincides with the standard complex structure on a Darboux chart around $p$. \end{itemize}

By a further perturbation of $G$, we claim that we may assume that (i)-(iv) still hold and \begin{itemize} \item[(v)] Around each critical point $p$ of $G$ there is a Darboux chart $\psi_p\co U_p\to B^{2n}(\ep)$ such that $\psi_p(p)=0$ and the second-order Taylor approximation around $0$ of $G\circ \psi_{p}^{-1}$ is exact, where $B^{2n}(\ep)$ denotes the standard symplectic $2n$-dimensional ball of some radius $\ep>0$.\end{itemize}

Indeed, for any critical point $p$, let $\psi\co U\to B^{2n}(r)$ be a Darboux chart sending $p$ to $0$, let $\mathcal{H}$ denote the Hessian of $G\circ \psi^{-1}$ at $0$, and choose a compactly supported smooth function $\alpha\co B^{2n}(2)\to [0,1]$ such that $\alpha|_{B^{2n}(1)}=1$.  For any small $\ep>0$ define a perturbation $G_{\ep}$ of $G$ by setting $G_{\ep}=G$ outside $\psi^{-1}(B^{2n}(2\ep))$ and requiring that, on $B^{2n}(2\ep)$, we have \[ (G_{\ep}\circ\psi^{-1})(x)=\alpha(\ep^{-1}x)\left((G\circ\psi^{-1})(0)+\frac{1}{2}\langle \mathcal{H}x,x\rangle\right)+(1-\alpha(\ep^{-1}x))(G\circ\psi^{-1})(x).\]  In particular the second-order Taylor approximation around $0$ of $G_{\ep}\circ\psi^{-1}$ is exact on $B^{2n}(\ep)$, so we now check that, for $\ep>0$ sufficiently small, $G_{\ep}$ still obeys (i)-(iv).  Now it is easy to check that there is a constant $C>0$ such that, for $k=0,1,2$, we  have \begin{equation}\label{3-k} \|G_{\ep}-G\|_{C^k}\leq C\ep^{3-k}.\end{equation} The $k=0$ version of this estimate obviously implies that (i) will hold for $G_{\ep}$ when $\ep$ is small enough.   $G_{\ep}$ clearly has no critical points other than $p$ (which is nondegenerate) in $\psi^{-1}(B^{2n}(\ep))$, and the fact that (by the invertibility of the Hessian $\mathcal{H}$) $\|\nabla G\|$ is bounded below in $\psi^{-1}(B^{2n}(2\ep)\setminus B^{2n}(\ep))$ by a constant multiple of $\ep$ together with the $k=1$ version of (\ref{3-k}) implies that $G_{\ep}$ has no critical points in $\psi^{-1}(B^{2n}(2\ep)\setminus B^{2n}(\ep))$ if $\ep$ is small enough.  So since $G_{\ep}$ coincides with $G$ outside $\psi^{-1}(B^{2n}(2\ep))$, for $\ep$ small enough replacing $G$ by $G_{\ep}$ will not introduce any new critical points, while keeping the existing ones nondegenerate, so $G_{\ep}$ obeys (ii) for small enough $\ep$.  Replacing $G$ by $G_{\ep}$ also does not affect the Hessians at the critical points, so (iv) is still satisfied.  As for (iii), there is a constant $R>0$ such that any time-$1$ trajectory of $X_{G_{\ep}}$ passing through $\psi^{-1}(B^{2n}(2\ep))$ is contained in $\psi^{-1}(B^{2n}(R\ep))$ for $\ep$ sufficiently small.  Since $\|\nabla X_{G_{\ep}}\|$ is (using the $k=2$ version of (\ref{3-k})) uniformly bounded by $\pi$ plus a constant multiple of $\ep$ in this region, the Yorke estimate \cite{Y} implies that there can be no nonconstant closed trajectories of $X_{G_{\ep}}$ having period at most one which pass through $\psi^{-1}(B^{2n}(2\ep))$ when $\ep$ is sufficiently small.  Since the trajectories of $X_{G_{\ep}}$ which do not pass through 
$\psi^{-1}(B^{2n}(2\ep))$ coincide with trajectories of $X_G$, it follows that condition (iii) is also preserved when we replace $G$ by $G_{\ep}$ for sufficiently small $\ep$.

Repeating this process at each of the finitely many critical points of $G$ gives a smooth function (still denoted $G$) which now obeys each of the properties (i)-(v).  Shrinking the Darboux neighborhoods $U_p$ if necessary, we may assume that the intersections $\overline{U_p}\cap \overline{U_q}$ are empty for distinct $p,q\in Crit(G)$.  Then for $\omega$-compatible almost complex structures $J$ which are generic among those which coincide with the standard almost complex structure on each $U_p$, the gradient flow of $G$ with respect to the associated metric $g_J$ will be Morse-Smale (see, \emph{e.g.}, \cite[Theorem 8.1]{SZ}, noting that any gradient trajectory must pass through the open set $M\setminus \cup_{p\in Crit(G)}\overline{U_p}$, and so the argument given in the proof of \cite[Theorem 8.1]{SZ} shows that perturbations of $J$ supported in this open set are sufficient to achieve the Morse--Smale condition). Given such an almost complex structure $J$, we may form the Morse complex $CM(G;K)$.  For any $\lambda\in (0,1]$, the negative gradient flow of $\lambda G$ with respect to $g_J$ will of course also be Morse--Smale (as the unstable and stable manifolds are independent of $\lambda$), and so we have a Morse complex $CM(\lambda G;K)$.  

The Morse--Smale condition ensures that there are no nonconstant negative gradient-flow trajectories $\gamma\co \R\to M$ for $G$ with $\gamma(s)\to p_{\pm}$ as $s\to \pm\infty$ with $ind(p_-)-ind(p_+)<1$ where $ind(p)$ denotes the Morse index of the critical point $p$, and, modulo time translation, there are finitely many such trajectories $\gamma$ with $ind(p_-)-ind(p_+)=1$; let $\mathcal{O}_G$ denote the set of such trajectories.  For $\lambda\in (0,1]$ the only negative gradient flow trajectories for $\lambda G$ which connect critical points whose Morse indices differ by at most one will be those given by the formula $\gamma^{\lambda}(s)=\gamma(\lambda s)$ where $\gamma\in\mathcal{O}_G$.  Considering instead the construction of the Floer complex associated to the $t$-independent almost complex structure $J$ and the Hamiltonian $\lambda G$, the nonconstant $t$-\emph{independent} solutions  to the Floer equation $\frac{\partial u}{\partial s}+J\left(\frac{\partial u}{\partial t}-X_{\lambda G}\right)=0$ which connect contractible periodic orbits with Conley--Zehnder indices differing by at most $1$ are precisely those maps of the form $u_{\gamma^{\lambda}}(s,t)=\gamma^{\lambda}(s)$ where $\gamma\in\mathcal{O}_G$ (the fact that the Conley--Zehnder and Morse indices correspond is a consequence of assumption (iv) on the Hessian of $G$ at its critical points).  In Theorem \ref{appthm} in the Appendix we show that, given $\gamma\in\mathcal{O}_G$,  the linearization of the Floer equation is surjective at $u_{\gamma^{\lambda}}$ for all but finitely many $\lambda\in (0,1]$.  Of course, this property of $u_{\gamma^{\lambda}}$ is unaffected by translations of the domain $\R\times S^1$ of $u$ in the $s$-variable, so since there are only finitely many time-translation-equivalence classes of trajectories $\gamma\in \mathcal{O}_G$ it follows that, for all but finitely many values of $\lambda\in (0,1]$, the linearization of the Floer equation associated to $J$ and $\lambda G$ is surjective at every $t$-independent solution having index one.  Let $\mathcal{E}$ denote the finite set of values $\lambda$ for which this surjectivity property fails.

We consider the boundary depths $\beta_{\frak{c}_0}(\phi_{\lambda G}^{1};K)$ as $\lambda$ varies through the interval $(0,1]$.  The argument given 
 in \cite[Section 22]{FO} proves that, by introducing $S^1$-equivariant abstract perturbations which are supported away from the fixed locus of the $S^1$ action given by $t$-translation, for each $\lambda\in (0,1]\setminus \mathcal{E}$ the Floer complex of $\lambda G$ in the contractible sector $\frak{c}_0$ may be constructed in such a way that the only contributions to its boundary operator come from the $t$-independent Floer trajectories $u_{\gamma^{\lambda}}$. (The idea of the argument is that when $J$ and $H$ are independent of $t$, solutions to the Floer equation which depend on both $s$ and $t$ occur in two-dimensional families due to reparametrizations of $\R\times S^1$, and so will not appear in transversely-cut-out moduli spaces of expected dimension one.)
  In \cite{FO} it is assumed that the Hamiltonian is a $C^2$-small Morse function, but for the purposes of the conclusion that only $t$-independent trajectories contribute to the boundary operator there are only three respects in which this assumption is used there (see \cite[pp. 1035, 1038]{FO}): it ensures that our condition (iii) holds; it ensures that the Conley--Zehnder index at each constant orbit of the Hamiltonian vector field coincides up to an additive constant with its Morse index, which in our case follows from condition (iv); and it guarantees the surjectivity of the linearizations of the Floer equation at the index-one $t$-independent solutions, which we have just arranged to hold in our case as well for  $\lambda\in (0,1]\setminus \mathcal{E}$.  
  
  Since $\int_{\R\times S^1} u_{\gamma^{\lambda}}^{*}\omega=0$, the  Floer trajectories $u_{\gamma^{\lambda}}$ for $\lambda G$ all connect pairs of elements of $Crit(\mathcal{A}_{\lambda G})$ the difference of whose actions belongs to the set $\{\lambda G(p)-\lambda G(q)|p,q\in Crit(G)\}$.  If $\lambda\in (0,1]\setminus\mathcal{E}$, so that the $u_{\gamma^{\lambda}}$ are the only trajectories which contribute to the boundary operator for the Floer chain complex $CF_{\frak{c}_0}(\lambda G;J)$ under appropriate perturbations as in the previous paragraph,   there is then a chain complex $D_*$ over $K$, spanned by the critical points of $\lambda G$, such that $CF_{\frak{c}_0}(\lambda G;J)= D_*\otimes \Lambda^{K,\Gamma_{\frak{c}_0}}$ (with the Floer boundary operator just given by coefficient extension from the boundary operator for $D_*$; the matrix elements for the boundary operator for $D_*$ are, like those of the Morse boundary operator, obtained by counting trajectories $u_{\gamma^{\lambda}}$ connecting two critical points, but the signs with which these contribute to the boundary operator for $D_*$ might in principle differ from the corresponding signs in the Morse complex).  It therefore follows from Remark \ref{coind} that for $\lambda\in (0,1]\setminus \mathcal{E}$, the boundary depth $\beta_{\frak{c}_0}(\phi_{\lambda G};K)$ coincides with the boundary depth of the chain complex $D_*$ over $K$, which by Proposition \ref{depthattained} belongs to the set $\{\lambda G(p)-\lambda G(q)|p,q\in Crit(G)\}$.  Thus \[ \frac{1}{\lambda}\beta_{\frak{c}_0}(\phi_{\lambda G}^{1};K)\in \{G(p)-G(q)|p,q\in Crit(G)\} \] for all $\lambda\in (0,1]\setminus \mathcal{E}$, and hence by continuity  for all $\lambda\in (0,1]$ since both the set $\mathcal{E}$ and the set $\{G(p)-G(q)|p,q\in Crit(G)\}$ are finite.  
    
Now for $0<\lambda \ll 1$, the Hamiltonian $\lambda G$ will be $C^2$-small enough that, by \cite[Section 22]{FO}, the Floer complex of $\lambda G$ coincides with its Morse complex $CM(\lambda G;K)\otimes \Lambda^{K,\Gamma_{\frak{c}_0}}$.\footnote{For this statement one does seem to need $\lambda G$ to be $C^2$-small, as once $\lambda$ is larger than the smallest element of $\mathcal{E}$ the possibility arises that a $t$-independent Floer trajectory might contribute to the Floer and Morse boundary operators with opposite sign, as can be seen by examining the argument on \cite[p. 1039]{FO} and the proof of Theorem \ref{appthm}.} As follows from   Remark \ref{coind}, the boundary depth of the Morse complex $CM(G;K)$ is unaffected by the coefficient extension to the Novikov ring.  Thus for all sufficiently small $\lambda\in (0,1]$ we have $\beta_{\frak{c}_0}(\phi_{\lambda G}^{1};K)=\beta_{Morse}(CM(\lambda G;K))$.  But then the functions \[ \lambda\mapsto \frac{1}{\lambda}\beta_{\frak{c}_0}(\phi_{\lambda G}^{1};K)\quad\mbox{and}\quad \lambda\mapsto \frac{1}{\lambda}\beta_{Morse}(CM(\lambda G;K)) \] are both continuous functions from $(0,1]$ to the finite set $\{G(p)-G(q)|p,q\in Crit(G)\}$, and so the fact that they coincide for all sufficiently small $\lambda$ implies that in fact they coincide for all $\lambda\in (0,1]$, and in particular for $\lambda=1$.

  Thus by Corollary \ref{fHcor}, \[ \beta(\phi_{G}^{1};K)\geq \beta_{\frak{c}_0}(\phi_{G}^{1};K)=\beta_{Morse}(CM(G;K))\geq y-\min f-2\delta.\]

So since $\osc(f\circ H-G)\leq 2\|f\circ H-G\|_{C^0}\leq 2\delta$ it then follows from Theorem \ref{mainpropham} (iii) that \[ \beta(\phi_{f\circ H}^{1};K)\geq y-\min f-4\delta.\]  Since this holds for any $y<\mm f$ and any $\delta>0$ we have completed the proof of Theorem \ref{mmcomp}.
\end{proof}

\begin{remark}
Of course, the above proof makes substantial use of the Kuranishi structure machinery from \cite{FO}.  If one prefers to do without this, and is willing to impose the somewhat strong topological assumption that the minimal Chern number of $(M,\omega)$ is at least equal to the complex dimension $n$ (this includes the case where $c_1(TM)$ vanishes on $\pi_2(M)$, in which case the minimal Chern number is considered to be $\infty$), then one can instead make use of results from \cite{FHS} to obtain Theorem \ref{mmcomp} with the field $K$ assumed to have characteristic $2$ rather than $0$.\footnote{The results that we use from \cite{FHS} require $n\geq 2$; if instead $n=1$ and $M$ is not $S^2$ (and so $M$ is aspherical) then one can obtain Theorem \ref{mmcomp} from a continuity argument similar to that used at the end of the above proof.}   Namely, first modify the function  $G$ from the proof of Theorem \ref{mmcomp} so that its Hessian at each of its critical points belongs to the set $\mathcal{S}_{reg}$ of \cite[Theorem 6.1]{FHS} (which due to the density of $\mathcal{S}_{reg}$ can easily be done in a way compatible with our conditions (i)-(v)). Then take an almost complex structure $J_0$ which is standard near the critical points and with respect to which the gradient flow of $G$ is Morse--Smale, and use Theorem \ref{appthm} to slightly rescale $G$, preserving conditions (i)-(v), so that all $t$-independent solutions to the Floer equation determined by $J_0$ with index at most $1$ are cut out transversely.  Then by \cite[Theorem 7.4]{FHS}, a generic small perturbation $J$ of $J_0$ will have the property that all $t$-\emph{dependent} solutions of the Floer equations $\frac{\partial u}{\partial s}+J\left(\frac{\partial u}{\partial t}-X_{\frac{1}{m}G}\right)=0$ for $m\in\Z_{+}$ that are not multiply-covered are cut out transversely; moreover as long as the perturbation $J$ is close enough to $J_0$ the implicit function theorem shows that the transversality property for the $t$-independent solutions from the previous sentence will still hold.  

We claim that in this situation the Floer complex (in the contractible sector $\frak{c}_0$) associated to $G$ and $J$ is well-defined and identical to the Morse complex.  To see this, note that any finite-energy solution $u\co \R\times S^1\to M$ to the Floer equation which is asymptotic to contractible (and hence constant) $1$-periodic orbits $p_{\pm}$ of $X_G$ as $s\to \pm\infty$ extends continuously to a map $\bar{u}\co S^2\to M$ when we identify $\R\times S^1$ with the complement of the north and south poles of $S^2$; write $c_1(u)$ for the Chern number of this sphere.  Now the index of the solution $u$ is given by \[ I(u)=ind(p_-)-ind(p_+)+2c_1(u).\]  If $u$ depends nontrivially on $t$, then there is $m\geq 1$ and a solution $v\co \R\times S^1\to M$ to the Floer equation associated to $\frac{1}{m}G$ such that $u(s,t)=v(ms,mt)$ for all $(s,t)\in \R\times S^1$ and such that $v$ is not multiply-covered.  In particular the solution $v$ is cut out transversely by \cite[Theorem 7.4]{FHS}, and so in view of the translation and rotation actions $v$ must have index $I(v)=ind(p_-)-ind(p_+)+2c_1(v)\geq 2$.  So since $c_1(u)=mc_1(v)$, if the original solution $u$ had $I(u)\leq 1$ it would need to hold that $c_1(v)<0$.  But then the assumption on the minimal Chern number gives $2c_1(v)\leq -2n$,   and so we would need to have $ind(p_-)-ind(p_+)\geq 2n+2$, which is impossible since the Morse index only takes values from $0$ to $2n$.  This proves that the only solutions to the Floer equation with index at most one are the $t$-independent ones, which coincide with the Morse trajectories and which we have arranged to be cut out transversely.  Thus the Floer boundary operator receives contributions precisely from the Morse trajectories that determine the boundary operator for the Morse complex; since we are working over a field of characteristic two these contributions are equal and the Floer boundary operator equals the Morse boundary operator (in particular it squares to zero). So the Floer complex associated to $G$ and $J$ is indeed identical to the Morse complex, and so $\beta_{\frak{c}_0}(\phi_{G}^{1};K)=\beta_{Morse}(G;K)$ and we can proceed just as in the last two paragraphs of the proof of Theorem \ref{mmcomp}.

\end{remark}

\begin{proof}[Proof of Theorem \ref{hammain}]  As above let $H\co M\to \R$ be an autonomous Hamiltonian such that all contractible closed orbits of $X_H$ are constant, and such that the interval $[0,1]$ is contained in the image of $H$ and consists entirely of regular values of $H$.

Fix once and for all a smooth function $g\co \R\to [0,1]$ such that \begin{itemize} \item $\max g=1$ \item  the support of $g$ is contained in the open interval $(0,1)$ \item The only local minima of $g$ are at points where $g(s)=0$.\end{itemize}
Now for $v=(v_i)_{i=1}^{\infty}\in \R^{\infty}$ define \begin{align*} 
f_v&\co \R\to \R \\ f_v(s)&=\sum_{i=1}^{\infty}v_ig\left(2^i(s-(1-2^{1-i}))\right) \end{align*} (In other words, the restriction of $f_v$ to the interval $I_i=[1-2^{1-i},1-2^{-i}]$ is equal to $v_i$ times the composition of $g$ with an affine map which takes $I_i$ to $[0,1]$).  The embedding $\Phi\co \R^{\infty}\to Ham(M\omega)$ in Theorem \ref{hammain} will then be given by \[ \Phi(v)=\phi_{f_v\circ H}^{1}.\]  Now the various $f_v\circ H$ all Poisson commute with each other (their Hamiltonian vector fields are all obtained by multiplying $X_H$ by a function), so since $f_{v+w}\circ H=f_v\circ H+f_w\circ H$ it is clear that $\Phi$ is a homomorphism.  Using this and the biinvariance of the Hofer metric we have \[ d(\Phi(v),\Phi(w))=\|\phi_{f_{v-w}\circ H}^{1}\|\leq \osc(v-w).\]  (The last inequality uses that $\max g=1$.)

On the other hand by Theorems \ref{mainpropham} (i) and \ref{mmcomp} we have \[ d(\Phi(v),\Phi(w))\geq \beta(\phi_{f_{v-w}\circ H}^{1};K)\geq \mm f_{v-w}\circ H-\min f_{v-w}\circ H\geq -\min f_{v-w}\circ H,\] where the last inequality uses that the properties of $g$ ensure that no $f_u\circ H$ has a negative local maximum.  At the same time Theorem \ref{mainpropham} (ii) shows that \[ \beta(\phi_{f_{v-w}\circ H}^{1};K)=\beta(\phi_{f_{w-v}\circ H}^{1};K)\geq -\min f_{w-v}\circ H.\]  Thus \[ 
d(\Phi(v),\Phi(w))\geq \max\{-\min f_{v-w}\circ H, -\min f_{w-v}\circ H\}=\max\{\max_i(w_i-v_i),\max_i(v_i-w_i)\}=\|v-w\|_{\ell_{\infty}}.\]
\end{proof}

\subsection{An ``energy-capacity inequality''}

If $U$ is an open subset of a symplectic manifold $(M,\omega)$, we denote by $Ham^c(U)$ the group of diffeomorphisms which may be generated by a function $G\co S^1\times M\to\mathbb{R}$ such that the support of $G$ is a compact subset of $S^1\times U$.

\begin{prop} \label{dispbeta} Let $(M,\omega)$ be a closed symplectic manifold, let $U\subset M$ be open, let $\psi\in Ham^c(U)$, and suppose that $\phi(\bar{U})\cap \bar{U}=\varnothing$ where $\phi\in Ham(M,\omega)$.  Then, for all $\frak{c}\in \pi_0(\mathcal{L}M)$ and all fields $K$ 
\[ \beta_{\frak{c}}(\psi\circ \phi;K)=\beta_{\frak{c}}(\phi;K).\]   In particular $\beta(\psi\circ\phi;K)=\beta(\phi;K)$.

\end{prop}

\begin{remark} A statement essentially equivalent to Proposition \ref{dispbeta} was proven in \cite[Lemma 3.6]{U09} under the additional assumption that the Hamiltonian generating $\psi$ could be taken to be either everywhere nonnegative or everywhere nonpositive.
\end{remark}

\begin{proof}  The condition that $\phi(\bar{U})\cap \bar{U}=\varnothing$ is an open one on $\phi$, so by an approximation argument and the continuity of $\beta$ we may assume that $\phi$ is nondegenerate.  

By means of appropriate time reparametrizations, let $H\co S^1\times M\to \R$ be a Hamiltonian generating $\phi$ having support in $(0,1/2)\times M$, and let $G\co S^1\to M$ be a Hamiltonian generating $\psi$ having support in $(1/2,1)\times U$.  For $0\leq s\leq 1$ define \[ L_s(t,m)=\left\{\begin{array}{ll} H(t,m) & 0\leq t\leq 1/2 \\ sG(t,m) & 1/2\leq t\leq 1\end{array}\right.\]

Then where $\psi_s\in Ham^c(U)$ is the Hamiltonian diffeomorphism generated by $sG$, the Hamiltonian $L_s\co S^1\times M\to\R$ generates $\psi_s\circ\phi$.

Taking inspiration from the Ostrover trick \cite{Os03}, we observe that all fixed points of $\phi$ are contained in $M\setminus\bar{U}$ and that $\psi_s(\phi(m))=m\Leftrightarrow \phi(m)=m$.  In particular each $\psi_s\circ\phi$ coincides with $\phi$ on a neighborhood of their common fixed point set, and so the nondegeneracy of $\phi$ implies that of $\psi_s\circ \phi$ for all $s$.  Since $\psi_0$ is the identity and $\psi_1=\psi$, it suffices to show that the functions $s\mapsto \beta_{\frak{c}}(\psi_s\circ\phi;K)$ are constant.


  Now the  maps $\gamma\co S^1\to M$ such that $\dot{\gamma}(t)=X_{L_s}(t,\gamma(t))$ are precisely those of the form $\gamma_p(t)=\phi_{L_s}^{t}(p)$ where $p\in Fix(\psi_s\circ \phi)$.  But as noted earlier, if $p\in Fix(\psi_s\circ \phi)$ then $p\in Fix(\phi)$ and $p\notin\bar{U}$, and so $\phi_{L_s}^{t}(p)=\phi_{H}^{t}(p)$ for all $t$.  Thus the orbits $\gamma_p$ are independent of $s$; moreover since $H$ is supported in $(0,1/2)\times M$ and $G$ is supported in $(1/2,1)\times U$ while $p\notin U$ we have \[ L_s(t,\gamma_p(t))=H(t,\gamma_p(t)) \]  
for all $t$.

Now by Theorem \ref{spectrality} and Remark \ref{indivcham}, for all $s$ and $\frak{c}$ we have \[ \beta_{\frak{c}}(\phi_{L_s}^{1};K)\in \{0\}\cup \{s-t|s,t\in \mathcal{S}_{L_s}^{\frak{c}}\}.\]  Recall from the introduction that $\mathcal{S}_{L_s}^{\frak{c}}$ is defined as follows: to define the filtrations on the Floer complexes we have chosen a basepoint $\gamma_{\frak{c}}$ for $c$; then $\mathcal{S}_{L_s}^{\frak{c}}$ consists of the values \[ \mathcal{A}_{L_s}([\gamma_p,u])=-\int_{[0,1]\times S^1}u^*\omega+\int_{0}^{1}L_s(t,\gamma_p(t))dt \] where $p$ varies over fixed points such that $\gamma_p\in\frak{c}$ and $u\co [0,1]\times S^1\to M$ varies over homotopies from $\gamma_{\frak{c}}$ to $\gamma_p$.  For every $p$ let us fix a homotopy $u_p$ from $\gamma_{\frak{c}}$ to $\gamma_p$; for any other homotopy $u$ we will have \[ \mathcal{A}_{L_s}([\gamma_p,u])-\mathcal{A}_{L_s}([\gamma_p,u_p])\in \Gamma_{\frak{c}}\] where the countable group $\Gamma_{\frak{c}}$ was introduced near the start of Section \ref{hamsect}.  But \[ \mathcal{A}_{L_s}([\gamma_p,u_p])=-\int_{[0,1]\times S^1}u_{p}^{*}\omega+\int_{0}^{1}L_s(t,\gamma_p(t))=\mathcal{A}_H([\gamma_p,u_p])\] is independent of $s$ by our earlier remarks. 

Consequently the set $\{s-t|s,t\in\mathcal{S}_{L_s}^{\frak{c}}\}$ is equal to \[ \{\mathcal{A}_H([\gamma_p,u_p])-\mathcal{A}_H([\gamma_q,u_q])+g|\gamma_p,\gamma_q\in\frak{c},\,g\in \Gamma_{\frak{c}}\};\] this is a countable set which is independent of $s$.  So $s\mapsto \beta_k(L_s;K)$ is a continuous function to a countable set and therefore is constant.
\end{proof}

\begin{cor}\label{encap} For a closed symplectic manifold $(M,\omega)$ and an open set $U\subset M$ let \[ c_{\beta}(U;M)=\sup\{\beta(\phi;\mathbb{Q})|\phi\in Ham^c(U)\}.\]  Then \[ c_{\beta}(U;M)\leq 2e(U;M) \] where $e(U;M)$ is the displacement energy: $e(U;M)=\inf\{\|\phi\| | \phi(\bar{U})\cap \bar{U}=\varnothing\}$. 
\end{cor}

\begin{proof} If $\psi\in Ham^c(U)$ and $\phi(\bar{U})\cap\bar{U}=\varnothing$ we have by Proposition \ref{dispbeta} and Theorem \ref{mainpropham}\[ |\beta(\psi;\mathbb{Q})-\beta(\phi;\mathbb{Q})|=|\beta(\psi;\mathbb{Q})-\beta(\psi\circ\phi;\mathbb{Q})|\leq \|\phi\|. \]  Hence \begin{equation}\label{dispest}
\beta(\psi;\mathbb{Q})\leq \beta(\phi;\mathbb{Q})+\|\phi\|\leq 2\|\phi\|.
\end{equation} 
\end{proof}

\section{Lagrangian Floer theory} \label{lagsect}

Let $L$ be a closed connected Lagrangian submanifold of a tame $2n$-dimensional symplectic manifold $(M,\omega)$.  We will always assume that the Floer homology of $L$ can be defined; this requires either that $L$ be monotone with minimal Maslov number $\mu_L\geq 2$,\footnote{Strictly speaking, for the discussion below we need to use a slightly different convention than usual for the definition of $\mu_L$.  In general the Maslov index induces a homomorphism $\mu\co H_2(M,L;\Z)\to \Z$.  The most common definition has $\mu_L$ equal to the positive generator of the group generated by $\mu(A)$ for $A$ in the image of the Hurewicz map $\pi_2(M,L)\to H_2(M,L;\Z)$ (and $\infty$ if this group is trivial).  Since we consider Floer chain groups corresponding to all elements $\frak{c}\in \pi_0(\mathcal{P}(L,L))$ rather than to just the trivial class $\frak{c}_0$, we need to consider somewhat more general classes $A\in \pi_2(M,L;\Z)$, namely the relative homology classes of cylinders having both boundary components on $L$.  Alternately, we could use just the trivial class $\frak{c}_0$ in the definition of the Floer complex, in which case the version of $\mu_L$ based on $\pi_2(M,L)$ would be appropriate, but then we would need to assume that the resulting Floer homology is nontrivial in order to obtain a version of Proposition \ref{newinvtlag} and hence to make the boundary depth well-defined as a function on $\mathcal{L}(L)$.} or else that $L$ be oriented, relatively spin and weakly unobstructed after bulk deformation in the sense of \cite[Sections 3.6, 3.8.5]{FOOO09}.  The fields $K$ over which Floer theory can be defined depend somewhat sensitively on the hypotheses we put on $L$: if $L$ is not relatively spin but is monotone with $\mu_L\geq 2$ then we need to work over a field of characteristic $2$; on the other hand, at least if the ambient manifold is not spherically positive, then the construction of Floer theory for relatively spin non-monotone Lagrangians as in \cite{FOOO09} requires one to work over a field of characteristic $0$.  Throughout this discussion we will assume that $K$ is a field satisfying the above requirements.

In the case in which one uses the machinery of \cite{FOOO09}, one needs to specify a relative spin structure on $L$ as well as a suitable bounding cochain in order to develop the theory; we will use the notation $\hat{L}$ to denote $L$ equipped with whatever such extra structure may be required in the case at hand (note that that the Floer homology, to be denoted by $HF(\hat{L},\hat{L})$, may depend on the extra structure).

Given this input, let \[ \mathcal{P}(L,L)=\{\gamma\in C^1([0,1], M)|\gamma(0),\gamma(1)\in L\}.\]  Then $\pi_0(\mathcal{P}(L,L))$ contains a distinguished component $\frak{c}_0$ which contains all constant paths at points of $L$ (and, more generally, all paths contained entirely within $L$).  As in the Hamiltonian case, let us choose for each $\frak{c}\in \pi_0(\mathcal{P}(L,L))$ a basepoint $\gamma_{\frak{c}}$, and also a symplectic trivialization $\tau_{\frak{c}}$ of $\gamma_{\frak{c}}^{*}TM$ under which, for $i=0,1$, $T_{\gamma(i)}L$ is identified with $\{i\}\times \R^n$.  Consider pairs $(\gamma,v)$ where $\gamma\in\frak{c}$ and $v\co [0,1]\times [0,1]\to M$ is $C^1$ and  obeys $v(s,i)\in L$ for $i=0,1$,  $v(0,t)=\gamma_{\frak{c}}(t)$ and $v(1,t)=\gamma(t)$.  Up to homotopy there is a unique symplectic trivialization of $v^{*}TM$ which extends $\tau_{\frak{c}}$ and, for $i=0,1$, identifies each $T_{v(s,i)}L$ with $\{(s,i)\}\times \R^n$; in particular $\tau_{\frak{c}}$ and $v$ induce a symplectic trivialization of $\gamma^{*}TM$ which identifies $T_{\gamma(i)}L$ with $\{i\}\times \R^n$.  Declare two such pairs $(\gamma,v)$ and $(\gamma',v')$ to be equivalent if and only if: \begin{itemize}\item[(i)] $\gamma=\gamma'$, and  \item[(ii)] $\int_{[0,1]^2}v^*\omega=\int_{[0,1]^2}v'^*\omega$.
\end{itemize}

For $\frak{c}\in \pi_0(\mathcal{P}(L,L))$ we let $\widetilde{\frak{c}}$ be the set of equivalence classes of such pairs $(\gamma,v)$ with $\gamma\in \frak{c}$ under the above equivalence relation, and define \[ \widetilde{\mathcal{P}}(L,L)=\cup_{\frak{c}\in\pi_0(\mathcal{P}(L,L))}\widetilde{\frak{c}}.\]  We then have a well-defined function \begin{align*} \mathcal{A}_H\co \widetilde{\mathcal{P}}(L,L)&\to\R \\ [\gamma,v]&\mapsto -\int_{[0,1]^{2}}v^*\omega+\int_{0}^{1}H(t,\gamma(t))dt. \end{align*}

The critical points of $\mathcal{A}_H$ are, as in the Hamiltonian case, those $[\gamma,v]$ with $\dot{\gamma}(t)=X_{H}(t,\gamma(t))$.  They thus correspond to time-one flowlines $\gamma$ of the (time-dependent) vector field $X_H$ with the property that $\gamma(0)\in L\cap (\phi_{H}^{1})^{-1}(L)$.  For $\frak{c}\in \pi_0(\mathcal{P}(L,L))$ let $\mathcal{O}_{\frak{c},L,H}$ denote the set of such flowlines $\gamma$ which belong to the homotopy class $\frak{c}$. Assume that the pair $(L,H)$ is nondegenerate in the sense that $L$ is transverse to $(\phi_{H}^{1})^{-1}(L)$; in particular $\cup_{\frak{c}}\mathcal{O}_{\frak{c},L,H}$ will then be finite.  

For $\frak{c}\in \pi_0(\mathcal{P}(L,L))$, a loop $\eta\co S^1\to\frak{c}$ gives rise in obvious fashion to a cylinder $u_{\eta}\co S^1\times [0,1]\to M$ whose boundary is mapped to $L$.  We then obtain a Maslov index $\mu_{\eta}$ by choosing an arbitrary symplectic trivialization of $u_{\eta}^{*}TM$ and taking the difference of the Maslov indices of the loops of Lagrangian subspaces given by $u|_{S^1\times\{i\}}^{*}TL$ in terms of the trivialization (of course the difference is independent of the trivialization).  Let $N_{\frak{c}}$ be the nonnegative generator of the subgroup of $\Z$ generated by the values $\mu_{\eta}$ for loops $\eta$ in $\frak{c}$.  Also, let $\Gamma_{\frak{c}}$ be the subgroup of $\R$ generated by the numbers $\int_{S^1\times[0,1]}u_{\eta}^{*}\omega$ for loops $\eta$ in $\frak{c}$.

We then have another well-defined function \[ \mu\co  \mathcal{O}_{\frak{c},L,H}\to \Z/N_{\frak{c}}\Z   \]  given by, for any $\gamma\in \mathcal{O}_{\frak{c},L,H}$, choosing a homotopy $v$ from $\gamma_{\frak{c}}$ to $\gamma$ through paths in $\mathcal{P}(L,L)$ and letting $\mu(\gamma)$ be, modulo $N_{\frak{c}}$, the Maslov--Viterbo index \cite{V} of   $(\gamma,v)$ \emph{i.e.},  the Maslov index of a loop of Lagrangian subspaces of $\R^{2n}$ determined by the aforementioned trivialization of $v^{*}TM$ along all but the right side of $\partial[0,1]^{2}$, and given over the right side by the path $(\phi_{H}^{t})_*T_{\gamma(0)}L$ (with appropriate oppositely-oriented $90$-degree rotations at $(s,t)=(1,0)$ and $(s,t)=(1,1)$ in order to obtain a continuous path).

As in the Hamiltonian case, we let $S_{(M,L)}=\cup_{\frak{c}\in \pi_0(\mathcal{P}(L,L))}\{\frak{c}\}\times \Z/N_{\frak{c}}\Z $, endowed with the action of $\mathbb{Z}$ given by addition in the second factor, and for $K$ an appropriate field, $(\frak{c},k)\in S_{(M,L)}$ and $\lambda\in\R$ we let \[ Crit_{\frak{c},k}^{\lambda}(\mathcal{A}_H)= \{[\gamma,v]|\gamma\in \mathcal{O}_{\frak{c},L,H},\,\mathcal{A}_H([\gamma,v])\leq \lambda,\,\mu([\gamma,v])=k\}\]
and \[ CF^{\lambda}_{\frak{c},k}(\hat{L}:H;K)=\left\{\left.\sum_{[\gamma,v]\in Crit_{\frak{c},k}^{\lambda}(\mathcal{A}_H)}a_{[\gamma,v]}[\gamma,v]\right|a_{[\gamma,v]}\in K,\,(\forall C\in\R)(\#\{[\gamma,v]|a_{[\gamma,v]}\neq 0,\,\mathcal{A}_H([\gamma,v])\geq C\}<\infty) \right\}.\]

Also let $CF_{\frak{c},k}(\hat{L}:H;K)=\cup_{\lambda}CF^{\lambda}_{\frak{c},k}(\hat{L}:H;K)$, and $CF_{\frak{c}}(\hat{L}:H;K)=\oplus_k CF_{\frak{c},k}(\hat{L}:H;K)$.  Under various hypotheses on the field $K$ and on $L$ (and on what we denote by $\hat{L}$, \emph{i.e.} on $L$ equipped with a relative spin structure and/or a bounding cochain as necessary) there are constructions of the Floer boundary operator $\partial_{J,H}$ on $CF(\hat{L}:H;K)$ in \cite{F88}, \cite{Oh93}, \cite{HL}, \cite{FOOO09}.  In all cases, these constructions (or straightforward modifications of them) give $CF(\hat{L}:H;K)$ the structure of a $S_{(M,L)}$-graded, $\R$-filtered complex over $K$ and each individual $CF_{\frak{c}}(\hat{L}:H;K)$ the structure of a $\Z/N_{\frak{c}}\Z$-graded, $\R$-filtered complex over $K$.\footnote{In the setup of \cite{FOOO09} there are certain possible bounding cochains for which this statement will not precisely be true due to grading-related issues; it will become true if we reduce the grading of $CF_{\frak{c}}(\hat{L}:H,K)$ from $\mathbb{Z}/N_{\frak{c}}\Z$ to $\Z/2\Z$ (as is possible, since in the setup of \cite{FOOO09} $L$ will be oriented and hence $N_{\frak{c}}$ will be even).   For ease of exposition we will ignore this distinction, which does not affect the ideas of any of the proofs to come.} Each $CF_{\frak{c},k}(\hat{L}:H;K)$ is a finite-dimensional vector space over $\Lambda^{K,\Gamma_{\frak{c}}}$, as is the full chain complex $CF_{\frak{c}}(\hat{L}:H;K)$, much like the situation in the Hamiltonian case.


Standard constructions of continuation maps that can be found, e.g., in \cite[Section 3.3]{Le08} and \cite[Section 5]{FOOO11}, can be adapted to prove the following analogue of Proposition \ref{revcont}.

\begin{prop}\label{revcontlag}  Given two Hamiltonians $H_-,H_+$ and appropriate auxiliary data used to construct the complexes $(CF(\hat{L}:H_{\pm};K),\partial_{J_{\pm},H_{\pm}})$, there exist:\begin{itemize} \item an $\mathcal{E}^+(H_+-H_-)$-morphism $\Phi\co (CF(\hat{L}:H_-;K) ,\partial_{J_-,H_-})\to (CF(\hat{L}:H_+;K),\partial_{J_+,H_+})$ \item an $\mathcal{E}^-(H_+-H_-)$-morphism $\Psi\co (CF(\hat{L}:H_+;K),\partial_{J_+,H_+})\to (CF(\hat{L}:H_-;K),\partial_{J_-,H_-})$
\item $\osc(H_+-H_-)$-homotopies $\mathcal{K}_{\pm}\co CF(\hat{L}:H_{\pm};K)\to CF(\hat{L}:H_{\pm};K)$ from $\Phi\circ \Psi$ and $\Psi\circ\Phi$ to the respective identities.\end{itemize}
In particular the Floer complexes  $(CF(\hat{L}:H_-;K),\partial_{J_-,H_-})$ and $(CF(\hat{L}:H_+;K),\partial_{J_+,H_+})$ are $\osc(H_+-H_-)$-quasiequivalent.
\end{prop}

Note that by Proposition \ref{quasicont} the special case in which $H_-=H_+$ is enough to imply that the boundary depth of the complex $(CF(\hat{L}:H;K),\partial_{J,H})$ is independent of the path of almost complex structures $J$ used to define it (indeed the methods in \cite{U09} that are used to prove Proposition \ref{oldinvt} can straightforwardly be adapted to prove that the filtered chain isomorphism type of $CF(\hat{L}:H;K)$ is independent of $J$, though we will not need to directly appeal to this fact).

We will also require the following analogue of Proposition \ref{newinvt}, in order to show that the boundary depth  depends only on the Lagrangian submanifold $(\phi_{H}^{1})^{-1}(L)$ and not on $H$.  In the weakly exact case a closely related argument appears in \cite[Section 2.1.3]{BC}.

\begin{prop}\label{newinvtlag} Let $G,H\co [0,1]\times M\to\R$ be two normalized Hamiltonians with the property that $(\phi_{G}^{1})^{-1}(L)=(\phi_{H}^{1})^{-1}(L)$. Then for appropriate paths $J_1,J_2$ of almost complex structures there is a shift-isomorphism $\Psi_*\co (CF(\hat{L}:G;K),\partial_{J_1,G})\to (CF(\hat{L}:H;K),\partial_{J_2,H})$.  In the case that the Floer homology $HF(\hat{L},\hat{L})$ is nonzero this shift-isomorphism restricts for all $\frak{c}\in\pi_0(\mathcal{P}(L,L))$ to a shift-isomorphism $CF_{\frak{c}}(\hat{L}:G;K)\to CF_{\frak{c}}(\hat{L}:H;K)$.
\end{prop}

\begin{proof} Define $\psi_t\co M\to M$ by \[ \psi_t=\phi_{H}^{t}\circ(\phi_{G}^{t})^{-1}.\]  The hypothesis on $H$ and $G$ shows that \begin{equation}\label{psill} \psi_1(L)=L.   \end{equation}  Let $F\co [0,1]\times M\to\R$ be the normalized Hamiltonian generating the path $\{\psi_t\}_{t\in[0,1]}$.  Since $\phi_{G}^{t}=\psi_{t}^{-1}\circ\phi_{H}^{t}$ we have \begin{equation}\label{fgh} G(t,m)=(H-F)(t,\psi_t(m)).\end{equation}

Define \begin{align*} \Psi\co \mathcal{P}(L,L)&\to \mathcal{P}(L,L) \\
(\Psi\gamma)(t)&=\psi_t(\gamma(t)); \end{align*}   this is well-defined by (\ref{psill}).  Denote the induced action on $\pi_0(\mathcal{P}(L,L))$ by $\Psi_*$.  We define a lift $\widetilde{\Psi}\co \widetilde{\mathcal{P}}(L,L)\to \widetilde{\mathcal{P}}(L,L)$ as follows.  For each $\frak{c}\in \mathcal{P}(L,L)$ we have chosen a basepoint $\gamma_{\frak{c}}\in\frak{c}$; we now choose additionally a homotopy $v_{\frak{c}}\co [0,1]^2\to M$ from $\gamma_{\frak{c}}$ to $\Psi\gamma_{\Psi_{*}^{-1}\frak{c}}$.  Now for $[\gamma,v]\in\widetilde{c}\subset \widetilde{\mathcal{P}}(L,L)$ define \[ \widetilde{\Psi}([\gamma, v])=[\Psi\gamma,v_{\Psi_*\frak{c}}\#\Psi v] \] (where for the homotopy $v\co [0,1]^2\to M$ from $\gamma_{\frak{c}}$ to $\gamma$ we define $(\Psi v)(s,t)=\psi_t(v(s,t))$, and where $\#$ denotes the obvious gluing operation).  Using (\ref{fgh}), a computation very similar to that in the proof of Proposition \ref{newinvt} shows that \[ 
\mathcal{A}_H(\widetilde{\Psi}[\gamma,v])=\mathcal{A}_G([\gamma,v])+\mathcal{A}_F([\Psi\gamma_{\frak{c}},v_{\Psi_*\frak{c}}]).\]
Moreover we have \begin{align*} \widetilde{\Psi}[\gamma,v]\in Crit(\mathcal{A}_H)& \Leftrightarrow \psi_1(\gamma(1))=\phi_{H}^{1}(\gamma(0)) \Leftrightarrow \gamma(1)=\phi_{G}^{1}(\gamma(0))\Leftrightarrow [\gamma,v]\in Crit(\mathcal{A}_G).\end{align*}  Thus for any $\frak{c}\in \pi_0(\mathcal{P}(L,L))$ and any $\lambda\in \R$, where we denote $\lambda_{\frak{c}}=\mathcal{A}_F([\Psi\gamma_{\frak{c}},v_{\Psi_*\frak{c}}])$, $\Psi$ induces a linear map $\Psi_*\co CF_{\frak{c}}^{\lambda}(\hat{L}:G;K)\to CF_{\Psi_*\frak{c}}^{\lambda+\lambda_{\frak{c}}}(\hat{L}:H;K).$  Moreover, given a path of almost complex structures $J_1(t)$, if we set $J_2(t)=\psi_{t*}J_1(t)\psi_{t*}^{-1}$ then just as in the proof of Proposition \ref{newinvt} we will have, for $u\co \R\times [0,1]\to M$, \[ \delbar_{J_2,H}(\Psi u)=\Psi_*\delbar_{J_1,G}u \] in obvious notation, as  a consequence of which $\Psi_*$ is an isomorphism of chain complexes.  Consideration of gradings then shows that $\Psi_*$ is a shift-isomorphism just as in the proof of Proposition \ref{newinvt}.

It remains to establish the final sentence of the proposition.  Assume then that $HF(\hat{L},\hat{L})\neq 0$. Clearly the proposition will follow if we show that $\Psi_*\frak{c}=\frak{c}$ for all $\frak{c}\in\pi_0(\mathcal{P}(L,L))$. Now in general $HF(\hat{L},\hat{L})$ splits as a direct sum $\oplus_{\frak{c}}HF_{\frak{c}}(\hat{L},\hat{L})$ where, for a generic Hamiltonian $H'$, $HF_{\frak{c}}(\hat{L},\hat{L})$ is the homology of the subcomplex $CF_{\frak{c}}(\hat{L}:H';K)$ generated by those $[\gamma,u]$ with $\gamma\in\frak{c}$ and $\dot{\gamma}(t)=X_{H'}(t,\gamma(t))$ (and of course this homology is independent of $H'$).     If we choose $H'$ to be $C^1$-small, all such $\gamma$ will be contained within a Darboux--Weinstein neighborhood of $L$ and so will be homotopic rel endpoints to a path entirely contained in $L$ (by the deformation retraction that shrinks the fibers of the Darboux--Weinstein neighborhood).  Thus where $\frak{c}_0\in \pi_0(\mathcal{P}(L,L))$ is the trivial class, all generators $[\gamma,v]$ for $CF(L:H';K)$ have $\gamma\in \frak{c}_0$.  This shows that we have $HF_{\frak{c}}(\hat{L},\hat{L})=0$ for all $\frak{c}\neq \frak{c}_0$.  So if $HF(\hat{L},\hat{L})\neq 0$ then $\frak{c}_0$ is distinguished as the unique class $\frak{c}$ in $\pi_0(\hat{L},\hat{L})$ such that $HF_{\frak{c}}(\hat{L},\hat{L})\neq 0$.  But on homology the shift-isomorphism that we have constructed above sends $HF_{\frak{c}_0}(\hat{L},\hat{L})$ to $HF_{\Psi_*\frak{c}_0}(\hat{L},\hat{L})$, so we must have $\Psi_*\frak{c}_0=\frak{c}_0$.

With this established one sees similarly to the proof of Proposition \ref{newinvt} that $\Psi_*\frak{c}=\frak{c}$ for all $\frak{c}$.  Indeed, since $\Psi_*\frak{c}_0=\frak{c}_0$, one finds that if $\gamma\in \mathcal{P}(L,L)$ then both $\Psi\gamma$ and $\gamma$ can be joined by homotopies (within $\mathcal{P}(L,L)$) to the path given by \[ t\mapsto\left\{\begin{array}{ll}\gamma(2t) & 0\leq t\leq 1/2 \\ \psi_{2t-1}(\gamma(1)) & 1/2\leq t\leq 1\end{array}\right.
\] Thus (when $HF(\hat{L},\hat{L})\neq 0$) we have $\Psi_*\frak{c}=\frak{c}$ for all $\frak{c}$.  With this established it is clear from the construction that $\Psi_*$ restricts as a shift-isomorphism  $CF_{\frak{c}}(\hat{L}:G;K)\to CF_{\frak{c}}(\hat{L}:H;K)$.
\end{proof}

\begin{remark}
The requirement that $HF(\hat{L},\hat{L})\neq 0$ in the last statement of Proposition \ref{newinvtlag} is necessary.  As is made clear in the proof, the last statement holds provided that, where $\psi_t\co M\to M$ is a Hamiltonian isotopy from the identity to a Hamiltonian diffeomorphism $\psi_1$ such that $\psi_1(L)=L$, and where we put $(\Psi\gamma)(t)=\psi_t(\gamma(t))$ for $\gamma\in \mathcal{P}(L,L)$, the map $\Psi$ acts as the identity on $\pi_0(\mathcal{P}(L,L))$.  As we argued above, this condition does hold when $HF(\hat{L},\hat{L})\neq 0$.  However, if we consider for instance the case in which $M=\R^2/\Z^2$ with its standard symplectic structure, and where $L$ is a small  contractible circle around $(0,0)$, it is easy to construct a Hamiltonian isotopy $\psi_t\co M\to M$ which restricts to $L$ as translation by $(t,0)$ for each $t\in [0,1]$.  In this case the associated map $\Psi\co \mathcal{P}(L,L)\to \mathcal{P}(L,L)$ obviously does not act as the identity on $\pi_0$.
\end{remark}

Consequently, whenever $L'\in\mathcal{L}(L)$ and $L\pitchfork L'$ we may \textbf{define $\beta_{\hat{L}}(L';K)$ to be the boundary depth of the chain complex $CF(\hat{L}:H;K)$ for any Hamiltonian $H\co [0,1]\times M\to\R$ with the property that $L'=(\phi_{H}^{1})^{-1}(L)$}.  Propositions \ref{revcontlag} and \ref{newinvtlag} readily imply that this quantity is independent of the choice of such $H$ (and of the almost complex structures and abstract perturbations involved in the construction of the Floer complex).  In case $L'=(\phi_{H}^{1})^{-1}(L)$ is not transverse to $L$, it is easy to see from the special case of Theorem \ref{mainproplag} (i) involving Lagrangians transverse to $L$ (to be proven presently) that we obtain a well-defined value $\beta_{\hat{L}}(L';K)$ as the limit of $\beta_{\hat{L}}((\phi_{H_n}^{1})^{-1}(L);K)$ for any $H_n$ with $H_n\to H$ in $C^2$ and 
$(\phi_{H_n}^{1})^{-1}(L)\pitchfork L$.

Also, in the case that $HF(\hat{L},\hat{L})\neq 0$, if $\frak{c}\in \pi_0(\mathcal{P}(L,L))$ and $L'\in \mathcal{L}(L)$ we denote by $\beta_{\hat{L},\frak{c}}(L';K)$ the boundary depth of the chain complex $CF_{\frak{c}}(\hat{L}:H;K)$ where $L'=(\phi_{H}^{1})^{-1}(L)$ assuming that $L'\pitchfork L$; again this depends only on $L$ and not on $H$ by Propositions \ref{revcontlag} and \ref{newinvtlag} because of the assumption that $HF(\hat{L},\hat{L})\neq 0$.  Of course we have $\beta_{\hat{L},\frak{c}}(L';K)\leq \beta_{\hat{L}}(L';K)$ when both are defined. If $L'$ and $L$ are not transverse we again define $\beta_{\hat{L},\frak{c}}(L';K)$ by continuity.

\begin{proof}[Proof of Theorem \ref{mainproplag} (i) $\mathrm{(}|\beta_{\hat{L}}(L_1;K)-\beta_{\hat{L}}(L_2;K)|\leq \delta(L_1,L_2)\mathrm{)}$] As indicated above we must first prove this in the case that $L_1\pitchfork L$ and $L_2\pitchfork L$ in order to even justify the definition of $\beta_{\hat{L}}(L';K)$ when $L'$ and $L$ are not transverse.  After we do this the general case of Theorem \ref{mainproplag} (i) will clearly follow by continuity.

 Choose a Hamiltonian $G\co [0,1]\times M\to \R$ so that $\phi_{G}^{1}(L_2)=L_1$ and a Hamiltonian $H$ such that $\phi_{H}^{1}(L_1)=L$.  So where $F(t,m)=H(t,m)+G(t,(\phi_{H}^{t})^{-1}(m))$ we have $\phi_{F}^{1}(L_2)=L$.  Thus $\beta_{\hat{L}}(L_2;K)$ is the boundary depth of $CF(\hat{L}:F;K)$, while $\beta_{\hat{L}}(L_1;K)$ is the boundary depth of $CF(\hat{L}:H;K)$.  So by Propositions \ref{quasicont} and \ref{revcontlag} we have \[ |\beta_{\hat{L}}(L_1;K)-\beta_{\hat{L}}(L_2;K)|\leq \osc(H-F)=\osc(G).\]  So since $G$ was an arbitrary Hamiltonian with $\phi_{G}^{1}(L_2)=L_1$ part (i) of Theorem \ref{mainproplag} follows (at least in the transverse case, and as noted earlier the general case then immediately follows by continuity).
\end{proof}

\begin{proof}[Proof of Theorem \ref{mainproplag} (ii) $\mathrm{(}$on $\beta_{\hat{L}}(L;K)\mathrm{)}$]  
Constructions in \cite[Section 5.6]{BC07} in the monotone case and from \cite{FOOO09c} in other cases provide a chain complex $\mathcal{C}(\hat{L};K)$, with the following properties: \begin{itemize} \item $\mathcal{C}(\hat{L};K)$ is, in the language of Section \ref{qcor}, a quantum correction of the Morse complex with coefficients in $K$ and with grading reduced modulo $N_{\frak{c}_0}$ of a suitable Morse function on $L$, such that for any $\ep>0$ the gap of $\mathcal{C}(\hat{L})$ in every grading may be arranged to  be at least $\sup_J \sigma(M,L,J)-\ep$.
\item For any $H$ such that $(\phi_{H}^{1})^{-1}(L)\pitchfork L$, $CF_{\frak{c}_0}(\hat{L}:H;K)$ is $\osc(H)$-quasiequivalent to $\mathcal{C}(\hat{L})$, and for $\frak{c}\neq \frak{c}_0$, $CF_{\frak{c}}(\hat{L}:H;K)$ is $\osc(H)$-quasiequivalent to the zero chain complex.
\end{itemize}
By taking a limit as $H\to 0$ it follows from Theorem \ref{mainproplag} (i) and Proposition \ref{quasicont} that $\beta_{\hat{L}}(L;K)$ may be computed as the boundary depth of $\mathcal{C}(\hat{L};K)$.  By Proposition \ref{qdef} this quantity is zero if $HF(\hat{L},\hat{L})\cong H_*(L)$, and otherwise is at least $\sup_J \sigma(M,L,J)-\ep$.  Since $\ep$ is arbitrary (depending on the choice of an almost complex structure in the construction of $\mathcal{C}(\hat{L};K)$, whereas $\beta_{\hat{L}}(L;K)$ is independent of this almost complex structure) we in fact have $\beta_{\hat{L}}(L;K)\geq \sup_J \sigma(M,L,J)$ if $HF(\hat{L},\hat{L})$ is not isomorphic to $H_*(L)$.
\end{proof}

\begin{proof}[Proof of Theorem \ref{mainproplag} (iii) $\mathrm{(}L\cap L_1=\varnothing\Rightarrow \beta_{\hat{L}}(L_1;K)=0\mathrm{)}$] This is obvious, as the trivial chain complex has boundary depth zero.\end{proof}

\begin{proof}[Proof of Theorem \ref{mainproplag} (iv) $\mathrm{(}\beta_{\hat{\Delta}}(\Gamma_{\phi};K)=\beta(\phi;K){)}$]
This property follows from a familiar comparison between Lagrangian and Hamiltonian Floer theory, as discussed for instance in \cite[Section 5.2]{BPS} in the exact case and in \cite[Section 6.2]{FOOO09b} in general.  Since, unlike these other references, we need to keep track of filtration levels and various other issues, let us review the argument.  Let $\Delta\subset M\times M$ be the diagonal and endow $M\times M$ with the symplectic structure $\Omega=(-\omega)\oplus\omega$,  There is a map (which modulo  issues relating to differentiability of paths would be a homeomorphism; in particular it induces a bijection on $\pi_0$) \begin{align*} \Upsilon\co \mathcal{L}M&\to\mathcal{P}(\Delta,\Delta)  \\
(\Upsilon\gamma)(t)&=\left(\gamma\left(1-\frac{t}{2}\right),\gamma\left(\frac{t}{2}\right)\right).\end{align*}
In particular if for every component $\frak{c}\in \pi_0(\mathcal{L}M)$ we have chosen a (smooth) basepoint $\gamma_{\frak{c}}$ we obtain basepoints $\Upsilon\gamma_{\frak{c}}$ for the components of $\mathcal{P}(\Delta,\Delta)$.  A symplectic trivialization of each $\gamma_{\frak{c}}^{*}TM$ (such as we have used as input in our formulation of Hamiltonian Floer theory) induces in the obvious way a symplectic trivialization of $(\Upsilon\gamma_{\frak{c}})^{*}T(M\times M)$, such that for $i=0,1$, $T_{\Upsilon\gamma_{\frak{c}}(i)}\Delta$ is identified with the diagonal in $\R^{2n}\times\R^{2n}$; at least after composing with the constant linear symplectomorphism $\left(\begin{array}{cc}1 & -1 \\ 0 & 1\end{array}\right)$ of $\R^{2n}\times \R^{2n}$ this gives a symplectic trivialization of $(\Upsilon\gamma_{\frak{c}})^{*}T(M\times M)$ suitable as an input for our formulation of Lagrangian Floer theory.

The map $\Upsilon$ evidently sends a homotopy $w$ from $\gamma_{\frak{c}}$ to $\gamma$ to a homotopy $\Upsilon w$ from $\Upsilon\gamma_{\frak{c}}$ to $\Upsilon\gamma$ given by \[ \Upsilon w(s,t)=\left(w\left(s,1-\frac{t}{2}\right),w\left(s,\frac{t}{2}\right)\right).\]  Clearly \[ \int_{[0,1]^{2}}(\Upsilon w)^{*}\Omega=\int_{[0,1]\times S^1}w^*\omega \] and so the assignment \[ \widetilde{\Upsilon}([\gamma,w])=[\Upsilon\gamma,\Upsilon w] \] gives a well-defined map $\widetilde{\mathcal{L}}M\to \widetilde{\mathcal{P}}(\Delta,\Delta)$ of the covering spaces which are the domains of our action functionals.

Now given a smooth function $H\co S^1\times M\to\R$ define $G\co [0,1]\times M\times M\to \R$ by \[ G(t,m_1,m_2)=\frac{1}{2}\left(H\left(1-\frac{t}{2},m_1\right)+H\left(\frac{t}{2},m_2\right)\right).\]

We plainly have, where the action functional on the left is that from Section \ref{hamsect} and the one on the right is from the present section, \[ \mathcal{A}_H([\gamma,w])=\mathcal{A}_G(\widetilde{\Upsilon}[\gamma,w]).\]  

Moreover critical points are easily seen to correspond under $\widetilde{\Upsilon}$: if $\gamma\co S^1\to M$ has $\dot{\gamma}(t)=X_{H}(t,\gamma(t))$ then $(\Upsilon\gamma)(t)=X_G(t,\Upsilon\gamma(t))$, and conversely if $\Gamma=(\gamma_1,\gamma_2)\co [0,1]\to M\times M$ with $\Gamma(0),\Gamma(1)\in \Delta$ obeys $\dot{\Gamma}(t)=X_G(t,\Gamma(t))$ then \[ \Upsilon^{-1}\Gamma(t)=\left\{\begin{array}{ll} \gamma_2(2t) & 0\leq t\leq 1/2 \\ \gamma_1(2-2s) & 1/2\leq t\leq 1\end{array}\right.\] will obey $\frac{d}{dt}\left(\Upsilon^{-1}\Gamma(t)\right)=X_H(t,\Upsilon^{-1}\Gamma(t))$ (in particular $\Upsilon^{-1}\Gamma$ will be smooth everywhere).

Consequently, at least at the level of modules, $\Upsilon$ induces an isomorphism  between the filtered Hamiltonian Floer groups $CF^{\lambda}(H;K)$ and the filtered Lagrangian Floer groups $CF^{\lambda}(\hat{\Delta}:G;K)$.  At least in the case where $M$ (and hence $\Delta$) is monotone and where $K=\mathbb{Z}/2$ it is a standard fact that this is an isomorphism of chain complexes: if one uses the $S^1$ family of almost complex structures $J_t$ on $M$ to define the differential on $CF(H;K)$, one should use the family $(-J_{1-t/2})\oplus J_{t/2}$ on $M\times M$ to define the differential on $CF(\hat{\Delta}:G;K)$, and then Floer trajectories $u\co \R\times S^1\to M$ on the Hamiltonian side will correspond to Floer trajectories $\Upsilon u(s,t)=\left(u\left(\frac{s}{2},1-\frac{t}{2}\right),u\left(\frac{s}{2},\frac{t}{2}\right)\right)$ on the Lagrangian side (and conversely---in particular if $v\co \R\times [0,1]\to M$ is a Floer trajectory on the Lagrangian side then one can appeal to elliptic regularity to show that $\Upsilon^{-1}v$ is smooth).  Moreover regularity of trajectories is preserved under this correspondence; indeed one can use $\Upsilon$ (and again appeal to elliptic regularity) to set up isomorphisms between the kernels of the respective linearizations and also between their cokernels.  Thus in the monotone case we have an isomorphism of filtered complexes between $CF(H;\mathbb{Z}/2)$ and $CF(\hat{\Delta}:G;\mathbb{Z}/2)$; hence the boundary depths of these complexes are the same.

Of course in the nonmonotone case (or indeed even in the monotone case if one wants to work in characteristic other than two) one needs to say somewhat more, since one needs to choose relative spin structures and bounding cochains in order to even define the chain complex  $CF(\hat{\Delta}:G;K)$.  In \cite[p. 32]{FOOO09b} the authors describe a relative spin structure on $\Delta$ with the property that $0$ is a bounding cochain for $\Delta$.  Accordingly we use this relative spin structure and the zero bounding cochain.  Since we are using the zero bounding cochain, there are no deformations involved in the Lagrangian Floer differential and so just as in the previous paragraph $\Upsilon$ sets up a correspondence between the moduli spaces of Hamiltonian and Lagrangian Floer trajectories, and also between the kernels and cokernels of the linearizations at these trajectories.  Consequently abstract perturbations as in \cite{FO}, \cite{FOOO09} can be constructed on either side so that the perturbed, transversely-cut-out moduli spaces will be in one-to-one correspondence.  To conclude that the Floer boundary operators coincide one then must check that the orientations (of the Kuranishi structures on the unperturbed moduli spaces, as in \cite[Appendix A]{FOOO09}) coincide.  

This latter fact follows quickly from the general method of constructing coherent orientations in Hamiltonian \cite[Section 21]{FO} and Lagrangian (\cite[Section 2.5]{HL},\cite[Section 8.1]{FOOO09}) Floer theory, together with the discussion on \cite[p. 33]{FOOO09b}.  We quickly sketch the argument.  With respect to the natural correspondence (similar to the one given by $\Upsilon$) between spheres  $u\co S^2\to M$ and discs  $\bar{u}\co (D^2,\partial D^2)\to (M\times M,\Delta)$, \cite{FOOO09b} shows that the canonical orientation of the determinant bundle of a Cauchy--Riemann operator on $u^*TM$ coincides with the orientation induced by the special relative spin structure that we are using on the determinant bundle of the corresponding Cauchy--Riemann operator with Lagrangian boundary conditions on $(\bar{u}^{*}T(M\times M),(\bar{u}|_{\partial D^2})^{*}T\Delta)$.  Now the determinant bundles of the appropriate linearizations at elements of Hamiltonian Floer moduli spaces are oriented as follows. First orient in arbitrary fashion the determinant bundles of appropriate ``left-cap'' operators $P^-(\gamma)$ associated to each $1$-periodic orbit $\gamma$ (the domains of these operators are the sections of a bundle over $D^2\cup_{\partial}([0,\infty)\times S^1)$).  This induces orientations of the determinant bundles first of the similar right-cap operators $P^+(\gamma)$ and then of the determinant bundles of the linearizations at elements of the Floer moduli space, by imposing compatibility under gluing with the canonical orientations of the determinant bundles of Cauchy--Riemman operators on bundles over $S^2$.  Similarly, in Lagrangian Floer theory for Hamiltonian-isotopic Lagrangians, one first orients arbitrarily the determinant bundles of left-cap operators $P^-(\bar{\gamma})$ whose domains are sections of bundles with Lagrangian boundary conditions over the union of a half-disc with an infinite strip, and then uses compatibility under gluing with the orientations of determinant bundles of Cauchy--Riemann operators over the disc that are imposed by the relative spin structure in order to orient first right-cap operators and then the determinant bundles of linearizations of elements of the Lagrangian Floer moduli spaces (this is explicitly explained in \cite{HL}; it is not difficult to see that the orientation prescription in \cite[Section 8.1.3]{FOOO09}, which addresses the more general situation where the Lagrangians might not be Hamiltonian-isotopic, reduces to the prescription of \cite{HL} in this special case).  Now in our case the correspondence $\Upsilon$ allows us to push forward the orientations for the Hamiltonian left-cap operators $P^-(\gamma)$ to orientations for the Lagrangian left-cap operators $P^-(\Upsilon\gamma)$.  Since, as shown on \cite[p. 33]{FOOO09}, the choice of relative spin structure ensures that the orientations for spheres in $M$ coincides with the orientations for discs in $M\times M$, the orientations that are imposed on the determinants of the linearizations at elements of the Floer moduli spaces will then coincide under $\Upsilon$.

This leads to an isomorphism of oriented Kuranishi structures between the moduli spaces on the Lagrangian and Hamiltonian sides.  Consequently $CF(H;K)$ is isomorphic as a filtered chain complex to $CF(\hat{\Delta}:G;K)$; in particular these complexes have the same boundary depths.
  
Given an arbitrary nondegenerate $\phi\in Ham(M,\omega)$ choose $H$ in the above discussion so that $\phi=(\phi_{H}^{1})^{-1}$.  We see that \[ \phi_{G}^{t}=\left(\phi_{H}^{1-t/2}\circ\phi\right)\times \phi_{H}^{t/2} \] and in particular \[ \phi_{G}^{1}=(\phi_{H}^{1/2}\circ\phi)\times \phi_{H}^{1/2}.\]  So \[ (\phi_{G}^{1})^{-1}(\Delta)=\{(p,q)\in M\times M|\phi_{H}^{1/2}\circ\phi(p)=\phi_{H}^{1/2}(q)\}=\{(p,q)\in M\times M|\phi(p)=q\}=\Gamma_{\phi}.\]
So $\beta_{\hat{\Delta}}(\Gamma_{\phi};K)$ is the boundary depth of $CF(\hat{\Delta}:G;K)$.  So since (using Theorem \ref{mainpropham} (ii))  $\beta(\phi;K)$ is the boundary depth of $CF(H;K)$ this proves the result for nondegenerate $\phi$, and then the case where $\phi$ is degenerate follows by continuity.

\end{proof}

We then quickly obtain one of our main applications:

\begin{proof}[Proof of Theorem \ref{diagmain}]  Since the map $\phi\mapsto \Gamma_{\phi}$ can only decrease the Hofer distance, the second inequality in Theorem \ref{diagmain} follows trivially from the corresponding inequality in Theorem \ref{hammain}.  As for the first inequality, by the invariance properties of Hofer's metric and by the fact that $(1_M\times \Phi(u))\Gamma_{\Phi(v)}=\Gamma_{\Phi(u+v)}$, we can reduce to the case where $w=0$; thus it suffices to show that $\delta(\Delta,\Gamma_{\Phi(v)})\geq \|v\|_{\ell_{\infty}}$.  Also note that $\beta_{\hat{\Delta}}(\Delta;K)=0$ by Theorem \ref{mainproplag} (ii), since $HF(\hat{\Delta},\hat{\Delta})$ is isomorphic to the singular homology of $\Delta$ (for the general, nonmonotone case see \cite[Theorem D]{FOOO09}).  Hence by Theorem \ref{mainproplag} (i) and (iv) we have \[ \delta(\Delta,\Gamma_{\Phi(v)})\geq \beta_{\hat{\Delta}}(\Gamma_{\Phi(v)};K)=\beta(\Phi(v);K).\]  But it was shown in the proof of Theorem \ref{hammain} that $\beta(\Phi(v);K)\geq \|v\|_{\ell_{\infty}}$, completing the proof.
\end{proof}

\begin{proof}[Proof of Theorem \ref{lagprod}] Just as with Theorem \ref{mainpropham} (v), this statement about boundary depths of products follows (assuming the K\"unneth property) directly from Theorem \ref{prodbeta} (of course, unlike in the Hamiltonian case, it is not automatically true that the Floer homologies of the factors are both nontrivial, which is why this additional assumption is necessary).\end{proof}

The calculation in Section \ref{prologue} is made relevant by the following:

\begin{prop}\label{lagmorse} Where $S^1=\R/\mathbb{Z}$, let $M$ be either $S^1\times \R$ or $S^1\times S^1$, with its standard symplectic structure.  For a Morse function $f\co S^1\to \R$ define $L_f=\{(x,f'(x))|x\in S^1\}\subset M$ (where of course the second component is taken modulo $1$ in case $M=S^1\times S^1$).  Then, where $L_0$ is endowed with the unique relative spin structure compatible with its standard orientation and with the zero bounding cochain, \[ \beta_{\hat{L}_0,\frak{c}_0}(L_f;K)= \beta_{Morse}(f;K).\]  
\end{prop}

\begin{proof}  Define $H\co M\to \R$ by $H(x,y)=-f(x)$.  Then for all $s>0$ we have $L_{sf}=(\phi_{sH}^{1})^{-1}(L_0)$.  Now it follows from \cite[Theorem 2]{Fl} that, there is $s_0>0$ such that for all positive $s\leq s_0$ and for appropriately-chosen almost complex structures in the definition of the Floer complex, all connecting trajectories arising in $CF(\hat{L}_0:sH;K)$ will degenerate to Morse trajectories, in view of which  $CF(\hat{L}_0:sH;K)$ (which in this range of $s$ coincides with $CF_{\frak{c}_0}(\hat{L}_0:sH;K)$) will be isomorphic as a filtered chain complex to $CM(-sf;K)$.  (Of course it  follows in particular that $HF(\hat{L}_0,\hat{L}_0)\neq 0$, so $\beta_{\hat{L}_0,\frak{c}_0}(\cdot;K)$ is well-defined as a function on $\mathcal{L}(L_0)$.)  Consequently we will have \[ \beta_{\hat{L}_0,\frak{c}_0}(L_{sf};K)=\beta_{Morse}(-sf;K)\mbox{ for }0<s\leq s_0. \]

Now for all $s>0$ the critical points of the action functional $\mathcal{A}_{sH}$ are those $[\gamma,v]\in \widetilde{\mathcal{P}}(L_0,L_0)$ where $\gamma\co [0,1]\to M$ has the form \[ \gamma(t)=(x_0,-stf'(x_0)).\]  Since we require $\gamma(0),\gamma(1)\in L_0$, in the case that $M=S^1\times \R$ this clearly requires that $f'(x_0)=0$.  In the case that $M=S^1\times S^1$, once $s$ is large enough there will be some additional critical points corresponding to nonzero values of $f'(x_0)$; however for these extra critical points $\gamma$ will represent a nontrivial class in $\pi_1(M,L_0)$, and it remains true that, for all $s$, the only critical points in the ``topologically trivial sector'' $\widetilde{\frak{c}}_0$ will have the form $[\gamma,v]$ where $\gamma(t)=(x_0,0)$ and $f'(x_0)=0$.

Thus in either case $CF_{\frak{c}_0}(\hat{L}_0;sH;K)$ is, for all $s$, generated by just those orbits arising from critical points of $f$.  Moreover since $\pi_2(M,L_0)=0$ the ``period group'' $\Gamma_{\frak{c}_0}$ is trivial (so the Novikov ring $\Lambda^{K,\Gamma_{\frak{c}_0}}$ over which $CF_{\frak{c}_0}(\hat{L}_0;sH;K)$ is defined just degenerates to $K$), and the action of a generator of $CF_{\frak{c}_0}(\hat{L}_0;sH;K)$ corresponding to a critical point $x_0$ is just $sH(x_0,0)=-sf(x_0)$.  So by (an easy special case of) Proposition \ref{depthattained} we will have \[ \beta_{\hat{L}_0,\frak{c}_0}(L_{sf};K)\in \{sf(x_0)-sf(x_1)|x_0,x_1\in Crit(f)\}.\]  Of course, $\beta_{Morse}(-sf;K)$ belongs to the same set.  Thus \[ s\mapsto \frac{1}{s}\beta_{\hat{L}_0,\frak{c}_0}(L_{sf};K)\mbox{ and }s\mapsto \frac{1}{s}\beta_{Morse}(-sf;K) \] are both continuous functions from $(0,1]$ to the finite set   $\{f(x_0)-f(x_1)|x_0,x_1\in Crit(f)\}$, and by the first paragraph they coincide for $s\leq s_0$, so they must in fact coincide for all $s$.  

\end{proof}

\begin{remark}  Clearly the same argument allows one, on a general cotangent bundle $T^*M$ of a closed manifold $M$, to relate the boundary depth of a graph of an exact $1$-form $df$ to the Morse-theoretic boundary depth of $f\co M\to\R$.  In this way one can obtain a proof of infinite diameter for the Hofer metric on Lagrangian submanifolds isotopic to the zero section $0_M$, though this latter result can be proven using older methods of Oh and Milinkovi\'c \cite{Mil}.  Combining this with   Theorem \ref{mainproplag} (v) yields (conditional on a K\"unneth formula) infinite diameter for $0_M\times L\subset T^*M\times N$ for any Lagrangian submanifold $L\subset N$ with nonvanishing Floer homology.
\end{remark}

\begin{cor}\label{betaosc} Suppose $f\co S^1\to \R$ is a smooth function such that, for some integer $m>1$, we have $f(x+1/m)=f(x)$ for all $x\in S^1$.  Then where $L_f=\{(x,f'(x))|x\in S^1\}\subset (S^1\times\R\mbox{ or }S^1\times S^1)$ we have \[ \beta_{\hat{L}_0,\frak{c}_0}(L_f;K)=\osc f.\]
\end{cor}

\begin{proof}  By the continuity properties of $\beta$ it suffices to prove the result when $f$ is Morse. Of course $HF(\hat{L}_0,\hat{L}_0)=H_*(L_0)$, so by Theorem \ref{mainproplag} (i) and (ii) and the fact that $\beta_{\hat{L}_0}(\cdot;K)\geq \beta_{\hat{L}_0,\frak{c}_0}(\cdot;K)$ we must have $\beta_{\hat{L}_0,\frak{c}_0}(L_f;K)\leq \osc f$.   So by Proposition \ref{lagmorse}  the corollary will follow provided that $\beta_{Morse}(f;K)\geq \osc f$.  But this is clear from Theorem \ref{s1calc} given our periodicity assumption on $f$: if in (\ref{mainlink}) we take $t_1,t_2$ to lie in $[0,1/m]$ and be, respectively, a global maximum and a global minimum, and if we then set $t_3=t_1+1/m$ and $t_4=t_2+1/m$, then $(t_1,t_2,t_3,t_4)$ will be in cyclic order with $\min\{f(t_1),f(t_3)\}-\max\{f(t_2),f(t_4)\}=\osc f$.
\end{proof}

\begin{proof}[Proof of Theorem \ref{s1prodmain}] The time-one map of the Hamiltonian $H_f\co T^2\times M\to\R$ defined by \[ H((x,y),p)=f(x)\] sends $L_g\times L$ to $L_{f+g}\times L$; in view of this and the invariance properties of $\delta$ it suffices to prove the theorem in case $f=0$; thus we are to show that \[ \osc g-C\leq \delta(L_0\times L,L_g\times L)\leq \osc g.\]  The second inequality is obvious from the first sentence of this proof (replacing $f$ by $g$; of course if $M$ is noncompact one has to cut off the Hamiltonian outside a neighborhood of $T^2\times L$, but this will not increase the Hofer norm).

Since $\pi_2(T^2,L_0)=0$ and $L\subset M$ is assumed monotone with minimal Maslov number at least 2 the K\"unneth property holds for $L_0$ and $L$, at least if one restricts to the Floer complexes corresponding to the component $\frak{c}_0$ of constant paths in each of the manifolds.  In particular $HF(\widehat{L_0\times L},\widehat{L_0\times L})\neq 0$ since we assume that $HF(\hat{L},\hat{L})\neq 0$, so the boundary depth $\beta_{\widehat{L_0\times L},\frak{c}_0}(\cdot;K)$ is well-defined as a function on $\mathcal{L}(L_0\times L)$.  Let $C=\beta_{\widehat{L_0\times L}}(L_0\times L;K)$. Thus $C=0$ in the case that $HF(L,L)$ is isomorphic to the singular homology of $L$ by Theorem \ref{mainproplag} (ii) together with the K\"unneth property.  Further Theorems \ref{mainproplag} (i) and \ref{lagprod} (or rather, a version of Theorem \ref{lagprod} for the restricted boundary depth $\beta_{\widehat{L_0\times L},\frak{c}_0}$, which follows equally easily from Theorem \ref{prodbeta}) show that \[ \delta(L_0\times L,L_g\times L)\geq \beta_{\widehat{L_0\times L}}(L_g\times L;K)-C\geq \beta_{\widehat{L_0\times L},\frak{c}_0}(L_g\times L;K)-C \geq \beta_{\hat{L}_0}(L_g;K)-C.\]  But the assumption that $g\in C^{\infty}_{m,0}(S^1)$ implies, by Corollary \ref{betaosc}, that $\beta_{\hat{L}_0}(L_g;K)=\osc g$.
\end{proof}

\section{Coefficient extension and attainment of the supremum}\label{alg2}

We now turn to some purely algebraic results needed to complete some of the proofs from the previous two sections.
In our conventions, the vector spaces underlying the Floer complexes that are denoted by $CF_{\frak{c}}$ have the following structure, with the function $\ell$ defined by \[ \ell\left(\sum a_{[\gamma,w]}[\gamma,w]\right)=\max\{\mathcal{A}_H([\gamma,w]) :a_{[\gamma,w]}\neq 0\}.\]

\begin{dfn}\label{filtvs} Let $K$ be a field and let $\Lambda=\Lambda^{K,\Gamma}$ be a Novikov field over $K$ (where $\Gamma\leq \mathbb{R}$ is an additive subgroup).  A \emph{finite-dimensional filtered $\Lambda$-vector space} $(C,\ell)$ is a $\Lambda$-vector space $C$ together with a function \[ \ell\co C\to\mathbb{R}\cup\{-\infty\}\] with the following property.  There is a finite basis $\{x_1,\ldots,x_m\}$ for $C$ such that $\ell(x_i)>-\infty$ for each $i$ and for all $\lambda_1,\ldots,\lambda_m\in\Lambda$ we have \begin{equation}\label{orthdef} \ell\left(\sum_{i=1}\lambda_ix_i\right)=\max_i (\ell(x_i)-\nu(\lambda_i)).\end{equation}  Such a basis $\{x_1,\ldots,x_m\}$ is called an \emph{orthogonal basis} for $C$.
\end{dfn}

In the case of the Hamiltonian Floer complexes $CF_{\frak{c}}(H;K)$ as defined in Section \ref{hamsect}, an orthogonal basis is given by the set $\{[\gamma^1,w^1],\ldots,[\gamma^m,w^m]\}$ where the $\gamma^i$ are the distinct elements of $\mathcal{O}_{\frak{c},H}$ and the $w^i$ are arbitrarily-chosen homotopies from the basepoint $\gamma_{\frak{c}}$ to $\gamma^i$.  Of course, a similar description applies to the Lagrangian Floer complexes.

\begin{dfn} A finite-dimensional filtered $\Lambda$-vector space is called \emph{standard} if it admits an \emph{orthonormal} basis, \emph{i.e.}, an orthogonal basis $\{x_1,\ldots,x_m\}$ such that $\ell(x_i)=0$ for all $i$.
\end{dfn}

In other words, $(C,\ell)$ is standard if and only if there is a vector space isomorphism $\Phi\co C\to \Lambda^m$ such that $\ell(c)=-\barnu(\Phi c)$ for all $c$; the image under $\Phi$ of any orthonormal basis for $C$ will then be an orthonormal basis for $\Lambda^m$ in the sense defined in Section \ref{qcor}.  Of course, not every finite-dimensional filtered $\Lambda$-vector space is standard, since in the standard case the image of the map $\ell$ is $\Gamma\cup\{-\infty\}$, which need not be the case in general.  On the other hand, if there is an orthogonal basis $\{x_1,\ldots,x_m\}$ for $C$ such that the real numbers $\ell(x_1),\ldots,\ell(x_m)$ all belong to the group $\Gamma$, then $(C,\ell)$ is standard, since then $\{T^{\ell(x_1)}x_1,\ldots,T^{\ell(x_m)}x_m\}$ is an orthonormal basis.

This shows that any finite-dimensional filtered $\Lambda^{K,\Gamma}$-vector space can be made standard after extending coefficients by tensoring with a larger Novikov field. Namely, if $\Gamma'$ is any subgroup of $\mathbb{R}$ containing both $\Gamma$ and the various $\ell(x_i)$, then $C\otimes_{\Lambda^{K,\Gamma}}\Lambda^{K,\Gamma'}$ acquires in an obvious way the structure of a  finite-dimensional filtered $\Lambda^{K,\Gamma'}$-vector space (define the function $\ell$ by the same formula as in (\ref{orthdef}), with the $\lambda_i$ now allowed to vary in $\Lambda^{K,\Gamma'}$ rather than just $\Lambda^{K,\Gamma}$), which is standard by the preceding paragraph.

\begin{prop}\label{coext}   Let $(C_0,\ell_0)$ and $(C_1,\ell_1)$ be  finite-dimensional filtered $\Lambda^{K,\Gamma}$-vector spaces, let $A\co C_0\to C_1$ be a $\Lambda^{K,\Gamma}$-linear map,  and let $x\in (Im\,A)\setminus\{0\}$.   If $\Gamma\leq \Gamma'$, consider the coefficient extension $A\otimes 1\co C_0\otimes_{\Lambda^{K,\Gamma}}\Lambda^{K,\Gamma'}\to C_1\otimes_{\Lambda^{K,\Gamma}}\Lambda^{K,\Gamma'}$.  Then \[ \inf\{\ell_0(y)|y\in C_0,\,Ay=x\}=\inf\{\ell_0(y)|y\in C_0\otimes_{\Lambda^{K,\Gamma}}\Lambda^{K,\Gamma'},\,(A\otimes 1)y=x\}.\]  In fact, for any $y\in C_0\otimes_{\Lambda^{K,\Gamma}}\Lambda^{K,\Gamma'}$ such that $0\neq(A\otimes 1)y\in C_1$, there is $y_0\in C_0$ such that $Ay_0=(A\otimes 1)y$ and $\ell_0(y_0)\leq \ell_0(y)$.
\end{prop}

\begin{proof} Since, via the inclusion $C_0\hookrightarrow C_0\otimes_{\Lambda^{K,\Gamma}}\Lambda^{K,\Gamma'}$, the set on the left hand side is contained in that on the right, the inequality ``$\geq$'' is trivial.  The last sentence of the proposition would clearly imply the reverse inequality, so it remains only to prove the last sentence.

For $i=0,1$ let $z_{1}^{i},\ldots,z_{m_i}^{i}$ be orthogonal bases for $C_i$.  Also define \[ \Xi=\left\{\left.\sum_{g\in \Gamma'}a_gT^g\in \Lambda^{K,\Gamma'}\right|(a_g\neq 0)\Rightarrow g\notin \Gamma\right\}.\]  Thus $\Xi$ is an additive subgroup of $\Lambda^{K,\Gamma'}$, and moreover is a vector space over $\Lambda^{K,\Gamma}$ (as we have $\Lambda^{K,\Gamma}\cdot \Xi=\Xi$), with $\Lambda^{K,\Gamma'}=\Lambda^{K,\Gamma}\oplus \Xi$ as $\Lambda_{K,\Gamma}$-vector spaces.  For $i=0,1$ let \[ D_i=\left\{\sum_{j=0}^{m_i}\xi_j z_{j}^{i}\left|\xi_j\in\Xi\right.\right\}\leq C_i\otimes_{\Lambda^{K,\Gamma}}\Lambda^{K,\Gamma'}.\]  We thus have \[ C_i\otimes_{\Lambda^{K,\Gamma}}\Lambda^{K,\Gamma'}=C_i\oplus D_i. \]  Now since the map $A$ is $\Lambda^{K,\Gamma}$-linear, we see that \[ (A\otimes 1)(C_0)\leq C_1\quad \mbox{and}\quad (A\otimes 1)(D_0)\leq D_1.\]  So suppose $0\neq x\in C_1$ and $y\in C\otimes_{\Lambda^{K,\Gamma}}\Lambda^{K,\Gamma'}$ with $(A\otimes 1)y=x$.  We can then write $y=y_0+y'$ with $y_0\in C_0$ and $y'\in D_0$.  But then since $x\in C_1$ we must have $Ay_0=x$ and $(A\otimes 1)y'=0$.  Moreover one easily sees that $\ell_0(y)=\max\{\ell_0(y_0),\ell_0(y')\}$, and so $\ell_0(y_0)\leq \ell_0(y)$, as desired.
\end{proof}

\begin{prop}\label{depthattained} Let $(C_0,\ell_0)$ and $(C_1,\ell_1)$ be two finite-dimensional filtered $\Lambda^{K,\Gamma}$-vector spaces and let $A\co C_0\to C_1$ be a not-identically-zero $\Lambda^{K,\Gamma}$-linear map.  Then there is $y_0\in C_0$ such that $Ay_0\neq 0$ and \begin{align*} \ell_0(y_0)-\ell_1(Ay_0)&=\inf\{\ell_0(y)-\ell_1(Ay_0)|y\in C_0,\,Ay=Ay_0\}
\\&=\sup_{x\in (Im\,A)\setminus\{0\}}\inf\{\ell_0(y)-\ell_1(x)|y\in C_0,\,Ay=x\}.\end{align*} 
\end{prop}

\begin{proof}  \textit{Step 1: We prove the proposition in the special case that $(C_0,\ell_0)$ and $(C_1,\ell_1)$ are both standard}.  In this case we may as well assume that, writing $\Lambda=\Lambda^{K,\Gamma}$, we have $(C_0,\ell_0)=(\Lambda^m,-\barnu)$ and $(C_1,\ell_1)=(\Lambda^n,-\barnu)$.

In this case, let $r$ be the rank of $A$, and use Lemma \ref{onlemma} to find an orthonormal basis $\{y_1,\ldots,y_m\}$ for $\Lambda^m$ such that $\{y_{r+1},\ldots,y_{m}\}$ is an orthonormal basis for $\ker A$ (namely, first choose $\{y_{r+1},\ldots,y_m\}$ as an orthonormal basis for $\ker A$, and then extend it to an orthonormal basis for all of $\Lambda^m$).  Also, let $\{x_1,\ldots,x_r\}$ be an orthonormal basis for $Im\,A$.   Thus for some invertible $r\times r$ matrix $P$ over $\Lambda$ we have $Ay_j=\sum_{i=1}^{r}P_{ij}x_i$ for $j=1,\ldots,r$, and $Ay_j=0$ for $j>r$.  Now define $B\co Im\,A\to C_0$ by \[ Bx_j=\sum_{i=1}^{r}(P^{-1})_{ij}y_i.\]  Thus \[ ABx_j=x_j\, (1\leq j\leq r),\quad BAy_j=y_j\, (1\leq j\leq r),\quad \mbox{and } BAy_j=0\,  (j>r).\]

So for any $x\in Im\,A$, the elements $y$ such that $Ay=x$ are precisely those of form $y=Bx+y'$ where $y'\in \ker A$.  Since $Bx\in span\{y_1,\ldots,y_r\}$, the orthonormality of our basis for $\Lambda^m$ shows that $\barnu(y)\leq \barnu(Bx)$, \emph{i.e.}, that $\ell_0(y)\geq \ell_0(Bx)$.   Consequently we have, for any $x\in (Im\,A)\setminus\{0\}$, \[ \ell_0(Bx)-\ell_1(x)=\inf\{\ell_0(y)-\ell_0(x)|Ay=x\}.\]

Now if $x=\sum_{j=1}^{r}\lambda_jx_j$ is any nonzero element of $Im\,A$ we have, using the general facts that $\nu(\lambda\mu)=\nu(\lambda)+\nu(\mu)$ and $\nu(\lambda+\mu)\geq\min\{\nu(\lambda),\nu(\mu)\}$ for $\lambda,\mu\in\Lambda$, \begin{align*} \ell_0(Bx)-\ell_1(x)&=\barnu\left(\sum_j \lambda_jx_j\right)-\barnu\left(B\sum_k \lambda_kx_k\right)
=\min_j \nu(\lambda_j)-\barnu\left(\sum_{i,k} (P^{-1})_{ik}\lambda_ky_i\right)\\&=
\min_j \nu(\lambda_j)-\min_i\left(\nu\left(\sum_k((P^{-1})_{ik})\lambda_k\right)\right)
\\&\leq\min_j\nu(\lambda_j)-\min_{i,k}\nu((P^{-1})_{ik}\lambda_k)=\min_j\nu(\lambda_j)-\min_{i,k}(\nu((P^{-1})_{ik}+\nu(\lambda_k)))
\\&\leq \min_j\nu(\lambda_j)-\min_{i,k}\nu((P^{-1})_{ik})-\min_k\nu(\lambda_k)=-\min_{i,k}\nu((P^{-1})_{ik}).
\end{align*}

Thus, if we choose $i_0,k_0$ such that $\nu((P^{-1})_{i_0k_0})=\min_{i,k}\nu((P^{-1})_{ik})$, we have \begin{equation}\label{levelbound}
\sup_{x\in (Im\,A)\setminus\{0\}}\ell_0(Bx)-\ell_1(x)\leq -\nu((P^{-1})_{i_0k_0}).
\end{equation}

On the other hand,  if we let $x=x_{k_0}$, we have $Bx=\sum_{i}(P^{-1})_{ik_0}y_i$ and so $\barnu(x)=0$ while $\barnu(Bx)=\min_i\nu (P^{-1})_{ik_0}=\nu((P^{-1})_{i_0k_0})$.  Thus \[ \ell_0(Bx)-\ell_1(x)=\barnu(x)-\barnu(Bx)=-\nu((P^{-1})_{i_0k_0})\quad (\mbox{ if }x=x_{k_0}).\]

Thus setting $y_0=Bx_{k_0}$, so that $x_{k_0}=Ay_0$, we have that \[ \ell_0(y_0)-\ell_1(Ay_0)=\sup_{x\in (Im\,A)\setminus\{0\}}\inf\{\ell_0(y)-\ell_0(x)|Ay=x\},\] proving the proposition in the case that the $(C_i,\ell_i)$ are standard and thus completing Step 1.

\textit{Step 2: We deduce the general case from the proof of Step 1.}
As noted just before Proposition \ref{coext}, for a suitable group $\Gamma'\leq\Gamma$, where $\Lambda'=\Lambda^{K,\Gamma'}$ it will hold that the filtered $\Lambda'$-vector spaces $C'_0=C_0\otimes_{\Lambda}\Lambda'$ and $C'_1=C_1\otimes_{\Lambda}\Lambda'$ are both standard. 

Step 1 and its proof provide a $\Lambda'$-linear map $B\co Im(A\otimes 1)\to C'_0$ with the properties that:\begin{itemize} \item $ABx=x$ for all $x$; \item For all nonzero $x$, we have \begin{equation}\label{alwaysinf}\ell_0(Bx)-\ell_1(x)=\inf\{\ell_0(y)-\ell_1(x)|y\in C'_0,\,(A\otimes 1)y=x\}; \mbox{ and}\end{equation}
\item There is a nonzero $x_0\in Im(A\otimes 1)$ such that \begin{equation}\label{x0sup} \ell_0(Bx_0)-\ell_1(x_0)=\sup_{0\neq x\in Im(A\otimes 1)}\ell_0(Bx)-\ell_1(x).\end{equation}\end{itemize}

Now we can write \[ x_0=\sum_{i=1}^{N}T^{g_i}x_{0,i}\] where $N\in\mathbb{N}\cup\{\infty\}$, each $g_i$ belongs to a distinct coset of $\Gamma$ in $\Gamma'$, and each $x_{0,i}\in Im\,A$ (write $x_0$ as a $\Lambda'$-linear combination of elements from a $\Lambda$-basis for $Im\,A$, and then group terms according to the cosets of their exponents). Using that the $g_i$ all belong to different cosets, we have \[ \ell_1(x_0)=\max_i \ell_1(T^{g_i}x_{0,i})=\max_i(\ell_1(x_{0,i})-g_i).\]  Meanwhile, in principle it may  not hold that $Bx_{0,i}\in C_0$, but it still follows from obvious properties of $\ell_0$ that \[ \ell_0(Bx_0)\leq \max_i \ell_0(T^{g_i}Bx_{0,i})=\max_i(\ell_0(Bx_{0,i})-g_i).\]  Now choose a value of $i$, say $i_0$, for which $\left(\ell_0(Bx_{0,i})-g_i\right)$ is maximized (in particular this implies $Bx_{0,i_0}\neq 0$, since $\ell_0(0)=-\infty$).  We then have \begin{align*}
\ell_0(Bx_0)-\ell_1(x_0)&\leq \left(\ell_0(Bx_{0,i_0})-g_{i_0}\right)-\max_i\left(\ell_1(x_{0,i})-g_i\right)
\\&\leq \left(\ell_0(Bx_{0,i_0})-g_{i_0}\right)-\left(\ell_1(x_{0,i_0})-g_{i_0}\right)=\ell_0(Bx_{0,i_0})-\ell_1(x_{0,i_0}).\end{align*}
But then by (\ref{x0sup}) we must have equality throughout the above string of inequalities.  In particular the nonzero element $x_{0,i_0}$ of $C_1\cap Im(A\otimes 1)\leq C'_1$ has \begin{equation}\label{inf2} \ell_0(Bx_{0,i_0})-\ell_1(x_{0,i_0})=\inf\{\ell_0(y)-\ell_1(x_{0,i_0})|y\in C'_0,\,(A\otimes 1)y=x_{0,i_0}\} \end{equation} and \[ \ell_0(Bx_{0,i_0})-\ell_1(x_{0,i_0})=\sup_{0\neq x\in Im(A\otimes 1)}\ell_0(Bx)-\ell_1(x).\]
While $Bx_{0,i_0}$ might not belong to $C_0$, Proposition \ref{coext} finds $y_0\in C_0$ such that $Ay_0=x_{0,i_0}$ and $\ell_0(y_0)\leq \ell_0(Bx_{0,i_0})$ (and so $\ell_0(y_0)= \ell_0(Bx_{0,i_0})$ by (\ref{inf2})).
This element $y_0$ is easily seen to  satisfy the requirements of the theorem: using Proposition \ref{coext} we have \[ \ell_0(y_0)-\ell_1(Ay_0)=\inf\{\ell_0(y)-\ell_1(Ay_0)|y\in C_0,\,Ay=Ay_0\}=\inf\{\ell_0(y)-\ell_1(Ay_0)|y\in C'_0,\,(A\otimes 1)y=Ay_0\},\] and \begin{align*} \ell_0(y_0)-\ell_1(Ay_0)&=\sup_{0\neq x\in Im(A\otimes 1)}\ell_0(Bx)-\ell_1(x)
\\&=\sup_{0\neq x\in Im(A\otimes 1)}\inf\{\ell_0(y)-\ell_1(x)|y\in C'_0,\,(A\otimes 1)y=x\}
\\&\geq \sup_{0\neq x\in Im\,A}\inf\{\ell_0(y)-\ell_1(x)|y\in C'_0,\,(A\otimes 1)y=x\}\\&=\sup_{0\neq x\in Im\,A}\inf\{\ell_0(y)-\ell_1(x)|y\in C_0,\,Ay=x\}
,\end{align*} so since $\ell_0(y_0)-\ell_1(Ay_0)$ belongs to the set over which the supremum is taken in the second-to-last expression we must have equality throughout.
\end{proof}

\begin{remark} \label{coind} Step 2 of the proof of Proposition \ref{depthattained} shows that, if $\Lambda=\Lambda^{K,\Gamma}$ and $\Lambda'=\Lambda^{K,\Gamma'}$ where $\Gamma\leq \Gamma'$, and if $A\co C_0\to C_1$ is a nonzero $\Lambda$-linear map between two finite-dimensional filtered $\Lambda$-vector spaces, then writing $C'_i=C_i\otimes_{\Lambda}\Lambda'$ we have\[ \sup_{x\in Im\,A\setminus\{0\}}\inf\{\ell_0(y)-\ell_1(x)|y\in C_0,\,Ay=x\}=\sup_{x\in Im(A\otimes 1)\setminus\{0\}}\inf\{\ell_0(y)-\ell_1(x)|y\in C'_0,\,(A\otimes 1)y=x\}.\]  This leads to the conclusion that the boundary depth is unaffected when we extend coefficients by passing to a larger Novikov field.

\end{remark}

\section{Tensor products} \label{tprod}

Let $(C,\ell_C)$ and $(D,\ell_D)$ be two finite-dimensional filtered $\Lambda$-vector spaces, with orthogonal bases $\{x_1,\ldots,x_m\}$ for $C$ and $\{y_1,\ldots,y_n\}$ for $D$.  These data then induce the structure of a finite-dimensional filtered $\Lambda$-vector space $(C\otimes_{\Lambda}D,\ell^{\otimes})$ on the tensor product, via the formula \begin{equation}\label{tendef} \ell^{\otimes}\left(\sum_{i,j}\lambda_{ij}x_i\otimes y_j\right)=\max_{i,j}\left(\ell_C(x_i)+\ell_D(y_j)-\nu(\lambda_{ij})\right).\end{equation}  This construction is canonical in the following sense:

\begin{lemma}\label{cantensor}
The definition of the function $\ell^{\otimes}\co C\otimes_{\Lambda}D\to \mathbb{R}\cup\{-\infty\}$ from (\ref{tendef}) is independent of the choice of orthogonal bases for $C$ and $D$: namely, if $\{w_1,\ldots,w_m\}\subset C$ and $\{z_1,\ldots,z_n\}\subset D$ are any other choices of orthogonal bases it continues to hold that \begin{equation}\label{tenchange} \ell^{\otimes}\left(\sum_{i,j}\mu_{ij}w_i\otimes z_j\right)=\max_{i,j}\left(\ell_C(w_i)+\ell_D(z_j)-\nu(\mu_{ij})\right).\end{equation}
\end{lemma}

\begin{proof}
By symmetry it is enough to just consider the effect of changing the orthogonal basis for $C$; thus we are to prove that $\ell^{\otimes}\left(\sum_{i,j}\mu_{ij}w_{i}\otimes y_j\right)=\max_{i,j}\left(\ell_C(w_i)+\ell_D(y_j)-\nu(\mu_{ij})\right)$ if $\{w_1,\ldots,w_m\}$ is a different orthogonal basis for $C$.

We first extend coefficients: choose $\Gamma'\leq \R$ to be a subgroup containing both $\Gamma$ and each of the $\ell(x_i)$ (hence also each of the $\ell(w_i)$), and write $\Lambda'=\Lambda^{K,\Gamma'}$.  Continue to denote by $\ell_C,\ell_D,\ell^{\otimes}$ the obvious extensions of these functions to, respectively, $C'=C\otimes_{\Lambda}\Lambda'$, $D'=D\otimes_{\Lambda}\Lambda'$, and $C'\otimes_{\Lambda'}D'$ (\emph{i.e.}, we just allow the $\lambda_i$ or $\lambda_{ij}$ in the defining formulas to vary in $\Lambda'$ rather than $\Lambda$).  

With these extended coefficients, \[ \{T^{\ell_C(w_1)}w_1,\ldots,T^{\ell_C(w_m)}w_m\}\mbox{ and } \{T^{\ell_C(x_1)}x_1,\ldots,T^{\ell_C(x_m)}x_m\} \] are \emph{orthonormal} bases for $C'$.  Denote by $N\in GL_m(\Lambda')$ the basis change matrix for these orthonormal bases, \emph{i.e.} the matrix such that \[ T^{\ell_C(w_l)}w_l=\sum_k N_{kl}T^{\ell_C(x_k)}x_k.\]  
It follows easily from Lemma \ref{onlemma} that $N$ has the form $N=N_0+N_+$ where $N_0\in GL_m(K)$ and all entries of $N_+$ belong to $\Lambda'_+=\{\lambda\in \Lambda'|\nu(\lambda)>0\}$.   From this it follows that \begin{equation}\label{nval} \mbox{for any $\alpha_1,\ldots,\alpha_m\in \Lambda'$, }\min_k \nu\left(\sum_l N_{kl}\alpha_l\right)=\min_k\nu(\alpha_k).\end{equation}

We then have, for any $\mu_{ij}\in \Lambda$, \begin{align*}
\ell^{\otimes}\left(\sum_{i,j}\mu_{ij}w_{i}\otimes y_j\right)&=\ell^{\otimes}\left(\sum_{i,j}\mu_{ij}\left(T^{-\ell_C(w_i)}\sum_k N_{ki}T^{\ell_C(x_k)}x_k\right)\otimes y_j\right)
\\&=\ell^{\otimes}\left(\sum_{k,j}\left(T^{\ell_C(x_k)}\sum_i N_{ki}T^{-\ell_C(w_i)}\mu_{ij}\right)x_k\otimes y_j\right)
\\&=\max_{k,j}\left(\ell_C(x_k)+\ell_D(y_j)-\nu\left(T^{\ell_C(x_k)}\sum_i N_{ki}T^{-\ell_C(w_i)}\mu_{ij}\right)\right)
\\&=\max_{k,j}\left(\ell_C(x_k)+\ell_D(y_j)-\ell_C(x_k)-\nu\left(\sum_i N_{ki}T^{-\ell_C(w_i)}\mu_{ij}\right)\right)
\\&=\max_j\left(\ell_D(y_j)-\min_k\nu\left(\sum_i N_{ki}T^{-\ell_C(w_i)}\mu_{ij}\right)\right)
\\&=\max_j\left(\ell_D(y_j)-\min_k\nu(T^{-\ell_C(w_k)}\mu_{kj})\right)=\max_{j,k}\left(\ell_D(y_j)+\ell_C(w_k)-\nu(\mu_{kj})\right),
\end{align*}
as desired, where the penultimate equality uses (\ref{nval}).
\end{proof}

This has the following useful immediate consequence:

\begin{cor}\label{orthstab}  Let $\{w_1,\ldots,w_m\}$ and $\{z_1,\ldots,z_n\}$ be any orthogonal bases for the finite-dimensional filtered $\Lambda$-vector spaces $(C,\ell_C)$ and $(D,\ell_D)$.  Then $\{w_i\otimes z_j|1\leq i\leq m,\,1\leq j\leq n\}$ is an orthogonal basis for $(C\otimes_{\Lambda}D,\ell^{\otimes})$, and $\ell^{\otimes}(w_i\otimes z_j)=\ell_C(w_i)+\ell_D(z_j)$.  In particular, if the bases $\{w_i\}$ and $\{z_j\}$ are orthonormal then so is the basis $\{w_i\otimes z_j\}$.
\end{cor}
\begin{proof} Indeed, this follows directly from the formula (\ref{tenchange}).\end{proof}

\begin{dfn}  If $(C,\ell)$ is a finite-dimensional filtered $\Lambda$-vector space and $U_1,\ldots,U_r$ are subspaces of $C$, we say that $U_1,\ldots,U_r$ are \emph{mutually orthogonal} (or that $U_1$ is orthogonal to $U_2,\ldots,U_r$) if whenever $u_i\in U_i$ for $1\leq i\leq r$ we have \[ \ell\left(\sum_{i=1}^{r}u_i\right)=\max_{1\leq i\leq r}u_i.\]
\end{dfn}

\begin{remark} If there is an orthogonal basis $\mathcal{B}=\{u_j\}$ for $C$ such that the various subspaces $U_i$ are spanned by disjoint subsets of $\mathcal{B}$, then the $U_i$ are mutually orthogonal.  In the case that $C$ is standard it follows straightforwardly from Lemma \ref{onlemma} that the converse holds; however I do not know if the converse still always holds when $C$ is not standard.  
\end{remark}

In general, if $(E,\ell)$ is a finite-dimensional filtered $\Lambda$-vector space we obtain a filtration on $E$ by setting \[ E^{\lambda}=\{e\in E|\ell(e)\leq \lambda\}.\] If $E$ admits a $\Lambda$-linear map $\partial\co E\to E$ such that $\partial^2=0$ and $\partial(E^{\lambda})\leq E^{\lambda}$, this gives $E$ the structure of an $\mathbb{R}$-filtered complex over $K$ in the sense of Definition \ref{rfilt} (with, for simplicity, the grading being given by a $1$-element set).  We thus have the boundary depth \[ b(E,\partial)=\inf\{\beta\geq 0|(\forall \lambda\in\R)(E^{\lambda}\cap Im\partial)\subset \partial(E^{\lambda+\beta})\}.\]  It is easy to check that \begin{equation}\label{altbeta} b(E,\partial)=\left\{\begin{array}{ll}\sup_{x\in Im\,\partial\setminus\{0\}}\inf\{\ell(y)-\ell(x)|\partial y=x\} & \mbox{ if }\partial\neq 0, \\ 0 & \mbox{ if }\partial=0.\end{array}\right.\end{equation}

So now suppose we are given two finite-dimensional filtered $\Lambda$-vector spaces $(C,\ell_C)$ and $(D,\ell_D)$ which are chain complexes, with $\Lambda$-linear operators $\partial_C\co C\to C$ and $\partial_D\co D\to D$ such that $\partial_{C}^{2}=\partial_{D}^{2}=0$ and $\partial_C(C^{\lambda})\leq C^{\lambda}$ and $\partial_{D}(D^{\lambda})\leq D^{\lambda}$.

Our aim is to compare the boundary depth of the tensor product complex $C\otimes_{\Lambda}D$ to the boundary depths of $C$ and $D$. For simplicity we will make the minimal assumptions on $C$ and $D$ necessary to get a natural chain complex structure on $C\otimes_{\Lambda}D$; namely we assume that either:
\begin{itemize} \item[(i)] The characteristic of the field $K$ underlying the Novikov field $\Lambda^{K,\Gamma}$ is $2$; or
\item[(ii)] $C$ has a $\mathbb{Z}_2$ grading, \emph{i.e.}, we have $C=C_0\oplus C_1$ where $C_0$ and $C_1$ are orthogonal $\Lambda$-linear subspaces, and $\partial_C(C_1)\leq C_0$ and $\partial_C(C_0)\leq C_1$. \end{itemize}
Define a $\Lambda$-linear map $(-1)^{|\cdot|}\co C\to C$ by setting it equal to the identity in Case (i) above, and in Case (ii), setting $(-1)^{|\cdot|}|_{C_0}$ equal to $1$ and $(-1)^{|\cdot|}|_{C_1}$ equal to $-1$.  Given that $C_0$ and $C_1$ are orthogonal, it is clear that $\ell_C((-1)^{|\cdot|}c)=\ell_C(c)$ for any $c\in C$.

If we define \[ \partial^{\otimes}\co C\otimes_{\Lambda}D\to C\otimes_{\Lambda}D \] by  \[ \partial^{\otimes}=\partial_C\otimes 1_D+(-1)^{|\cdot|}\otimes \partial_D,\] then one has $\partial^{\otimes}\circ\partial^{\otimes}=0$ in either of the above two cases, and moreover $\partial^{\otimes}((C\otimes_{\Lambda}D)^{\lambda})\leq (C\otimes_{\Lambda}D)^{\lambda}$ where the filtration on $C\otimes_{\Lambda}D$ is that induced by $\ell^{\otimes}$, so we may consider the boundary depth $b(C\otimes_{\Lambda}D,\partial^{\otimes})$.

\begin{theorem}\label{prodbeta} Under the above circumstances:
\begin{itemize}\item[(a)] $b(C\otimes_{\Lambda}D,\partial^{\otimes})\geq \min\{b(C,\partial_C),b(D,\partial_D)\}.$
\item[(b)] If the homology $H(D,\partial_D)$ is nonzero, then $b(C\otimes_{\Lambda}D,\partial^{\otimes})\geq b(C,\partial_C)$; and if $H(C,\partial_C)$ is nonzero, then $b(C\otimes_{\Lambda}D,\partial^{\otimes})\geq b(D,\partial_D)$.\end{itemize}
\end{theorem}

\begin{proof} By enlarging the Novikov field $\Lambda$, we may arrange that $C$ and $D$ (and hence $C\otimes_{\Lambda}D$) admit orthonormal bases and so are standard; by Remark \ref{coind} this will not affect the boundary depths.  So assume that $C$ and $D$ are standard. (In the $\mathbb{Z}_2$-graded case this also implies that $C_0$ and $C_1$ are each standard, as they admit orthonormal bases by Lemma \ref{onlemma}.)  

Also, the theorem is straightforward in the case that one or the other of the differentials $\partial_C$ and $\partial_D$ is identically zero, so we assume that both of them are nonzero.

Now using Lemma \ref{onlemma} we may choose an orthonormal basis $\{x_1,\ldots,x_r\}$ for $Im\,\partial_C$; then extend this to an orthonormal basis $\{x_1,\ldots,x_r,x_{r+1},\ldots,x_s\}$ for $\ker\partial_C$; and finally extend this to an orthonormal basis $\{x_1,\ldots,x_m\}$ for all of $C$.  In the $\mathbb{Z}_2$-graded case, since $C_0$ and $C_1$ are assumed orthogonal, this may be done (and we assume it is done) in such a way that each $x_i$ belongs either to $C_0$ or to $C_1$.  Likewise, choose an orthonormal basis $\{z_1,\ldots,z_p\}$ for $Im\,\partial_D$; extend this to an orthonormal basis $\{z_1,\ldots,z_q\}$ for $\ker\partial_D$ and subsequently to an orthonormal basis $\{z_1,\ldots,z_n\}$ for all of $D$.  Write \begin{align*} B^C&=span_{\Lambda}\{x_1,\ldots,x_r\}  \quad  &B^D&=span_{\Lambda}\{z_1,\ldots,z_p\} ,\\
H^C&=span_{\Lambda}\{x_{r+1},\ldots,x_s\}\quad &H^D&=span_{\Lambda}\{z_{p+1},\ldots,z_q\},\\
F^C&=span_{\Lambda}\{x_{s+1},\ldots,x_m\}\quad &F^D&=span_{\Lambda}\{z_{q+1},\ldots,z_n\}.\end{align*}

In the $\mathbb{Z}_2$-graded case, since the $x_i$ are all chosen to belong either to $C_0$ or to $C_1$, the spaces $B^C$, $H^C$, and $F^C$ are all preserved by the operator $(-1)^{|\cdot|}\co C\to C$.

Thus $C=B^C\oplus H^C\oplus F^C$; the subspaces $B^C,H^C,F^C$ are mutually orthogonal; and $\partial_C$ maps $F^C$ bijectively to $B^C$ with kernel $B^C\oplus H^C$.  Also, if $x\in B^C$ then the unique element $y$ of $F^C$ with the property that $\partial_Cy=x$ obeys \[ \ell_C(y)-\ell_C(x)=\inf\{\ell_C(y')-\ell_C(x)|\partial_Cy'=x\},\] for if $\partial_C y'=x$ then $y'-y\in B^C\oplus H^C$ and so by the orthogonality of $B^C,H^C,$ and $F^C$ we have $\ell_C(y')=\max\{\ell_C(y),\ell_C(y'-y)\}\geq \ell_C(y)$.
Of course, similar remarks apply to $B^D,H^D$, and $F^D$.

By Corollary \ref{orthstab}, the set $\{x_i\otimes z_j|1\leq i\leq m,1\leq j\leq n\}$ forms an orthonormal basis for $C\otimes_{\Lambda}D$.  Consequently we have an \emph{orthogonal} decomposition (where all tensor products are over $\Lambda$) \[
C\otimes D=\left((\ker \partial_C)\otimes (\ker\partial_D)\right)\oplus (F^C\otimes B^D)\oplus (F^C\otimes H^D)\oplus (B^C\otimes F^D)\oplus (H^C\otimes F^D)\oplus (F^C\otimes F^D).\]  Furthermore, it's easy to see that \[ \ker \partial^{\otimes}\leq  \left((\ker \partial_C)\otimes (\ker\partial_D)\right)\oplus (F^C\otimes B^D)\oplus (B^C\otimes F^D).\]
In particular, \begin{equation}\label{allorth}\mbox{the subspaces }\ker\partial^{\otimes},\,F^C\otimes F^D,\,F^C\otimes H^D,\mbox{ and }H^C\otimes F^D\mbox{ are mutually orthogonal}.\end{equation}

Since we assume that $\partial_C$ and $\partial_D$ are both nonzero, Proposition \ref{depthattained} and (\ref{altbeta}) show that there are nonzero $x\in B^C$, $z\in B^D$ such that \[ b(C,\partial_C)=\inf\{\ell_C(w)-\ell_C(x)|\partial_Cw=x\}\mbox{ and }b(D,\partial_D)=\inf\{\ell_D(y)-\ell_D(z)|\partial_Dy=z\}.\]  Moreover, as noted earlier, if we choose $w\in F^C$ to be the unique element of $F^C$ with $\partial_C w=x$, then by the orthogonality of $F^C$ and $\ker \partial_C$, $w$ has the infimal filtration level of all primitives of $x$, and so \[ \ell_C(w)-\ell_C(\partial_C w)=\beta(C,\partial).\]  Similarly, if $y\in F^D$ is chosen as the unique primitive of $z$ which belongs to $F^D$, then \[ \ell_D(y)-\ell_D(\partial_Dy)=\beta(D,\partial).\]

Now $w\otimes y\in F^C\otimes F^D$, and $F^C\otimes F^D$ is orthogonal to $\ker \partial^{\otimes}$, so if $\alpha\in \ker\partial^{\otimes}$ then $\ell^{\otimes}(w\otimes y+\alpha)\geq \ell^{\otimes}(w\otimes y)$. Thus, \[ \inf\{\ell^{\otimes}(\beta)|\partial^{\otimes}\beta=\partial^{\otimes}(w\otimes y)\}=\ell^{\otimes}(w\otimes y)=\ell_C(w)+\ell_D(y)\] where the last equality follows from expanding out $w$ and $y$ in terms of the orthonormal bases $\{x_i\}$ and $\{y_j\}$ and using Corollary \ref{orthstab}.
Also, using that $B^C\otimes F^D$ is orthogonal to $F^C\otimes B^D$, we have \begin{align*} \ell^{\otimes}(\partial^{\otimes}(w\otimes y))&=\ell^{\otimes}(x\otimes y+(-1)^{|\cdot|}w\otimes z)
\\&=\max\{\ell^{\otimes}(x\otimes y),\ell^{\otimes}((-1)^{|\cdot|}w\otimes z)\}=\max\{\ell_C(x)+\ell_D(y),\ell_C((-1)^{|\cdot|}w)+\ell_D(z)\} \\&=\max\{\ell_C(x)+\ell_D(y),\ell_C(w)+\ell_D(z)\}.\end{align*} (In particular $\partial^{\otimes}(w\otimes y)\neq 0$ since $w,x,y,z$ are all nonzero.)

Thus \begin{align*} \inf\{\ell^{\otimes}(\beta)-\ell^{\otimes}(\partial^{\otimes}(w\otimes y))&|\partial^{\otimes}\beta=\partial^{\otimes}(w\otimes y)\}=\ell^{\otimes}(w\otimes y)-\ell^{\otimes}(\partial^{\otimes}(w\otimes y))
\\&=\ell_C(w)+\ell_D(y)-\max\{\ell_C(x)+\ell_D(y),\ell_C(w)+\ell_D(z)\}\\&=\min\{\ell_C(w)-\ell_C(x),\ell_D(y)-\ell_D(z)\}=\min\{b(C,\partial_C),b(D,\partial_D)\}.\end{align*}  In view of (\ref{altbeta}) this proves part (a) of the theorem.

Now assume that $H(D,\partial_D)\neq 0$, which is equivalent to the subspace $H^D\leq D$ being nonzero.  Choose a nonzero element $z$ of $H^D$, and let $w\in F^C$, $x\in B^C$ be as above, so that $\partial w=x$ and $b(C,\partial_C)=\ell_C(w)-\ell_C(x)$.  Then $\partial^{\otimes}(w\otimes z)=x\otimes z$; further since $w\otimes z\in F^C\otimes H^D$ and $F^C\otimes H^D$ is orthogonal to $\ker \partial^{\otimes}$ we have \[ \ell^{\otimes}(w\otimes z)=\inf\{\ell^{\otimes}(\beta)|\partial^{\otimes}\beta=x\otimes z\}.\]  Hence \begin{align*} \inf\{\ell^{\otimes}(\beta)-\ell^{\otimes}(x\otimes z)|\partial^{\otimes}\beta=x\otimes z\}&=\ell^{\otimes} (w\otimes z)-\ell^{\otimes}(x\otimes z)\\&=(\ell_C(w)+\ell_D(z))-(\ell_C(x)+\ell_D(z))=b(C,\partial_C),\end{align*} which proves the first statement of part (b) of the theorem.

The second statement of part (b) is of course proven in essentially the same way, taking appropriate account of signs: if $C$ has nontrivial homology, so that $H^C\neq 0$, choose a nonzero element $x\in H^C$, and as before choose $y\in F^D$, $z\in B^D$ so that $\partial_D y=z$ and $\ell_D(y)-\ell_D(z)=b(D,\partial_D)$.  Then $x\otimes y\in H^C\otimes F^D$, which is orthogonal to $\ker\partial^{\otimes}$, and \[ \ell^{\otimes}(\partial^{\otimes}(x\otimes y))=\ell^{\otimes}((-1)^{|\cdot|}x\otimes z)=\ell_C((-1)^{|\cdot|}x)+\ell_D(z)=\ell_C(x)+\ell_D(z),\] so we get as before that \[ b(C\otimes_{\Lambda}D,\partial^{\otimes})\geq \ell^{\otimes}(x\otimes y)-\ell^{\otimes}(\partial^{\otimes}(x\otimes y))=\ell_D(y)-\ell_D(z)=b(D,\partial_D).\]

\end{proof}

\section{Finite diameter for $S^1$ in $\R^2$}\label{circlesect}

In this section we prove that, where $L_0=\{(x,y)|x^2+y^2=1\}\subset \R^2$, Hofer's metric $\delta$ on the space $\mathcal{L}(L_0)$ of Lagrangian submanifolds Hamiltonian-isotopic to $L_0$ has finite diameter.  The argument is fairly simple and perhaps known, but I have not been able to find it in the literature.

\begin{lemma} \label{disjointsame} Let $L_1,L_2\in \mathcal{L}(L_0)$ be such that $L_0\cap L_1=L_0\cap L_2=\varnothing$.  Then $\delta(L_0,L_1)=\delta(L_0,L_2)$. 
\end{lemma}

\begin{proof} We repeatedly use the following immediate consequence of the invariance property of $\delta$: for $\psi\in Ham(\R^2)$ and for $L,L'\in \mathcal{L}(L_0)$ such that $\psi(L)=L$, we have $\delta(L,L')=\delta(L,\psi(L'))$.

Given $L_1,L_2$ as in the lemma, let $\phi\in Ham(\R^2)$ be such that $\phi(L_1)=L_2$.  Let $R\in\R$ be so large that, where $D(R)=\{x^2+y^2<R^2\}$, we have \[ L_1\cup L_2\cup (supp(\phi))\subset D(R).\]
Let \[ L(R)=\{(x-R-1)^2+y^2=1\}.\]  Since $L(R)$ is disjoint from the support of $\phi$ we have \begin{equation} \label{Rsame}
\delta(L_1,L(R))=\delta(L_2,L(R)).
\end{equation}
Now for $i=0,1,2$, let $V_i$ denote the bounded component of $\R^2\setminus L_i$, and let $W_i$ denote the unbounded component.  Thus in each case $V_i$ has area $\pi$, and $V_i=\R^2\setminus \bar{W}_i$.  Since $L_0\cap L_1=\varnothing$, $L_0\cap V_1$ and $L_0\cap W_1$ are both relatively open and closed in $L_0$; hence by the connectedness of $L_0$ either $L_0\subset V_1$ or $L_0\subset W_1$. 

We claim that $L_0\subset W_1$.  If this were not the case, so that $L_0\subset V_1$, then since $\R^2=V_0\cup L_0\cup W_0$ and $L_0\cap W_1=\varnothing$,  $W_1$ would be the union of the disjoint open sets $W_1\cap V_0$ and $W_1\cap W_0$; by the connectedness of $W_1$ and the fact that $V_0$ is bounded while $W_1$ is unbounded it would follow that $W_1\subset W_0$, and hence that $V_0\subset V_1$.  So since $\partial V_0=L_0\subset V_1$ we would have $\bar{V}_0\subset V_1$, and hence  a neighborhood of $\bar{V}_0$ would still be contained in $V_1$.  But since $V_0$ and $V_1$ are open sets of equal area this is impossible.  This contradiction shows that indeed $L_0\subset W_1$. 

Hence $V_1$ is the disjoint union of the open sets $V_1\cap V_0$ and $V_1\cap W_0$.  As before it is impossible for $V_1\subset V_0$ by area considerations, so since $V_1$ is connected $V_1\subset W_0$, and so $\bar{V}_1\cap V_0=\varnothing$.  Thus $\bar{V}_0=L_0\cup V_0$ is disjoint from $\bar{V}_1$, and so $\bar{V}_0\subset W_1$.  Of course the same argument shows that $\bar{V}_0\subset W_2$.

Write $V(R)$ for the bounded component of $\R^2\setminus L(R)$; since $\overline{V(R)}\cap D(R)=\varnothing$ we have $\overline{V(R)}\subset W_i$ for $i=1,2$.  Since also $\bar{V}_0\subset W_i$, it in particular holds that the points $(1,0)$ and $(R,0)$ belong to the unbounded component $W_i$ of $\R^2\setminus L_i$ for $i=1,2$.  So for $i=1,2$ let $\gamma_i$ be a path in $W_i$ connecting $(1,0)$ to $(R,0)$.   

Since $\bar{V}_0\cup\gamma_i\cup \overline{V(R)}\subset W_i$ and $W_i$ is open, we may take a neighborhood $U_i$ of  $\bar{V}_0\cup\gamma_i\cup \overline{V(R)}$ with still $U_i\subset W_i$.  It is then straightforward to find a Hamiltonian isotopy supported in $U_i$ whose time-one map $\phi_i$ has the property that $\phi_i(L_0)=L(R)$.  In particular since the support of the isotopy is disjoint from $L_i$ we have $\phi_i(L_i)=L_i$.  Consequently \begin{equation}
\label{iR} \delta(L_0,L_i)=\delta(L(R),L_i) \quad (i=1,2). \end{equation}
Thus by (\ref{Rsame}) and (\ref{iR}) we have \[ \delta(L_0,L_1)=\delta(L(R),L_1)=\delta(L(R),L_2)=\delta(L_0,L_2),\] as desired.
\end{proof}

\begin{cor} Choose any $L_1\in\mathcal{L}(L_0)$ such that $L_0\cap L_1=\varnothing$.  Then for all $L,L'\in \mathcal{L}(L_0)$ we have \[ \delta(L,L')\leq 2\delta(L_0,L_1).\] 
\end{cor}

\begin{proof} By the invariance of $\delta$ we may assume that $L=L_0$.  By applying a sufficiently distant translation to $L_0$ we may find $L_2\in\mathcal{L}(L_0)$ so that $L_0\cap L_2=L'\cap L_2=\varnothing$.  By Lemma \ref{disjointsame} and the invariance of $\delta$ we have $\delta(L_0,L_2)=\delta(L',L_2)=\delta(L_0,L_1)$.  Thus \[ \delta(L_0,L')\leq \delta(L_0,L_2)+\delta(L_2,L')= 2\delta(L_0,L_1).\]
\end{proof}

Using Chekanov's theorem \cite{Ch98} and Lemma \ref{disjointsame} it is easy to see that the common value of $\delta(L_0,L_1)$ for all $L_1\in\mathcal{L}(L_0)$ which are disjoint from $L_0$ is precisely $\pi$.  Thus we have shown that the diameter of $\mathcal{L}(L_0)$ is at most $2\pi$.

\appendix
\section{Transversality for $t$-independent Floer trajectories}

This appendix provides the details necessary for a technical point in the proof of Theorem \ref{mmcomp}, namely that a $t$-independent solution of the Floer equation associated to suitable $t$-independent almost complex structures and Hamiltonians can be arranged to be cut out transversely by slightly rescaling the Hamiltonian.  We begin with some preparation from linear algebra.

Throughout this section let $V$ be a finite-dimensional real inner product space.  The norm of a linear operator on $V$ (or on $V\oplus V$) will always refer to its operator norm with respect to the inner product.
Define the linear map $E\co V\oplus V\to V\oplus V$ by \[ E\left(\begin{array}{c} x\\y\end{array}\right)=\left(\begin{array}{c} -x\\y\end{array}\right) .\]

\begin{prop}\label{eigen} Let $B_1,B_2\co V \to  V$ be symmetric linear operators and define a linear operator $B\co V\oplus V\to V\oplus V$ by  \[ B\left(\begin{array}{c} x\\y\end{array}\right)=\left(\begin{array}{c} B_1x+B_2y\\B_2x+B_1 y\end{array}\right) .\] Then for all real $\mu$ with $|\mu|<\frac{1}{\|B_1\|+\|B_2\|}$ the operator $E+\mu B$ has precisely $(\dim V)$-many positive eigenvalues and $(\dim V)$-many negative eigenvalues, counting multiplicities.  Moreover where $\Pi_{B}^{+}(\mu),\Pi_{B}^{-}(\mu)\co V\oplus V\to V\oplus V$ denote the orthogonal projections onto the spans of those eigenvectors with, respectively, positive or negative eigenvalue, $\Pi_{B}^{+}$ and $\Pi_{B}^{-}$ are real analytic functions of the parameter $\mu\in \left(-\frac{1}{\|B_1\|+\|B_2\|},\frac{1}{\|B_1\|+\|B_2\|}\right)$.
\end{prop}

\begin{proof} Note first that for any $\mu\in\mathbb{C}$, all eigenvalues $\lambda$ of the operator $E+\mu B$ (acting on the complexification of $V\oplus V$) obey \begin{equation} \label{lambdamu}|\lambda|+|\mu|(\|B_1\|+\|B_2\|)\geq 1 \end{equation}
Indeed, an eigenvector $\left(\begin{array}{c}x\\y\end{array}\right)$ of $E+\mu B$ with eigenvalue $\lambda$ will have \begin{align*}
\mu B_1 x+\mu B_2 y &= x+\lambda x \\ \mu B_2 x+\mu B_1 y &=-y+\lambda y 
\end{align*}  If $\|x\|\geq \|y\|$ then the first equation above yields \[ |\mu|\|B_2\|\|x\|\geq |\mu|\|B_2\|\|y\|\geq (1-|\lambda|-|\mu|\|B_1\|)\|x\|,\] while if $\|x\|\leq \|y\|$ then the second equation above yields \[ |\mu|\|B_2\|\|y\|\geq |\mu|\|B_2\|\|x\|\geq (1-|\lambda|-|\mu|\|B_1\|)\|y\|.\]  Since either $\|x\|\geq \|y\|$ and $\|x\|$ is nonzero, or else $\|x\|\leq \|y\|$ and $\|y\|$ is nonzero, after dividing one or the other of the above inequalities by $\|x\|$ or $\|y\|$, as appropriate, we obtain (\ref{lambdamu}).

In particular it follows from (\ref{lambdamu}) that none of the operators $E+\mu B$ with $|\mu|<\frac{1}{\|B_1\|+\|B_2\|}$ has zero as an eigenvalue.  Of course, if we restrict $\mu$ to be real, then the $E+\mu B$ are all symmetric operators and therefore have entirely real spectrum. For $\mu=0$ the spectrum of $E+\mu B=E$ consists of the eigenvalues $-1$ and $1$, each with multiplicity $\dim V$.  As $\mu$ varies through the open interval $\left(-\frac{1}{\|B_1\|+\|B_2\|},\frac{1}{\|B_1\|+\|B_2\|}\right)$, since none of the eigenvalues of $E+\mu B$ cross zero it follows from continuity considerations that the total dimension of the negative eigenspaces of $E+\mu B$ will continue to be $\dim V$ for real $\mu$ with $|\mu|<\frac{1}{\|B_1\|+\|B_2\|}$, and likewise for the total dimension of the positive eigenspaces.

It remains to prove the assertion about the analyticity of the projections $\Pi_{B}^{\pm}(\mu)$ as functions of $\mu$. Denote the image of $\Pi^{\pm}(\mu)$ by $W_{\pm}(\mu)$ (so $W_-(\mu)$ is the span of the eigenvectors having negative eigenvalue, and $W_{+}(\mu)$ is the span of the eigenvectors having positive eigenvalue).  Since $E+\mu B$ is (for real $\mu$) symmetric, eigenvectors corresponding to distinct eigenvalues are orthogonal, and so $W_+(\mu)$ is orthogonal to $W_-(\mu)$.  Thus the orthogonal projections $\Pi_{B}^{\pm}(\mu)$ are just the projections associated to the direct sum decomposition $V\oplus V=W_+(\mu)\oplus W_-(\mu)$.  The desired conclusion now follows from a standard argument found, \emph{e.g.}, in \cite[II.1.4]{K}: given $\mu_0\in \left(-\frac{1}{\|B_1\|+\|B_2\|},\frac{1}{\|B_1\|+\|B_2\|}\right)$, choose contours $C_{\pm}$ in the complex plane disjoint from the eigenvalues of $E+\mu_0 B$ such that $C_+$ encloses precisely the positive eigenvalues of $E+\mu_0 B$ and $C_-$ encloses precisely the negative eigenvalues of $E+\mu_0 B$.  Then for $\mu$ sufficiently close to $\mu_0$ it will continue to hold that $C_+$ encloses precisely the positive eigenvalues of $E+\mu B$ and $C_-$ encloses precisely the negative eigenvalues of $E+\mu B$, and where $I$ denotes the identity the projections in question are given by the formulas \[ \Pi_{B}^{+}(\mu)=-\frac{1}{2\pi i}\int_{C_+}(E+\mu B-zI)^{-1}dz,\qquad \Pi_{B}^{-}(\mu)=-\frac{1}{2\pi i}\int_{C_-}(E+\mu B-zI)^{-1}dz.\] These expressions for $\Pi_{B}^{\pm}$ are manifestly analytic in $\mu$.
\end{proof}

\begin{prop}\label{finiteta}  Let $B_1,B_2\co \R\to Hom_{\R}(V,V)$ be two continuous maps such that there exist $T>0$ and symmetric linear operators $B_{1}^{\pm},B_{2}^{\pm}$ with $B_i(s)=B_{i}^{+}$ for $s\geq T$ and $B_i(s)=B_{i}^{-}$ for $s\leq -T$.  Let $\eta_0>\max\{\|B_{1}^{-}\|+\|B_{2}^{-}\|,\|B_{1}^{+}\|+\|B_{2}^{+}\|\}$, and define $B\co \R\to End(V\oplus V)$ by\[ B(s)\left(\begin{array}{c} x\\y\end{array}\right)=\left(\begin{array}{c} B_1(s)x+B_2(s)y\\B_2(s)x+B_1(s) y\end{array}\right) .\]  Then the set \[ \mathcal{S}=\left\{\eta\in [\eta_0,\infty)\left| 
\begin{array}{c} \mbox{There is a nonzero solution }v\in W^{1,2}(\R;V\oplus V) \mbox{ to } \\ \frac{dv}{ds}+(\eta E+B(s))v(s)=0         
\end{array}\right.\right\}\] is finite.
\end{prop}

\begin{proof}
First we show that $\mathcal{S}$ is bounded above.  For $i=1,2$ write $\|B_i\|=\sup_s\|B_i(s)\|$ (since the $B_i$ are continuous and asymptotically constant this supremum is of course finite) and let $\eta_1=\|B_1\|+\|B_2\|$.  Suppose that $v(s)=\left(\begin{array}{c}x(s)\\y(s)\end{array}\right)$ (where $x,y\co \R\to V$) is a nonzero class-$W^{1,2}$ solution to   $\frac{dv}{ds}+(\eta E+B(s))v(s)=0$.  One then has \begin{align*} \frac{1}{2}\frac{d}{ds}\|x(s)\|^2&=\eta\|x(s)\|^2-\langle x(s),B_1(s)x(s)\rangle-\langle x(s),B_2(s) y(s)\rangle \\ \frac{1}{2}\frac{d}{ds}\|y(s)\|^2&=-\eta\|y(s)\|^2-\langle y(s),B_2(s)x(s)\rangle - \langle y(s),B_1(s)x(s)\rangle \end{align*} which yields \begin{equation}\label{minusplus} \frac{d}{ds}\left(\|x(s)\|^2-\|y(s)\|^2\right)\geq 2(\eta-\eta_1)(\|x(s)\|^2+\|y(s)\|^2).\end{equation}  Of course, by the uniqueness of solutions to linear ODE's and the assumption that $v$ is nonzero, we have $\|x(s)\|^2+\|y(s)\|^2> 0$ for all $s$. Now assume for contradiction that the number $\eta$ associated to our solution obeys $\eta>\eta_1$.   Then (\ref{minusplus}) implies first that if at any $s_0\in \R$ we had $\|x(s)\|^2\geq \|y(s)\|^2$, then for $s_1$ slightly larger than $s_0$ we would have $\|x(s_1)\|^2-\|y(s_1)\|^2>0$.  But then another application of (\ref{minusplus}) implies that $f(s)=\|x(s)\|^2-\|y(s)\|^2$ obeys the differential inequality $f'(s)\geq 2(\eta-\eta_1)f(s)$, which since $f(s_1)>0$ and $\eta>\eta_1$ would force $f(s)$ to diverge to $\infty$ as $s\to\infty$, which is obviously incompatible with $v$ being of class $W^{1,2}$.  Thus we have the desired contradiction unless $\|y(s)\|^2>\|x(s)\|^2$ for all $s$.  But in this case $g(s)=\|y(s)\|^2-\|x(s)\|^2$ is an everywhere-positive function obeying the differential inequality $g'(s)\leq -2(\eta-\eta_1)g(s)$, which forces $g$ to diverge to $\infty$ as $s\to -\infty$, again contradicting the assumption that $v$ was of class $W^{1,2}$.  This contradiction shows that if $\eta>\eta_1$ then no solution of the relevant type can exist, proving that the set $\mathcal{S}$ in the statement of the theorem has $\mathcal{S}\subset [\eta_0,\eta_1]$.

Our strategy now will be to identify $\mathcal{S}$ with the intersection of the zero loci of a collection of real analytic functions defined on $[\eta_0,\infty)$.  These functions obviously will not all be identically zero since we have already established that $\mathcal{S}\subset [\eta_0,\eta_1]$, so since the zero set of a nonconstant analytic function on a connected subset of $\R$ is always discrete, the proposition will follow from such an identification.

Where $B^+\in End(V\oplus V)$ is the common value of $B(s)$ for all $s\geq T$, and $B^-\in End(V\oplus V)$ is the common value of $B(s)$ for all $s\leq -T$, the operators $B^{\pm}$ are of the type considered in Proposition \ref{eigen}, and our choice of the parameter $\eta_0$ ensures that, for all $\mu$ in an open interval containing $[0,\frac{1}{\eta_0}]$ the symmetric operators $E+\mu B^{\pm}$ have the property that the orthogonal projections $\Pi_{B_{\pm}}^{-}(\mu)$ onto the spans of their negative eigenspaces vary analytically in $\mu$ and have rank $\dim V$, while the projections $\Pi_{B_{\pm}}^{+}(\mu)$ onto the spans of their positive eigenspaces also vary analytically in $\mu$ and have rank $\dim V$.

Now for $\eta\geq \eta_0$ the differential equation \begin{equation} \label{etab}\frac{dv}{ds}+(\eta E+B(s))v(s)=0 \end{equation} reduces for $s\leq -T$ to the autonomous equation 
$\frac{dv}{ds}+(\eta E+B^-)v(s)=0$, and for $s\geq T$ to the autonomous equation $\frac{dv}{ds}+(\eta E+B^+)v(s)=0$.  Any $W^{1,2}$ solutions to (\ref{etab}) must have $v(-T)$ belonging to the span of those eigenvectors of $\eta E+B^-$ with negative eigenvalues, and must have $v(T)$ belonging to the span of those eigenvectors of $\eta E+B^+$ with positive eigenvalues.  In other words, a $W^{1,2}$ solution to (\ref{etab}) defined on all of $\R$ must have $v(-T)\in Im(\Pi_{B_{-}}^{-}(\eta^{-1}))$ and $v(T)\in Im(\Pi_{B_{+}}^{+}(\eta^{-1}))$.  

Denote by $\Phi_{\eta}\co \R\to GL(V\oplus V)$ the unique solution to the initial value problem \[ \frac{d\Phi}{ds}+(\eta E+B(s))\Phi(s)=0\qquad \Phi(-T)=I\] where $I$ is the identity.  Thus $\Phi_{\eta}$ is the fundamental solution to (\ref{etab}) in the sense that any solution $v\co \R\to V\oplus V$ to (\ref{etab}) will have $v(s)=\Phi(s)v(-T)$ for all $s\in \R$.  Thus any solution to (\ref{etab}) which is of class $W^{1,2}$ will have $v(-T)\in Im(\Pi_{B_{-}}^{-}(\eta^{-1}))$, $v(T)\in Im(\Pi_{B_{+}}^{+}(\eta^{-1}))$, and $v(T)=\Phi_{\eta}(v(-T))$.  As such, we will have $\eta\in \mathcal{S}$ if and only if the images of the linear maps $\Phi_{\eta}\circ \Pi_{B_{-}}^{-}(\eta^{-1})$ and 
$\Pi_{B_{+}}^{+}(\eta^{-1})$ have nontrivial intersection.  Said differently, since the image of  $\Pi_{B_{+}}^{+}(\eta^{-1})$ is the same as the kernel of  $\Pi_{B_{+}}^{-}(\eta^{-1})$, we have $\eta\in \mathcal{S}$ if and only if the linear map  $\Pi_{B_{+}}^{-}(\eta^{-1})\circ\Phi_{\eta}(T)\circ \Pi_{B_{-}}^{-}(\eta^{-1})$ has rank strictly less than $\dim V$.  By Proposition \ref{eigen}, the maps $\Pi_{B_{+}}^{-}(\eta^{-1})$ and $\Pi_{B_{-}}^{-}(\eta^{-1})$ both vary analytically with $\eta\in [\eta_0,\infty)$; let us now check that $\Phi_{\eta}(T)$ varies analytically with $\eta$.

Indeed, this follows readily from the standard Picard iteration formula for $\Phi_{\eta}$: we will have \begin{align*} \Phi_{\eta}(T)&=I+\sum_{n=1}^{\infty}\int_{\{-T\leq s_n\leq\cdots\leq s_1\leq T\}}(\eta E+B(s_1))(\eta E+B(s_2))\cdots(\eta E+B(s_n))ds_n\cdots ds_1\\&=\sum_{n=0}^{\infty}\left(\sum_{m=0}^{n}C_{m,n}(T)\eta^{n-m}\right) \end{align*} where $C_{0,0}(T)=I$ and for $0\leq m\leq n$ and $n\geq 1$ \[ C_{m,n}(T)=\sum_{1\leq i_1<\cdots<i_m\leq n}\int_{\{-T\leq s_n\leq\cdots\leq s_1\leq T\}}E^{i_1-1}B(s_{i_1})E^{i_2-i_1-1}B(s_{i_2})\cdots B(s_{i_m})E^{n-i_m}ds_n\cdots ds_1.\]  Now the fact that $\{-T\leq s_n\leq\cdots\leq s_1\leq T\}$ has volume $\frac{(2T)^n}{n!}$ (along with the fact that $\|E\|=1$) gives an estimate \begin{equation} \label{cest}\|C_{m,n}(T)\|\leq \binom{n}{m}\frac{(2T)^n\|B\|^m}{n!} \end{equation}  From this estimate one easily sees that for any $k\geq 0$ the series $\sum_{m=0}^{\infty}C_{m,m+k}(T)$ is absolutely convergent to an operator having norm bounded above by $e^{2T\|B\|}\frac{(2T)^k}{k!}$, and that our above series expression for $\Phi_{\eta}(T)$ can be rearranged to give \[ \Phi_{\eta}(T)=\sum_{k=0}^{\infty}\eta^k\left(\sum_{m=0}^{\infty}C_{m,m+k}(T)\right);\] moreover, another application of (\ref{cest}) shows that this power series in $\eta$ has infinite radius of convergence, confirming the analyticity of $\Phi_{\eta}(T)$ as a function of $\eta$.

Thus, if we fix a basis for $V\oplus V$ and represent the $\eta$-dependent linear map $\Pi_{B_{+}}^{-}(\eta^{-1})\circ\Phi_{\eta}(T)\circ \Pi_{B_{-}}^{-}(\eta^{-1})$ by a matrix with respect to the fixed basis, this matrix will vary analytically with $\eta\in [\eta_0,\infty)$, and our set $\mathcal{S}$ will consist of those $\eta$ such that all $(\dim V)\times (\dim V)$ minors of the matrix are zero.  This confirms that $\mathcal{S}$ is the common zero locus of a collection of analytic functions of $\eta\in [\eta_0,\infty)$, so since $\mathcal{S}$ is bounded it must be finite.
\end{proof}

We now apply these results to Floer theory. Let $(M,\omega)$ be a closed $2n$-dimensional symplectic manifold and let $G\co M\to \R$ be a Morse function, which we will assume to have the property that around each critical point $p\in Crit(H)$ there is a Darboux chart $\phi_p\co U_p\cong B^{2n}(\ep)$ such that the second-order Taylor approximation to $G\circ \phi_{p}^{-1}$ is exact. (In other words, the Hessian of $G\circ\phi_{p}^{-1}$ is constant on $B^{2n}(\ep)$.)  Shrinking the $U_p$ if necessary, we may assume that $\overline{U_p}\cap \overline{U_q}=\varnothing$ for each pair of distinct critical points $p$ and $q$.  Let $J$ be an $\omega$-compatible almost complex structure on $M$ having the properties that \begin{itemize} \item[(i)] On each of the Darboux balls $U_p$, $J$ coincides with the pullback by $\phi_p$ of the standard complex structure on $B^{2n}(\ep)$.
\item[(ii)] Where $g_J(\cdot,\cdot)=\omega(\cdot,J\cdot)$ is the Riemannian metric induced by $\omega$ and $J$, the gradient flow of $G$ with respect to $g_J$ is Morse--Smale.
\end{itemize}

Of course, all of the above conditions will continue to hold if $G$ is replaced by $\lambda G$ for any  $\lambda>0$.

The almost complex structure $J$ satisfying (i) and (ii) will be fixed throughout the following discussion, and we will use $\nabla$ to denote the covariant derivative determined by the Levi-Civita connection of the metric $g_J$.

Consider a solution $\gamma\co\R\to M$ to the negative gradient flow equation \begin{equation}\label{gradflow} \dot{\gamma}(s)+\nabla G(\gamma(s))=0 \end{equation} obeying the finite energy condition $\int_{-\infty}^{\infty}\|\dot{\gamma}(s)\|_{g_J}^{2}ds<\infty$.  For any such $\gamma$ there are critical points $p_{\pm}\in Crit(G)$ such that $\gamma(s)\to p_{\pm}$ exponentially quickly as $s\to\pm\infty$.  As is well-known, the Morse--Smale condition is equivalent to the statement that for any such $\gamma$ the linearization $\mathcal{G}_{\gamma}\co W^{1,2}(\gamma^*TM)\to L^2(\gamma^*TM)$ of (\ref{gradflow}) is surjective, where $\mathcal{G}_{\gamma}$ is given by the formula \[ \mathcal{G}_{\gamma}(\zeta)=\nabla_s\zeta+\nabla_{\zeta}\nabla G(\gamma(s)).\]

The solution $\gamma$ to (\ref{gradflow}) gives rise to a solution \begin{align*} u_{\gamma}\co \R\times S^1 &\to M \\ u_{\gamma}(s,t)&=\gamma(s)\end{align*} to the Floer equation \begin{equation}\label{floereq} \frac{\partial u}{\partial s}+J(u(s,t))\left(\frac{\partial u}{\partial t}-X_{G}(u(s,t))\right)=0,\end{equation} and indeed all finite-energy $t$-independent solutions to (\ref{floereq}) evidently have the form $u=u_{\gamma}$ for some solution $\gamma$ to (\ref{gradflow}).  

We consider the question of whether the linearization of (\ref{floereq}) at the solution $u_{\gamma}$ is surjective.  In effect we will show that this is in fact the case if the Hessian of $G$ near its critical points is not too large and if $G$ is replaced by $\lambda G$ for a suitable real parameter $\lambda$ which may be taken arbitrarily close to $1$.  More precisely, if $\gamma$ is a finite-energy solution to (\ref{gradflow}) then for any $\lambda>0$ the map $ \gamma^{\lambda}(s)=\gamma(\lambda s)$ will be a solution to the version of (\ref{gradflow}) obtained by replacing $G$ by $\lambda G$, and hence we will have a solution \[ u_{\gamma^{\lambda}}(s,t)=\gamma(\lambda s)\] to the Floer equation associated to the Hamiltonian $\lambda G$.  We prove:

\begin{theorem} \label{appthm} Where $G$ and $J$ are as above, fix a finite-energy solution $\gamma\co \R\to M$ to (\ref{gradflow}) having $\gamma(s)\to p_{\pm}\in Crit (G)$ as $s\to\pm\infty$.  Assume that the Hessians $\mathcal{H}_{\pm}$ of $G$ at $p_{\pm}$ have operator norms $\|\mathcal{H}_{\pm}\|<\pi$.  Then for all but finitely many $\lambda\in (0,1]$ it holds that the linearization \[ \mathcal{F}_{u_{\gamma^{\lambda}}}\co W^{1,2}(\R\times S^1;u_{\gamma^{\lambda}}^*TM)\to L^2(\R\times S^1;u_{\gamma^{\lambda}}^*TM) \] of the Floer operator $u\mapsto \frac{\partial u}{\partial s}+J(u(s,t))\left(\frac{\partial u}{\partial t}-X_{\lambda G}(u(s,t))\right)$ at $u_{\gamma^{\lambda}}$ is surjective.
\end{theorem}

\begin{proof}  First note that, by virtue of the fact that $\|\mathcal{H}_{\pm}\|<2\pi$, the Fredholm index of the linearization $\mathcal{F}_{u_{\gamma^{\lambda}}}$ is equal to $ind(p_-)-ind(p_+)$ where $ind$ denotes the Morse index (see, \emph{e.g.}, \cite[Theorem 4.1 and Lemma 7.2]{SZ}); in turn this latter quantity is equal to the index of the linearization $\mathcal{G}_{\gamma^{\lambda}}\co W^{1,2}((\gamma^{\lambda})^{*}TM)\to L^2((\gamma^{\lambda})^*TM)$ of the negative gradient flow operator, which is surjective by the Morse--Smale condition.  So it suffices to show that, for all but finitely many $\lambda$, we have $\dim\ker \mathcal{F}_{u_{\gamma^{\lambda}}}\leq \dim\ker \mathcal{G}_{\gamma^{\lambda}}$.  Now any element $\zeta\in \ker \mathcal{G}_{{\gamma}^{\lambda}}$ gives rise to an element $\xi_{\zeta}\in\ker\mathcal{F}_{u_{\gamma^{\lambda}}}$ by the prescription $\xi_{\zeta}(s,t)=\zeta(s)\in T_{\gamma^{\lambda}(s)}M=T_{u_{\gamma^{\lambda}}(s,t)}M$, and conversely any $t$-independent element $\xi\in \ker\mathcal{F}_{u_{\gamma^{\lambda}}}$ is of this form. So for any $\lambda\in (0,1]$, to show that  $\dim\ker \mathcal{F}_{u_{\gamma^{\lambda}}}\leq \dim\ker \mathcal{G}_{\gamma^{\lambda}}$ we just need to show that these $\xi_{\zeta}$ are the only elements of $\ker\mathcal{F}_{u_{\gamma^{\lambda}}}$, \emph{i.e.}, that all elements of $\ker\mathcal{F}_{u_{\gamma^{\lambda}}}$ are $t$-independent.

  The linearization $\mathcal{F}_{u_{\gamma^{\lambda}}}$ is given by, for $\xi\in W^{1,2}(\R\times S^1;u_{\gamma^{\lambda}}^*TM)$, \[ \mathcal{F}_{u_{\gamma^{\lambda}}}\xi=\nabla_s\xi +J(\gamma(\lambda s))\frac{\partial \xi}{\partial t}+\lambda (\nabla_{\xi}\nabla G)(\gamma(\lambda s)).\]  For $\lambda>0$ and for a section $\sigma$ of $u_{\gamma^{\lambda}}^{*}TM$ define a section $\hat{\sigma}$ of $u_{\gamma}^{*}TM$ by \[ \hat{\sigma}(s,t)=\sigma(s/\lambda,t).\]  Then where we define $\widehat{\mathcal{F}}_{u_{\gamma}}^{\lambda}\co W^{1,2}(\R\times S^1;u_{\gamma}^{*}TM)\to L^2(\R\times S^1;u_{\gamma}^{*}TM)$ by \[ \widehat{\mathcal{F}}_{u_{\gamma}}^{\lambda}\xi=\nabla_s\xi+\lambda^{-1} J(\gamma(s))\frac{\partial \xi}{\partial t}+(\nabla_{\xi}\nabla G)(\gamma(s)),\] we have \[ \widehat{\mathcal{F}_{u_{\gamma^{\lambda}}}\xi}=\lambda \widehat{\mathcal{F}}_{u_{\gamma}}^{\lambda}\widehat{\xi}.\]  Since obviously $\hat{\sigma}$ is $t$-independent if and only if $\sigma$ is $t$-independent, it now suffices to show that, for all but finitely many $\lambda\in (0,1]$, every element $\xi\in W^{1,2}(\R\times[0,1];u_{\gamma}^{*}TM)$ of the kernel of the operator $\widehat{\mathcal{F}}_{u_{\gamma}}^{\lambda}$ is $t$-independent.

  To achieve this we consider the Fourier decomposition (in the $t$-variable) of a hypothetical element $\xi\in \ker \widehat{\mathcal{F}}_{u_{\gamma}}^{\lambda}$.  For $k\in \Z$ and for a section $\xi$ of $u_{\gamma}^{*}TM$ define a new section $e^{-2\pi ktJ}\xi$ by $(e^{-2\pi ktJ}\xi)(s,t)=(\cos(2\pi kt))\xi(s,t)-\sin(2\pi kt)J(\gamma(s))\xi(s,t)$.  Moreover define a section $\xi_k$ of $\gamma^*TM$ by \[ \xi_k(s)=\int_{0}^{1}e^{-2\pi ktJ}\xi(s,t)dt \] (of course this is well-defined since $\xi(s,t)\in T_{\gamma(s)}M$ for all $t$).
  
  We observe that \begin{align*} (\nabla_s\xi)_k(s)&=\int_{0}^{1}e^{-2\pi ktJ}\nabla_s\xi(s,t)\\&=\nabla_s\left(\int_{0}^{1}(\cos(2\pi kt)I-\sin(2\pi kt)J)\xi(s,t)dt\right)+(\nabla_s J)(\gamma(s))\left(\int_{0}^{1}\sin(2\pi kt)\xi(s,t)dt\right)
  \\&= \nabla_s(\xi_k)(s)+\frac{1}{2}(\nabla_s J)J(\gamma(s))\int_{0}^{1}\left(e^{-2\pi ktJ}-e^{2\pi ktJ}\right)\xi(s,t)dt
  \\&=\left(\nabla_s\xi_k +\frac{1}{2}(\nabla_s J)J\xi_k-\frac{1}{2}(\nabla_s J)J\xi_{-k}\right)(s).\end{align*}

Also, \begin{align*}
\left(J\frac{\partial \xi}{\partial t}\right)_k(s)&=J(\gamma(s))\int_{0}^{1}e^{-2\pi ktJ}\frac{\partial \xi}{\partial t}dt\\&=J(\gamma(s))\left(\int_{0}^{1}\frac{\partial}{\partial t}\left(e^{-2\pi kt J}\xi(s,t)\right)dt+2\pi kJ(\gamma(s))\int_{0}^{1}e^{-2\pi ktJ}\xi(s,t)dt\right)\\&
=-2\pi k\xi_k(s) \end{align*} by periodicity and the fact that $J(\gamma(s))^2=-I$.  

Moreover, if we resolve the Hessian operator $\mathcal{H}(s)\co T_{\gamma(s)}M\to T_{\gamma(s)}M$ (defined by $\mathcal{H}(s)v=\nabla_v\nabla G(\gamma(s)))$ into its complex-linear and complex-antilinear parts as  \[ \mathcal{H}^{1,0}(s)=\frac{1}{2}\left(\mathcal{H}(s)-J(\gamma(s))\mathcal{H}(s)J(\gamma(s))\right) \quad \mbox{and}\quad \mathcal{H}^{0,1}(s)=\frac{1}{2}\left(\mathcal{H}(s)+J(\gamma(s))\mathcal{H}(s)J(\gamma(s))\right) \] we see that \[ e^{-2\pi ktJ} \mathcal{H}\xi=\mathcal{H}^{1,0}e^{-2\pi ktJ}\xi+\mathcal{H}^{0,1}e^{2\pi ktJ}\xi,\] and so \[ (\mathcal{H}\xi)_k(s)=(\mathcal{H}^{1,0}\xi_k)(s)+(\mathcal{H}^{0,1}\xi_{-k})(s).\]  

These calculations show that, for $\xi\in W^{1,2}(\R\times S^1;u_{\gamma}^{*}TM)$ and $k\in \Z$, we have \[ \left(\widehat{\mathcal{F}}_{u_{\gamma}}^{\lambda}\xi\right)_k=\nabla_s\xi_k +\left(-2\pi k\lambda^{-1}+\mathcal{H}^{1,0}+\frac{1}{2}(\nabla_s J)J\right)\xi_k+\left(\mathcal{H}^{0,1}-\frac{1}{2}(\nabla_s J)J\right)\xi_{-k}.\]  Thus an element $\xi\in \ker\widehat{\mathcal{F}}_{u_{\gamma}}^{\lambda}$ has, for each $k\in \Z_{>0}$, \begin{equation}\label{fourier} \nabla_s\left(\begin{array}{c}\xi_k\\ \xi_{-k}\end{array}\right)+\left(\begin{array}{cc} -2\pi k\lambda^{-1}I+\mathcal{H}^{1,0}+\frac{1}{2}(\nabla_s J)J & \mathcal{H}^{0,1}-\frac{1}{2}(\nabla_s J)J \\ \mathcal{H}^{0,1}-\frac{1}{2}(\nabla_s J)J & 2\pi k\lambda^{-1}I+\mathcal{H}^{1,0}+\frac{1}{2}(\nabla_s J)J \end{array}\right)\left(\begin{array}{c}\xi_k\\ \xi_{-k}\end{array}\right)=0 \end{equation}

Now let us choose a unitary trivialization of $\gamma^*TM$ which, over those $s\in \R$ with $|s|$ large enough such that $\gamma(s)$ lies in one of the Darboux charts around the critical points $p_{\pm}$ in which $J$ was assumed to be standard and the Hessian $\mathcal{H}_{\pm}$ of $H$ was assumed constant, coincides with the trivialization of $\gamma^*TM$ induced by these Darboux charts.  Rewriting (\ref{fourier}) in terms of this trivialization gives equations, for $v_k\co \R\to \R^{2n}\times \R^{2n}$ and $k\in \Z_{>0}$, \begin{equation}\label{fourtriv} \frac{dv_k}{ds}+(2\pi k\lambda^{-1} E+B(s))v_k(s)=0\end{equation} where the smooth map $B\co \R\to Hom_{\R}(\R^{2n}\times \R^{2n},\R^{2n}\times \R^{2n})$ is independent of $k$ and $\lambda$ and coincides with $\left(\begin{array}{cc} \mathcal{H}_{\pm}^{1,0} & \mathcal{H}_{\pm}^{0,1} \\ \mathcal{H}_{\pm}^{0,1} & \mathcal{H}_{\pm}^{1,0} \end{array}\right)$ when $\pm s$ is large enough such that $\gamma(s)$ is in the Darboux chart around $p_{\pm}$.  Now $\mathcal{H}_{\pm}^{1,0}$ and $\mathcal{H}_{\pm}^{0,1}$ are symmetric since $\mathcal{H}_{\pm}$ is, and we have $\|\mathcal{H}_{\pm}^{1,0}\|+\|\mathcal{H}_{\pm}^{0,1}\|\leq 2\|\mathcal{H}_{\pm}\|<2\pi$.  So by Proposition \ref{finiteta} the set \[ S=\left\{\mu\in [1,\infty)\left|\begin{array}{c}\mbox{There is a nonzero, class-$W^{1,2}$ solution to}\\ \frac{dv}{ds}+(2\pi\mu E+B(s))v(s)=0\end{array}\right.\right\} \] is finite. If $k_0$ is any integer larger than the largest element of $S$ it in particular follows that for $k\geq k_0$ there is no $\lambda\in (0,1]$ such that (\ref{fourtriv}) has a nontrivial $W^{1,2}$ solution.  Moreover, since $S$ is finite, for any $k\in\{1,\ldots,k_0-1\}$ there are only finitely many $\lambda\in (0,1]$ such that (\ref{fourtriv}) has a nontrivial $W^{1,2}$ solution.  Combining these two facts shows that there are only finitely many $\lambda\in (0,1]$ such that there exists any $k$ so that (\ref{fourtriv}) has a $W^{1,2}$ solution.

Consequently we obtain that, if $\lambda\in (0,1]$ is not among these finitely many exceptional values, then any element $\xi\in \ker\widehat{\mathcal{F}}_{u_{\gamma}}^{\lambda}$ has $\xi_k=0$ for all $k\in \Z\setminus\{0\}$.  So the section $\underline{\xi}$ defined by $\underline{\xi}(s,t)=\xi(s,t)-\xi_0(s)$ has $\underline{\xi}_k=0$ for all $k\in\Z$.  Thus $\underline{\xi}$ is $L^2$-orthogonal to any section of the form $(s,t)\mapsto e^{2\pi ktJ}\zeta(s)$ for $\zeta\in L^2(\gamma^*TM)$.  Since linear combinations of sections of this latter form are dense in $L^2$ it follows that $\underline{\xi}=0$, and hence that $\xi(s,t)=\xi_0(s)$ for all $s$.  This proves that, for all but finitely many values of $\lambda$, all elements of $\ker\widehat{\mathcal{F}}_{u_{\gamma}}^{\lambda}$ are $t$-independent, which as explained earlier suffices to prove the theorem.
\end{proof}

\end{document}